\pdfoutput=1
\documentclass[letterpaper,11pt]{article}
\usepackage[margin=1in]{geometry}
\usepackage[
            CJKbookmarks=true,
            bookmarksnumbered=true,
            bookmarksopen=true,
            colorlinks=true,
            citecolor=red,
            linkcolor=blue,
            anchorcolor=red,
            urlcolor=blue
            ]{hyperref}

\usepackage[title]{appendix}
\usepackage{array}        
\usepackage{natbib}
\usepackage{graphicx}
\usepackage{subcaption}
\usepackage{booktabs} 

\usepackage{algorithm}
\usepackage{algorithmic}

\usepackage{amsmath,amsthm,amsfonts,amssymb}
\usepackage{hyperref}
\usepackage{color}
\usepackage{enumitem}
\usepackage{bm}
\usepackage{multirow}
\usepackage{bbm}
\usepackage{subfiles}
\usepackage{xspace}
\usepackage{lmodern}
\usepackage{caption}
\usepackage{mathtools} 
\usepackage{tikz}

\newtheorem{theorem}{Theorem}
\newtheorem{lemma}{Lemma}

\newtheorem{assumption}{Assumption}
\newtheorem{proposition}{Proposition}

\theoremstyle{definition}
\newtheorem{remark}{Remark}
\newtheorem{definition}{Definition}

\usetikzlibrary{intersections,calc,decorations.pathreplacing,through,arrows}
\usetikzlibrary{ decorations.markings,positioning}

\tikzset{global scale/.style={
    scale=#1,
    every node/.append style={scale=#1}
  }
}

\newcommand{\maC}{\mathcal{C}}

\newcommand{\maE}{\mathcal{E}}

\newcommand{\maG}{\mathcal{G}}

\newcommand{\maM}{\mathcal{M}}
\newcommand{\maN}{\mathcal{N}}

\newcommand{\maP}{\mathcal{P}}
\newcommand{\maQ}{\mathcal{Q}}

\newcommand{\maS}{\mathcal{S}}
\newcommand{\maT}{\mathcal{T}}

\newcommand{\prob}[1]{ \mathbb{P}\left[ #1 \right] }
\newcommand{\expect}[1]{ \mathbb{E}\left[ #1 \right] }

\newcommand{\pth}[1]{\left( #1 \right)}
\newcommand{\qth}[1]{\left[ #1 \right]}
\newcommand{\sth}[1]{\left\{ #1 \right\}}

\newcommand{\calN}{{\mathcal{N}}}

\newcommand{\calS}{{\mathcal{S}}}
\newcommand{\calT}{{\mathcal{T}}}

\newcommand{\ti}{\tilde}

\newcommand{\overlap}{\mathrm{overlap}}

\newcommand{\sfE}{{\mathsf{E}}}

\newcommand{\sfV}{{\mathsf{V}}}

\newcommand{\sfd}{{\mathsf{d}}}

\newcommand{\tr}{\mathrm{Tr}}

\newcommand{\indc}[1]{{\mathbf{1}_{\left\{{#1}\right\}}}}

\newcommand{\ER}{Erd\H{o}s-R\'enyi }
\DeclareMathOperator*{\argmax}{arg\,max}
\DeclareMathOperator*{\argmin}{arg\,min}

\def\var{\mathrm{Var}}
\def\E{\mathbb{E}}

\newcommand{\score}{S}

\title{
Attributed Network Alignment: Statistical Limits and Efficient Algorithm}
\author{Dong Huang, Chenyang Tian and Pengkun Yang\thanks{D.\ Huang and P.\ Yang are with the Department of Statistics and Data Science, Tsinghua University.
C. Tian is with Weiyang College, Tsinghua University.
P. Yang is supported in part by National Key R\&D Program of China 2024YFA1015800, 
Tsinghua University Dushi Program 2025Z11DSZ001, and High Performance Computing Center, Tsinghua University.}}
\begin{document}
\date{}
\maketitle
\begin{abstract}
This paper studies the problem of recovering a hidden vertex correspondence between two correlated graphs when both edge weights and node features are observed. While most existing work on graph alignment relies primarily on edge information, many real-world applications provide informative node features in addition to graph topology. To capture this setting, we introduce the featured correlated Gaussian Wigner model, where two graphs are coupled through an unknown vertex permutation, and the node features are correlated under the same permutation. We characterize the optimal information-theoretic thresholds for exact recovery and partial recovery of the latent mapping. On the algorithmic side, we propose QPAlign, an algorithm based on a quadratic programming relaxation, and demonstrate its strong empirical performance on both synthetic and real datasets. Moreover, we also derive theoretical guarantees for the proposed procedure, supporting its reliability and providing convergence guarantees.
\end{abstract}

\begin{keywords}%
  {Graph alignments, information-theoretic threshold, algorithm, attributed network}
\end{keywords}

\tableofcontents

\section{Introduction}

Graph alignment is a fundamental problem in network science and machine learning, with applications in many areas. For example, in computer vision, 3-D shapes can be represented as graphs and a significant problem for pattern recognition and image processing is determining whether two graphs represent the same object under rotations~\cite{berg2005shape,mateus2008articulated}; in natural language processing, each sentence can be represented as a graph and the ontology problem refers to uncovering the correlation between different knowledge graphs that are in different languages~\cite{haghighi2005robust}; in computational biology, proteins can be regarded as vertices and the interactions between them can be formulated as weighted edges~\cite{singh2008global, vogelstein2015fast}.

Since real-world scenarios often present challenges due to the noise in real data, many studies focused on random graph models to serve as a pivotal step, including graph alignment problem in \ER model~\cite{wu2022settling,ding2023matching,huang2024information}, Gaussian Wigner model~\cite{fan2019spectral,araya2024seeded,ding2024polynomial}, stochastic block model~\cite{onaran2016optimal,lyzinski2018information,chai2024efficient}, and graphon model~\cite{zhang2018consistent}.
However, previous works on random graph alignment mainly focus on models that rely solely on topological information. In real scenarios, however, feature information often plays a crucial role. For instance, in the ACM–DBLP dataset, node features such as paper titles or author names are essential for identifying corresponding entities across the two graphs~\cite{tang2023robust,bommakanti2024fugal}.
This motivates the study of alignment models that incorporate both structural and feature information, beyond purely topology-based settings. 

{
While existing work on attributed graph alignment mostly builds on correlated \ER or community-based model with binary edges and node attributes~\cite{zhang2024attributed,yang2024exact,yang2025exact}, many real-world networks are both weighted and attributed. For example, in gene co-expression networks~\cite{zhang2005general}, edge weights encode co-expression strength while genes carry functional annotations, and in social networks~\cite{leskovec2010predicting}, edges record rating or interaction strengths while nodes have profile or content features.
}
To bridge this gap, we investigate the following featured correlated Gaussian Wigner model, in which the random graphs are generated from Gaussian distributions.
\begin{definition}[Featured correlated Gaussian Wigner model]
{ Let $G_1$ and $G_2$ be two weighted random graphs with vertex sets $V(G_1),V(G_2)$ such that $|V(G_1)| = |V(G_2)| = n$.}
    Let $\pi^*$ denote a latent bijective mapping from $V(G_1)$ to $V(G_2)$. We say that a pair of graphs $(G_1,G_2)$ follows featured correlated Gaussian Wigner model $\maG(n,d,\rho,r)$ if \begin{enumerate}
        \item each pair of weighted edges $\beta_{uv}(G_1)$ and $\beta_{\pi^*(u)\pi^*(v)}(G_2)$ for any $u,v\in V(G_1)$ are correlated standard normals with correlation $\rho\in (0,1)$;
        \item  each pair of features $(\bm{x}_u,\bm{y}_{\pi^*(u)})$ for any $u\in V(G_1)$ follows multivariate normal distribution $\maN(0,\Sigma_d)$ with $\Sigma_d=\begin{bmatrix}
I_d & rI_d \\
rI_d & I_d
\end{bmatrix}$, where the dimension $d\in \mathbb{N}$ and the correlation $r\in (0,1)$. Moreover, we assume that the features are independent with the weighted edges.
    \end{enumerate}
\end{definition}
We assume features are standardized so that each coordinate has unit variance and is independent, which justifies the identity matrices in $\Sigma_d$. 
Edge weights are likewise centered and variance-normalized, with correlations $\rho$ and $r$ capturing structural and feature dependence, respectively.
{Furthermore, we assume $\rho,r\in(0,1)$ without loss of generality, since the model is invariant under flipping the signs of all edge weights or node features in $G_2$, and hence negative correlations can be reduced to positive ones.}
Indeed, this model bridges two significant extremes, it reduces to correlated Gaussian Wigner model~\cite{ding2021efficient} when $r=0$ and correlated Gaussian database model~\cite{dai2019database} when $\rho = 0$. Given $G_1$ and $G_2$ under $\maG(n,d,\rho,r)$, our goal is to recover the latent vertex mapping $\pi^*$. 
Specifically, given two permutations $\pi^*,\hat{\pi}:V(G_1)\mapsto V(G_2)$, denote the fraction of their overlap by $\overlap(\pi^*,\hat{\pi}) = \tfrac{1}{n}\vert\sth{ v\in V(G_1):\pi^*(v) = \hat{\pi}(v)}\vert$.
To quantify the performance of an estimator $\hat{\pi}$, we say $\hat{\pi}(G_1,G_2)$ achieves \begin{itemize}
    \item \emph{partial recovery}, if $\overlap(\hat{\pi},\pi^*)\ge \delta$ for  $\delta\in (0,1)$;
    \item \emph{exact recovery}, if $\overlap(\hat{\pi},\pi^*)=1$.
\end{itemize}

\subsection{Main Results}\label{subsec:main-results}



In this subsection, we present our main results on information-theoretic thresholds. Let $\maS_n$ denote the set of bijective mappings $\pi:V(G_1)\mapsto V(G_2)$. Our goal is to determine the correlation required for successful recovery of $\pi^*$. Next, we introduce our main theorems.

\begin{theorem}[Partial Recovery]\label{thm:main-partial}
Under featured correlated Gaussian Wigner model, if $d=\omega(\log n)$ and  $n\log (\tfrac1{1-\rho^2})+2d\log (\tfrac1{1-r^2})\ge (4+\epsilon)\log n$ for some constant $\epsilon>0$, then there exists an estimator $\hat\pi$ such that, for any fixed constant $0<\delta<1$ and $\pi^*\in \maS_n$, we have 
    $$\prob{\mathrm{overlap}(\hat \pi,\pi^*)\ge\delta}=1-o(1).$$
    Conversely, for any constant $0<\delta<1$, if $n\log (\frac1{1-\rho^2})+2d\log (\frac1{1-r^2})\le c\log n$ for some constant $c$, then for any estimator $\hat\pi$,
    $$\prob{\mathrm{overlap}(\hat \pi, \pi^*)<\delta}\ge 1-\frac{c}{4\delta},$$
    where $\pi^*$ is uniformly distributed over $\maS_n$.
\end{theorem}

{ The upper bound holds uniformly over all $\pi^*\in\maS_n$, while the lower bound is obtained by analyzing the Bayes risk under the uniform prior on $\pi^*$, which is least favorable in our permutation-symmetric model and therefore leads to the same threshold for the minimax risk. Consequently, the threshold is valid for both minimax and Bayesian risks. 
As for the assumption $d=\omega(\log n)$, it is standard in Gaussian database alignment: identifying a vertex among $n$ candidates requires $\Theta(\log n)$ bits, while each feature coordinate contributes only $O(1)$ bits of discriminative information, so the total feature dimension must eventually dominate $\log n$ (see, e.g., \cite{dai2019database}).}
Indeed, this assumption is commonly adopted in prior work on attributed graph alignment; see, e.g., \cite{dai2023gaussian,yang2025exact}.
Even without this assumption, we can still derive the optimal rate up to universal constants.
Theorem~\ref{thm:main-partial} characterizes the optimal rate for the information-theoretic threshold in the partial recovery regime. In particular, the special case $\delta = 1$ corresponds to exact recovery. For obtaining a sharper constant in this regime, we have the following theorem.

\begin{theorem}[Exact Recovery]\label{thm:exact-main}
Under featured correlated Gaussian Wigner model, if $d=\omega(\log n)$ and $n\log (\frac1{1-\rho^2})+d\log (\frac1{1-r^2})\ge (4+\epsilon)\log n$ for some constant $\epsilon>0$, then there exists an estimator $\hat \pi$ such that, for any $\pi^*\in \maS_n$, we have 
    $$\prob{\mathrm{overlap}(\hat \pi, \pi^*)=1}=1-o(1).$$

    Conversely, if $r^2\ge\frac{40}{d}$ and $n\log (\frac1{1-\rho^2})+d\log (\frac1{1-r^2})+4\log d\le (4-\epsilon)\log n$ for some constant $\epsilon>0$, then for any estimator $\hat \pi$, 
    $$\prob{\mathrm{overlap}(\hat \pi, \pi^*)\neq 1}=1-o(1),$$
    where $\pi^*$ is uniformly distributed over $\maS_n$.
\end{theorem}

The technical condition $r^2 \ge 40/d$ is only imposed to sharpen the leading constant in this regime: Theorem~\ref{thm:main-partial} already yields the optimal rate without this assumption under the special case $\delta = 1$, while under $d = n^{o(1)}$ and $d = \omega(\log n)$ such a lower bound on the feature signal ensures that our exact recovery threshold attains the optimal constant. 
In comparison with the special case $\delta = 1$ in Theorem~\ref{thm:main-partial}, Theorem~\ref{thm:exact-main} establishes a sharper information-theoretic threshold for exact recovery under certain conditions on $d$. 
Indeed, the difference between the $2d$-term in partial recovery and the $d$-term in exact recovery arises because exact recovery requires distinguishing much smaller distances between candidate alignments, which in turn necessitates stronger correlation and thus a stronger condition.

For the special case $r = 0$, \cite{wu2022settling} showed that in the correlated Gaussian Wigner model, there is a phase transition from possible exact recovery to impossible recovery as the quantity $\frac{n\log(1/(1-\rho^2))}{\log n}$ changes from $4+\epsilon$ to $4-\epsilon$. For another special case $\rho = 0$, \cite{dai2019database} demonstrated that in the Gaussian database model exact recovery is possible when $d\log(1/(1-r^2)) \ge (4+\epsilon)\log n$, and impossible when $d\log(1/(1-r^2)) \le (4-\epsilon)\log n$, under the same condition on the feature dimension $d$.
In our new model $\maG(n,d,\rho,r)$, we show that the optimal threshold is determined by the two components of the model: structure and features. Specifically, the term $n\log(1/(1-\rho^2))$ captures the contribution of structural information, while $d\log(1/(1-r^2))$ captures the contribution of feature information. Indeed, when $n\log(1/(1-\rho^2)) = C_1\log n$ and $d\log(1/(1-r^2))=C_2\log n$ with $C_1,C_2<4$, and $C_1+C_2>4$, there exists an estimator $\hat\pi$ that achieves exact recovery only when both edge and feature information are available. This demonstrates that our approach goes beyond the performance attainable by relying on either structural or feature information alone. Indeed, balancing the edge and feature information requires a careful choice of weighting coefficients in our estimator design, instead of simply adding the two parts together. See Section~\ref{sec:possibility} for details.


\subsection{Related Work}

\paragraph{Attributed graph alignment}
In the information-theoretic perspective,
\cite{zhang2024attributed} proposed the attributed \ER pair model, where the edges and the features follows Bernoulli distribution under a latent bijective mapping, and it derived the information-theoretic thresholds for recovering the latent mappings in both possibility and impossibility regimes. ~\cite{yang2024exact} proposed the correlated Gaussian-attributed \ER model and derived the optimal information-theoretic thresholds for exact recovery by analyzing the $k$-core estimator. Both work proposed random graph model with additional feature and found that the graph matching becomes feasible in a wider regime through the information of attributed nodes. There are also many algorithms for attributed graph alignment, including methods based on subgraph counting~\cite{du2017first,liu2019g,wang2025efficient}, spectral methods~\cite{zhang2016final}, optimal transport~\cite{tang2023robust}, and neighborhood statistics~\cite{wang2024feasible}.



\paragraph{Other graph models} Many information-theoretic properties of the correlated Gaussian Wigner model and correlated \ER model have been extensively investigated~\cite{cullina2016improved,cullina2017exact,ganassali2021impossibility,wu2022settling,wu2023testing,ding2023matching,ding2023detection,hall2023partial,huang2024information,du2025optimal,huang2025sample}, along with a rich line of algorithmic developments~\cite{babai1980random,bollobas1982distinguishing,barak2019nearly,dai2019analysis,ganassali2020tree,ding2021efficient,mao2021random,piccioli2022aligning,mao2023exact,mao2023random,fan2019spectral,fan2023spectral,ding2023polynomial,ding2024polynomial,araya2024seeded,ganassali2024statistical,muratori2024faster,du2025algorithmic}. 
However, the marginal distributions inherent in these models makes it different from graph models in practical applications. Therefore, it is crucial to explore more general graph models, such as graphon model~\cite{wolfe2013nonparametric,gao2015rate}, inhomogeneous graph model~\cite{racz2023matching,song2023independence,ding2025efficiently}, geometric random graph model~\cite{wang2022random,bangachev2024detection,gong2024umeyama,sentenac2025online}, planted cycle model~\cite{mao2023detection,mao2024informationtheoretic}, and multiple graph model~\cite{ameen2024exact,ameen2025detecting}.


 \subsection{Our Contribution}
We study the graph alignment problem under a featured correlated Gaussian Wigner model, where both weighted edges and node features are correlated through an unknown vertex permutation. Our contributions advance both the theoretical understanding and the algorithmic practice of attributed graph alignment.

We derive optimal information-theoretic thresholds for both partial and exact recovery. In contrast to most existing theoretical work on attributed graph alignment, which focus on unweighted \ER or stochastic block models~\cite{yang2013community,yang2024exact}, our results apply to graphs with continuous edge weights, revealing the gap between partial recovery and exact recovery in the featured correlated Gaussian Wigner model. Moreover, unlike prior works on Gaussian-attributed models~\cite{yang2024exact,yang2025exact} that require two-step algorithms to achieve optimality, our characterization is directly derived from maximum likelihood objective and does not rely on regime-specific estimators. This makes the theoretical conditions directly interpretable and more amenable to algorithmic realization.

Algorithmically, we propose QPAlign (see Section~\ref{sec:alg}) to achieve efficient recovery for attributed graphs. While existing methods for attributed graph alignment~\cite{bommakanti2024fugal,zeng2023parrot} typically lack theoretical guarantees, QPAlign has provable convergence guarantees and admits the oracle permutation as a feasible optimum of the relaxed objective. Extensive experiments on synthetic data and real datasets indicate that QPAlign performs effectively in regimes predicted by our theory and aligns well with the information-theoretic recovery limits.




\section{Information-theoretic Thresholds}

\subsection{Possibility Results}\label{sec:possibility}
We first introduce our estimator. Given two graphs $(G_1,G_2)\sim \maG(n,d,\rho,r)$ under the featured correlated Gaussian Wigner model, our goal is design an estimator $\hat{\pi}(G_1,G_2)$ to recover the latent bijective mapping $\pi^*:V(G_1)\mapsto V(G_2)$. Let $\maP$ denote the joint distribution of $(G_1,G_2)$, $P(\cdot,\cdot)$ denote the distribution of two correlated standard normals with correlation $\rho$, and $f(\cdot,\cdot)$ denote the multivariate normal distribution $\maN(0,\Sigma_d)$ with $\Sigma_d =\begin{bmatrix}
I_d & rI_d \\
rI_d & I_d
\end{bmatrix}$. 
Let $\varphi(x) = \tfrac{x}{1-x^2}$ and 
$\score_\pi(G_1,G_2) = \varphi(\rho)\sum_{e\in E(G_1)} \beta_e(G_1) \beta_{\pi(e)}(G_2)+\varphi(r)\sum_{v\in V(G_1)} \bm{x}_v \bm{y}_{\pi(v)}$.
Then the likelihood function $\maP_{G_1,G_2\mid \pi^*}\propto \exp\pth{S_{\pi^*}(G_1,G_2)}.$
Consequently, our estimator takes  the form 
\begin{align}\label{eq:estimator}
    \hat{\pi}\in \argmax_{\pi\in \calS_n}\, S_{\pi}(G_1,G_2).
\end{align}

Indeed, $S_\pi(G_1,G_2)$ represents the similarity score under $\pi$ between $G_1$ and $G_2$.
Then the estimator is equivalent to $\hat{\pi}\in \argmax_\pi \score_\pi(G_1,G_2)$. For any two bijections $\pi,\pi':V(G_1)\mapsto V(G_2)$, let $\sfd(\pi,\pi') = n(1-\overlap(\pi,\pi'))$. To prove the recovery guarantee of $\hat{\pi}$, it suffices to show that \begin{align*}
    \score_{\pi^*}(G_1,G_2)&>\max_{\pi:\sfd(\pi,\pi^*)\ge d_0}\score_\pi(G_1,G_2)=\max_{k\ge d_0} \max_{\pi:\sfd(\pi,\pi^*)=k} \score_\pi(G_1,G_2)
\end{align*}
with high probability, where the thresholds $d_0=1$ and $d_0 = (1-\delta)m$ correspond to the exact and partial recoveries, respectively. Let $\maT_k$ denote the set of bijective mappings such that $\sfd(\pi,\pi^*) = k$. Then the failure event satisfies \begin{align*}
    &~\sth{\sfd(\hat{\pi},\pi^*) = k}\subseteq \sth{\exists \pi'\in \maT_k:\score_{\pi^*}(G_1,G_2)\le \score_{\pi'}(G_1,G_2)}.
\end{align*}
Accordingly we bound $\prob{\sfd(\hat{\pi},\pi^*) = k}$ separately for large and small values of $k$. The next two propositions provide those bounds.

\begin{proposition}\label{prop:partial-possible}
    If $d=\omega(\log n)$ and $n\log\pth{\frac{1}{1-\rho^2}}+2d\log\pth{\frac{1}{1-r^2}}\ge (4+\epsilon)\log n$ with some constant $0<\epsilon<1$, then for any constant $0<\delta<1$ and $k\ge \delta n$, there exists $\hat\pi$ such that \begin{align*}
        &\prob{\sfd(\hat{\pi},\pi^*) = k}\le \exp\pth{-nh\pth{\frac{k}{n}}}\indc{k\le n-1}+\exp\pth{-2\log n}\indc{k=n}+\exp\pth{-\frac{\epsilon k\log n}{32}},
    \end{align*}
    where $h(x) = -x\log x-(1-x)\log(1-x)$ is the binary entropy function.
\end{proposition}

In view of Proposition~\ref{prop:partial-possible}, an upper bound is established for any $\sfd(\hat{\pi},\pi^*) = k$ with $k \ge \delta n$. Summing over all $k \ge \delta n$ yields an error estimate for $\prob{\overlap(\hat{\pi},\pi^*) \ge \delta}$ in the partial recovery regime. However, the proposition only controls the error probability when $k \ge \delta n$. To derive an error bound in the exact recovery regime, it remains necessary to handle the case $k < \delta n$. Specifically, we establish the following proposition.

\begin{proposition}\label{prop:exact-possible}
    If $n\log\pth{\frac{1}{1-\rho^2}}+d\log\pth{\frac{1}{1-r^2}}\ge (4+\epsilon)\log n$ with some constant $0<\epsilon<1$, then for any $k\le \frac{\epsilon}{16} n$, the estimator $\hat{\pi}$ in~\eqref{eq:estimator} satisfies \begin{align*}
        \prob{\sfd(\hat{\pi},\pi^*) = k}\le \exp\pth{-\frac{\epsilon}{8}k\log n}.
    \end{align*}
\end{proposition}

The proofs of Propositions~\ref{prop:partial-possible} and~\ref{prop:exact-possible} are deferred to Appendices~\ref{apd:proof-partial-possible} and~\ref{apd:proof-exact-possible}, respectively. By combining these two propositions, we obtain the possibility results stated in Theorems~\ref{thm:main-partial} and~\ref{thm:exact-main}, through summing over $k \ge \delta n$ and $k \ge 1$, respectively.
{
The main task behind Propositions~\ref{prop:partial-possible} and~\ref{prop:exact-possible} is to control the MLE score difference
$Z_\pi = S_\pi(G_1,G_2) - S_{\pi^*}(G_1,G_2)$ uniformly over all permutations under a
mixed Gaussian channel (continuous edges and high-dimensional features). For permutations
with macroscopic Hamming distance, we adapt the cycle decomposition and Gaussian
moment-generating-function computation from~\cite{wu2022settling} to this structural--feature
setting, obtaining sharp bounds with exponent $n\log(1/(1-\rho^2)) + 2d\log(1/(1-r^2))$. For permutations very close to $\pi^*$, we
represent $S_{\pi^*}-S_\pi$ as a quadratic form in a jointly Gaussian vector (edges and
features together), decorrelate the coordinates, and apply the Hanson-Wright inequality~\cite{hanson1971bound}
to get uniform bounds and the sharp constant in the exact recovery threshold.
}

\begin{remark}
    When two attributed graphs are only partially correlated through a latent injective mapping $\pi:S\subseteq V(G_1)\to T\subseteq V(G_2)$, the estimator in~\eqref{eq:estimator} can be naturally extended by optimizing over the set of injective mappings rather than bijections. Similar techniques have been used to establish information-theoretically optimal rates for partially overlapping graph alignment~\cite{huang2024information}. We leave this extension to future work.
\end{remark}

\subsection{Impossibility Results}\label{sec:impossibility}

In this subsection, we present information-theoretic impossibility results. For the converse arguments, we adopt a Bayesian formulation by endowing the ground-truth permutation $\pi^*$ with the uniform prior on $\maS_n$; under this prior, the MLE $\hat\pi$ minimizes the error probability among all estimators.
For impossibility results, it suffices to prove the failure of MLE, which corresponds to show the existence of a permutation $\pi'$ that achieves a higher likelihood than the true permutation $\pi^*$. However, such strategy only proves impossibility results for exact recovery regime (See Proposition~\ref{prop:impossible-exact}). We will show the impossibility results for the partial recovery regime by Fano's method (see, e.g. ~\cite[Section 2.10]{cover2006elements}). 

Let $M_\delta$ be a packing set of $\maS_n$ such that two distinct elements $\pi,\pi'\in \maM_\delta$ differs from a certain threshold. Specifically, we choose $\min_{\pi\neq \pi'\in \maM} \sfd(\pi,\pi')>(1-\delta)n$ in partial recovery regime and $\maM_1 = \maS_n$. The cardinality of $M_\delta$ measures the complexity of the parameter space under the corresponding metric.
Let $\maP$ denote the joint distribution of $(G_1,G_2)$, $\maQ$ be any distribution over $(G_1,G_2)$, and $D_{\mathrm{KL}}$ denote the Kullback–Leibler (KL) divergence. 
We then bound the mutual information $I(\pi^*;G_1,G_2)$ by $\max_{\pi\in \maS_n} D_{\mathrm{KL}}(\maP_{G_1,G_2|\pi}\Vert \maQ_{G_1,G_2})$.

By Fano's inequality, with $\pi^*$ being the discrete uniform prior in the packing set $\maM_\delta$, for any estimator $\hat{\pi}$, we have \begin{align}\label{eq:Fano}
    \prob{\overlap(\hat{\pi},\pi^*)<\delta}\ge 1-\frac{I(\pi^*;G_1,G_2)+\log 2}{\log |\maM_\delta|}.
\end{align}
Specifically, we have the following proposition.


\begin{proposition}[Impossibility result, partial recovery]\label{prop:impossible-partial}
    For any $0<\delta\le 1$, if $n\log\pth{\frac{1}{1-\rho^2}}+2d\log\pth{\frac{1}{1-r^2}}\le c\log n$ for some constant $c$, then \begin{align*}
        \prob{\overlap(\hat{\pi},\pi^*)<\delta}\ge 1-\frac{c}{4\delta}.
    \end{align*}
\end{proposition}

Proposition~\ref{prop:impossible-partial} provides an impossibility result for partial recovery, highlighting the relationship between the recovery probability and the threshold. This result also extends to the exact recovery regime when $\delta = 1$. The following proposition strengthens this conclusion under an assumption for the exact recovery regime, achieving both vanishing error and a sharp constant in the threshold.

\begin{proposition}[Impossibility result, exact recovery]\label{prop:impossible-exact}
    If $n\log\pth{\frac{1}{1-\rho^2}}+d\log\pth{\frac{1}{1-r^2}}+4\log d\le (4-\epsilon)\log n$  for some constant $\epsilon>0$ under the assumption $r^2\ge \frac{40}{d}$, then for any estimator $\hat{\pi}$, \begin{align*}
        \prob{\hat{\pi}\neq \pi^*} = 1-o(1).
    \end{align*}
\end{proposition}

By Propositions~\ref{prop:exact-possible} and~\ref{prop:impossible-exact}, we derive sharp thresholds for exact recovery with a gap of $4\log d$. When $\log n\ll d=n^{o(1)}$, the threshold is tight at the constant level. 


\section{QPAlign: Quadratic Programming relaxation for attributed graph Alignment}\label{sec:alg}

In Sections~\ref{sec:possibility} and~\ref{sec:impossibility}, we have shown that the MLE in~\eqref{eq:estimator} achieves the optimal information-theoretic thresholds. However, this estimator requires an exhaustive search over all possible mappings in $\maS_n$, which has a runtime of order $n!$. To address this computational bottleneck, we propose QPAlign, an approximation algorithm for attributed graph alignment.

Let $[n]\triangleq\sth{1,2,\cdots,n}$.
Without loss of generality, we assume $V(G_1) = V(G_2) = [n]$, $\pi:[n]\mapsto [n]$ and $E(G_1) = E(G_2) = \sth{(i,j):1\le i<j\le n}$.
Let $\Pi$ be the permutation matrix of $\pi$ with $\Pi_{ij} = \indc{\pi(i)=j}$ for any $1\le i,j\le n$, and define $\lambda\triangleq \tfrac{\varphi(\rho)}{\varphi(\rho)+\varphi(r)}$. Then the MLE $\hat{\pi}$ in~\eqref{eq:estimator} is equivalent to minimizing
\begin{equation}\label{eq:estimator-1}
\begin{aligned}
     \lambda &~\sum_{1\le i<j\le n}\pth{\beta_{ij}(G_1)-\beta_{\pi(i)\pi(j)}(G_2)}^2+(1-\lambda) \sum_{1\le i\le n} \Vert \bm{x}_i-\bm{y}_{\pi(i)}\Vert^2.
\end{aligned}
\end{equation}
Denote $A_1,A_2$ as the adjacent matrices of $G_1,G_2$. Let $B_1^i = \mathrm{diag}\sth{\bm{x}_{1i},\bm{x}_{2i},\cdots,\bm{x}_{ni}}$ and $B_{2}^i=\mathrm{diag}\sth{\bm{y}_{1i},\bm{y}_{2i},\cdots,\bm{y}_{ni}}$ for any $i\in [d]$, where $\bm{x}_{ki}$ corresponds to the $i-$component of vector $\bm{x}_k$.
{
Then minimizing~\eqref{eq:estimator-1} is equivalent to minimize the following function
\begin{align*}
    &~f(\Pi)\triangleq\lambda \Vert A_1\Pi-\Pi A_2 \Vert_F^2+(1-\lambda) \sum_{i=1}^d \Vert B_1^i\Pi-\Pi B_2^i\Vert_F^2,
\end{align*}
}
where $\Pi\in \mathbb{P}^n\triangleq\sth{\mathbf{P}\in \sth{0,1}^{n\times n},\mathbf{P1}=\mathbf{1},\mathbf{P}^\top \mathbf{1} = \mathbf{1}}$. Indeed, this is an instance of the \emph{quadratic assignment problem} (QAP)~\cite{pardalos1994quadratic,burkard1998quadratic}, which is NP-hard to solve or to approximate~\cite{makarychev2010maximum}.
To obtain a computationally efficient algorithm for estimating $\hat{\pi}$, we employ a relaxation approach.  
Relaxing the set of permutations to  Birkhoff polytope (the set of doubly stochastic matrices)
\begin{align*}
\mathbb{W}^n\triangleq\{\mathbf{W}\in [0,1]^{n\times n}: \mathbf{W1}=\mathbf{1},\mathbf{W}^\top \mathbf{1} = \mathbf{1},
0\le \mathbf{W}_{ij}\le 1 & \text{ for all }i,j\; \},
\end{align*}
{
we derive the quadratic programming (QP) relaxation $\min_{\Pi\in\mathbb W^n}f(\Pi).$
}
{We  solve the above quadratic programming by projected gradient descent. Specifically, we project the matrix to $\mathbb{W}_n$ by Euclidean projection: $\Pi^{k+1} = \mathsf{Proj}_{\mathbb{W}_n}(\Pi^k-\eta \nabla f(\Pi^k))$, where $\eta$ is the step size. We have the following Theorem on the convergence guarantee for the gradient descent method.}
{
\begin{proposition}\label{prop:convergenceguarantee}
    For any two graphs $G_1, G_2$, there exists a constant $L$ such that for any step size $\eta\le L^{-1}$, for any $0<\delta<1$, if $d>32\log\frac{n}{\sqrt\delta}$, then with probability at least $1-\delta$, 
    $$\|\Pi^K-\Pi'\|_F\le \sqrt{\frac{n}{(1-\lambda)(1-r)d\eta K}}$$
    for any $K\ge 1$, where $\Pi'\in \argmin_{\Pi\in \mathbb W^n} f(\Pi)$, and $\Pi^0$ is the initial state.
\end{proposition}
}

 Relative to~\cite{fan2019spectral}, the node-feature term in our objective plays an important role to their regularization: it ensures the convergence rate remains bounded whenever $\lambda\neq 1$ and $r<1$. Consequently, incorporating node features renders our algorithm stable without introducing any extra regularization term. 
Indeed, relaxing to Birkhoff polytope is widely adopted in graph matching~\cite{vogelstein2015fast,bommakanti2024fugal}, and has been proved tight in random graph models~\cite{fan2019spectral,fan2023spectral}.
In practice, we often regard $\lambda$ as a tuning parameter and adapt model selection technique for picking $\lambda$. 
{By Proposition~\ref{prop:convergenceguarantee}, we obtain a standard $O(1/K)$ convergence guarantee for the projected gradient descent scheme with exact Euclidean projections onto $\mathbb W_n$. However, in our implementation, we replace these exact projections by a few iterations of Sinkhorn scaling~\cite{sinkhorn1964relationship} as a fast approximate projection onto $\mathbb W_n$. In practice, since Sinkhorn requires nonnegative entries, we apply it to the truncated matrix $(\Pi^{(t+1)})_{+}$, obtained by setting all negative entries of $\Pi^{(t+1)}$ to zero.}

Define $D\in\mathbb{R}^{n\times n}$ as $D_{kj}=\|\bm{x}_k-\bm{y}_j\|_2^2$. We note that $\sum_{i=1}^d \Vert B_1^i\Pi-\Pi B_2^i\Vert_F^2 = \sum_{k,j} D_{kj} \, \Pi_{kj}^2$. Note that $\sum_{k,j}\Pi_{kj}(1-\Pi_{kj}) = 0$ for permutation matrix $\Pi$. We turn this constraint into a regularizer with parameter $\mu$ and derive the following program:
\begin{equation}\label{eq:program-final}
\begin{aligned}
\min_{\Pi \in \mathbb W_n} &\left\{
 \lambda \|A_1\Pi - \Pi A_2\|_F^2 
\;+\; (1-\lambda) \sum_{k,j} D_{kj} \, \Pi_{kj}^2\right.\left.+\mu\sum_{k,j} \Pi_{kj}(1-\Pi_{kj})\right\}.
\end{aligned}
\end{equation}
We solve the problem in~\eqref{eq:program-final} by QPAlign in Algorithm~\ref{alg:qap-relax}. In the following, we outline a general recipe for QPAlign.
\begin{itemize}
    \item \emph{Gradient descent.} We update $\Pi^{(t+1)} = \Pi^{(t)}-\eta G^{(t)}$ with step size $\eta>0$, where the gradient is given by \begin{align*}
        G^{(t)} = &~2\lambda(A_1^\top E - E A_2^\top)+2(1-\lambda )\pth{D\circ \Pi^{(t)}}+\mu \pth{J_{n\times n}-2\Pi^{(t)}},
    \end{align*}
    where $E = A_1 \Pi^{(t)} - \Pi^{(t)} A_2$ and $(D\circ\Pi^{(t)})_{ij} = D_{ij} \Pi^{(t)}_{ij},J_{ij}=1$ for any $i,j$.
    \item \emph{Sinkhorn normalization.} After each gradient step, update $\Pi^{(t+1)}$ on the set of doubly stochastic matrices $\mathbb{W}^n$ using $K$ iterations of the Sinkhorn normalization procedure~\cite{sinkhorn1964relationship}.
    \item \emph{Rounding via Hungarian algorithm.}  Once convergence is reached, the final doubly stochastic matrix is converted into a permutation matrix by solving the problem
    $\argmax_{\pi\in \maS_n} \sum_{i} \Pi_{i,\pi(i)}^{(t)}$
    using the Hungarian algorithm~\cite{kuhn1955hungarian,munkres1957algorithms}.
\end{itemize}

\paragraph{Time complexity} The construction of the feature-distance matrix $D$ requires $O(dn^2)$ time. Each gradient step has a complexity of { $O(n^3)$}, and the Sinkhorn algorithm with $K$ iterations takes {$O(Kn^2)$}. Consequently, performing $T$ iterations of gradient descent and Sinkhorn projection costs {$O(T(K+n)n^2)$}. The final rounding via the Hungarian algorithm requires $O(n^3)$~\cite{munkres1957algorithms}. Overall, the time complexity of QPAlign is $O((d +T(K+n))n^2)$. We have conducted experiments in Section~\ref{sec:num} on graphs with up to 3,000 nodes. For larger graphs, one possible direction is to first prune the candidate matches for each node using features or local structural signatures, then solve the alignment problem on a compressed graph, and finally replace dense matrix scaling with sparse, approximate, or distributed Sinkhorn updates.



 \begin{algorithm}[t]
 \caption{QPAlign: Quadratic Programming relaxation for attributed graph Alignment}
 \label{alg:qap-relax}
 \begin{algorithmic}[1]
 \STATE\textbf{Input:} Adjacency matrices $A_1,A_2\in\mathbb{R}^{n\times n}$; node features $\bm{x}_i,\bm{y}_i$, $1\le i\le n$; weights $\lambda,\mu>0$; step size $\eta>0$; Sinkhorn iters $K$; max iters $T$; tolerance $\mathrm{tol}$.
 \STATE\textbf{Output:} Estimated permutation $\hat{\pi}$.
 \STATE Construct feature-distance matrix $D \in \mathbb{R}^{n\times n}$ with $D_{ij} = \Vert \bm{x}_i-\bm{y}_j\Vert_2^2$.
 \STATE Initialize $\Pi^{(0)}$.
 \FOR{$t = 0,\dots,T-1$}
   \STATE $E \leftarrow A_1\Pi^{(t)} - \Pi^{(t)} A_2$. 
   \STATE $f^{(t)}\leftarrow \lambda \Vert E\Vert_F^2+(1-\lambda) \sum_{i,j} D_{ij} (\Pi_{ij}^{(t)})^2+\mu\sum_{i,j}\Pi_{ij}^{(t)} (1-\Pi_{ij}^{(t)})$.
   \STATE $G^{(t)}\leftarrow 2\lambda(A_1^\top E - E A_2^\top)+2(1-\lambda )\pth{D\circ \Pi^{(t)}}+\mu \pth{J_{n\times n}-2\Pi^{(t)}}$.
   \STATE Gradient step: $\Pi^{(t+1)} \leftarrow \Pi^{(t)} - \eta G^{(t)}$.
   \STATE {Truncate negative entries: ${\Pi}^{(t+1)} \leftarrow (\Pi^{(t+1)})_{+}$, $(\cdot)_+$: elementwise $\max\{\cdot,0\}$.}
   \STATE Update $\Pi^{(t+1)}$ on the Birkhoff polytope via Sinkhorn: $\Pi^{(t+1)} \leftarrow \textbf{Sinkhorn}(({\Pi}^{(t+1)})_+,K)$.
   \IF{$t>0$\text{ and }$|f^{(t)}-f^{(t-1)}|<\mathrm{tol}$}
      \STATE \textbf{break}
   \ENDIF
 \ENDFOR
 \STATE Round to a permutation via Hungarian algorithm: $\hat \pi \leftarrow \arg\max_{\pi\in \maS_n} \sum_i \Pi^{(t)}_{i,\pi(i)}$.
 \STATE \textbf{Return:} $\hat\pi$.
 \end{algorithmic}
 \end{algorithm}

In practice, when solving the quadratic program in~\eqref{eq:program-final}, the parameter $\lambda$ is typically difficult to estimate from the observations $G_1$ and $G_2$, since the correlation structure between the two graphs (and hence the relative reliability of topology versus features) is usually unknown. In the following, we establish recovery guarantees that hold uniformly over all $\lambda\in(\delta,1-\delta)$, for any fixed constant $\delta\in(0,1/2)$. Define $\mathsf{Edge}=n\log\frac{1}{1-\rho^2}$ and $\mathsf{Vertex}=d\log\frac{1}{1-r^2}$, which quantify the mutual information contributed by edges and vertices, respectively.

\begin{assumption}\label{assm:choiceoflambda}
We assume that there exists a universal constant $\Gamma$, such that $\frac1\Gamma\le \frac{\mathsf{Edge}}{\mathsf{Vertex}}\le  \Gamma$.
\end{assumption}
This assumption is reasonable in many practical settings. If it is known a priori that either the topological information or the feature information is unreliable, one may instead employ algorithms that rely solely on the reliable component to achieve the recovery objective. Therefore, we focus here exclusively on the balanced regime, in which neither the topological nor the feature information is extremely unreliable, and both provide comparably informative signals.
For any $0<\lambda<1$, we consider the estimator \begin{align*}
     \hat \pi_\lambda=\mathop{\text{argmax}}_{\pi\in \maS_n}\left\{\lambda\sum_{e\in E(G_1)}\beta_e(G_1)\beta_{\pi(e)}(G_2)+(1-\lambda)\sum_{v\in V(G_1)}\bm x_v^\top\bm y_{\pi(v)}\right\}.
\end{align*}
The following proposition provides theoretical guarantee on $\hat\pi_\lambda$ for any $\lambda\in (\delta,1-\delta)$.


\begin{proposition}\label{prop:possible-all-lambda}
    Under Assumption~\ref{assm:choiceoflambda}, for any constant $\delta\in (0,1/2)$,
    if $d=\omega(\log n)$ and $n\log\frac1{1-\rho^2}+d\log\frac1{1-r^2}\ge C_0\log n$ for some constant $C_0 = C_0(\delta, \Gamma)$, then there exists an estimator $\hat\pi_\lambda$ only depending on $\lambda,G_1,G_2$ such that, for all $\lambda\in (\delta, 1-\delta)$, $$\prob{\hat\pi_\lambda\ne \pi^*}=o(1).$$
\end{proposition}

 Proposition~\ref{prop:possible-all-lambda} bridges the gap between the information-theoretic results and the theoretical guarantee for the algorithm, demonstrating the achievability of the oracle solution $\hat\pi_\lambda$ for $\lambda$.

\begin{remark}[Regularization term]\label{rmk:regular-term}
We add a regularization term $\sum_{k,j}\Pi_{kj}(1-\Pi_{kj})$ in~\eqref{eq:program-final}. This term encourages the entries of $\Pi$ to approach 0 or 1. {While some previous works employ a penalty of the form $\|\Pi\|_F^2$ to obtain an explicit solution (see, e.g.~\cite{fan2020spectral}), our relaxation directly pushes the solution toward the boundary of the Birkhoff polytope.}
As a result, the intermediate matrix $\Pi^{(t)}$ becomes more concentrated near the vertices of the Birkhoff polytope, which allows the final rounding via the Hungarian algorithm to be more precise. Therefore, the inclusion of this regularization term helps to improve the overall accuracy of the estimated permutation $\hat{\pi}$.
\end{remark}

\section{Numerical Results}\label{sec:num}

\subsection{Simulation Studies}



In this subsection, we provide numerical results for QPAlign in Algorithm~\ref{alg:qap-relax} on synthetic data. One related model is the correlated Gaussian-attributed \ER model~\cite{yang2024exact}, where the correlated pairs of edges follow a multivariate Bernoulli distribution with connection probability $p$ and correlation $\rho$. Fixing $n=3000$ and $d=512$ 
(with $p=0.5$ in the \ER case), we run the algorithm and report the value of $\overlap(\hat{\pi},\pi^*)$ for varying correlations $\rho\in[0,1]$ and $r\in[0,1]$. We evaluate our method with step size $\eta=10^{-4}$, $T=400$, $K=80$, $\lambda=0.1$ and $\mu=0.1$ for the experiments reported here.

Our results are summarized in Figure~\ref{fig:heatmaps}. Figure~\ref{fig:heatmap1} displays the heatmap of the overlap under the featured correlated Gaussian Wigner model. We observe that the overlap increases smoothly with both $\rho$ and $r$, starting from nearly zero when both correlations vanish and approaching one when both correlations are close to unity. This indicates that the estimator $\hat\pi$ gradually aligns with the ground truth $\pi^*$ as the signal in either edges or features becomes stronger. For the low-dimensional setting with $n=100$ and $d=16$, we present additional results in Figure~\ref{fig:heatmaps-n100} in Appendix~\ref{apd:synthetic-data}, which exhibit the same qualitative trend, confirming the effectiveness of our method in both regimes.

\begin{figure}[htbp]
    \centering
    \begin{subfigure}[b]{0.49\textwidth}  
        \centering
        \includegraphics[width=0.8\textwidth]{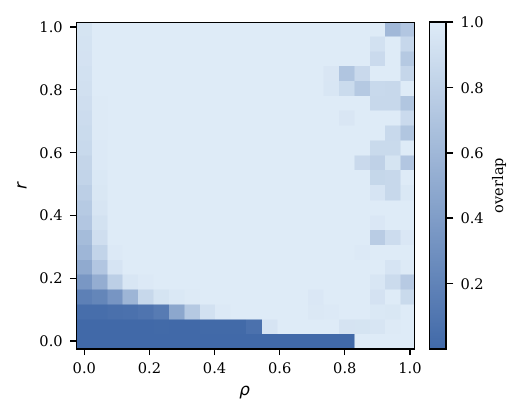}
        \caption{Gaussian Wigner model: $n=3000$ and $d=512$.}
        \label{fig:heatmap1}
    \end{subfigure}
    \hfill
    \begin{subfigure}[b]{0.49\textwidth}
        \centering
        \includegraphics[width=0.8\textwidth]{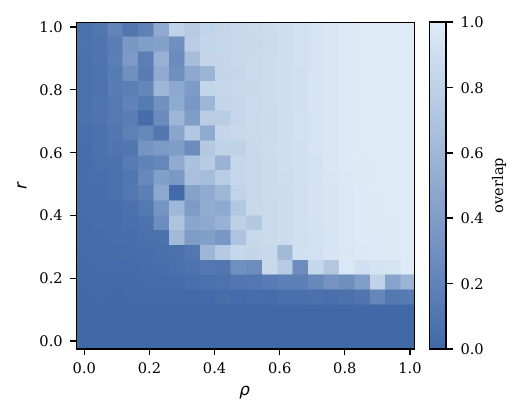}
        \caption{\ER model: $n=3000$ and $d=512$.}
        \label{fig:heatmap2}
    \end{subfigure}
    \caption{Overlap between the estimator $\hat{\pi}$ in Algorithm~\ref{alg:qap-relax} and the ground truth $\pi^*$ in two models with $n=3000$ and $d=512$, evaluated across varying correlations $\rho\in [0,1]$ and $r\in [0,1]$.}
    \label{fig:heatmaps}
\end{figure}
Importantly, these numerical results are consistent with the information-theoretic exact recovery thresholds given in Theorem~\ref{thm:exact-main}. 
We also note that in certain intermediate regimes, there exists a statistic-computation gap: while exact recovery is theoretically possible, computationally achieving it may require stronger correlations. {See Figure~\ref{fig:IT-QPAlign} for a more detailed comparison between the information-theoretic thresholds and the empirical phase-transition boundaries of QPAlign.} {The numerical behavior aligns with the theoretical thresholds established in Theorem~\ref{thm:exact-main}, while also highlighting the presence of a statistical–computational gap in certain intermediate regimes. In particular, Figure~\ref{fig:IT-QPAlign} illustrates how the empirical phase-transition boundaries of QPAlign, for different choices of $\lambda \in \sth{0.1,0.2,\ldots,1.0}$, closely track the information-theoretic limit, thereby providing a direct connection between the algorithm’s empirical success and our derived thresholds.}

\begin{figure}[htbp]
    \centering
    \includegraphics[width=0.35\textwidth]{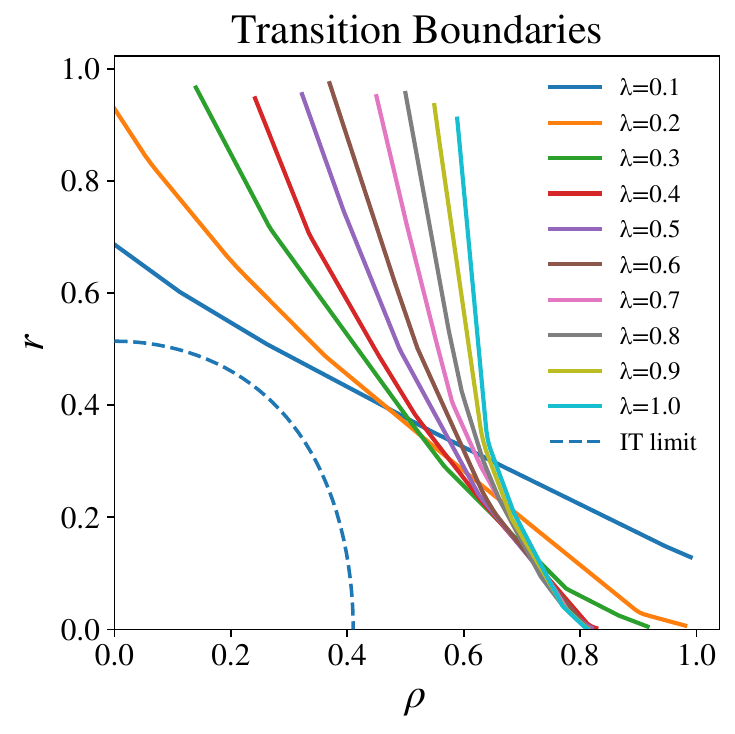}
\caption{Phase-transition boundaries of QPAlign under different regularization parameters $\lambda$, together with the information-theoretic exact recovery limit.} 
\label{fig:IT-QPAlign}
\end{figure}

Figure~\ref{fig:heatmap2} shows the corresponding result under the featured correlated \ER model with $p=0.5$. 
A qualitatively similar pattern is observed: the algorithm achieves high overlap once either $\rho$ or $r$ is sufficiently large. Together, these experiments confirm that our algorithm behaves stably across different correlation regimes and successfully interpolates between weak and strong signal cases, with $\overlap(\hat{\pi},\pi^*)$ ranging from $0$ to $1$ as $(\rho,r)$ varies from $(0,0)$ to $(1,1)$. The results show that our approach is effective for both weighted and unweighted graphs, which broadens the applicability relative to prior methods. 
See more comparisons with the benchmarks in Appendix~\ref{apd:synthetic-data}.
We also conduct ablation studies on synthetic data with $n=100$ and $d=16$ to verify the effectiveness of regularization. In the Gaussian Wigner model, using a positive regularization weight $\mu=0.1$ (instead of $\mu=0$) improves the overlap on 86.57\% of the parameter grid points, while in the \ER model the corresponding proportion is 60.74\%, where $\mu$ is the weight on the regularization term introduced in equation~\eqref{eq:program-final}.



\subsection{Real Data Analysis}
\paragraph{ACM-DBLP dataset}
The ACM-DBLP dataset~\cite{tang2008arnetminer} is a widely used benchmark for attributed graph alignment, containing 2,224 ground-truth matched pairs. In our construction, each node represents a paper from either the ACM or DBLP source, and edges are weighted by co-authorship relations.  
The ground-truth alignment is given by the set of papers appearing in both sources. 
We implemented the experiments with hyperparameters $\mu=0.01$, $T=1000$, $K=200$, step size $\eta=10^{-5}$, and $\lambda \in \sth{0,0.2,0.4,0.6,0.8,1}$. 

\paragraph{Douban (Online–Offline) dataset}
The Douban Online–Offline dataset~\cite{tang2008arnetminer} is another widely used benchmark for attributed graph alignment, consisting of two graphs that share 1,118 ground-truth matched pairs. Each node represents a user, with edges in the online graph encoding platform interactions (e.g., replying to a post) and edges in the offline graph capturing co-attendance at social events. Node features are given by user locations. The online graph strictly contains all users from the offline graph, and the ground-truth alignment is defined by the users present in both. We implemented the experiments with hyperparameters $\mu=0$, $T=1000$, $K=200$, step size $\eta=5\times 10^{-3}$, and $\lambda \in \sth{0,0.2,0.4,0.6,0.8,1}$.

Indeed, the ACM-DBLP dataset corresponds to a featured Gaussian-Wigner graph, while the Douban dataset is treated as a featured \ER graph, each representing different structural settings for graph alignment. We compare our method with three types of baselines: 1) based solely on edge structure (Grampa~\cite{fan2019spectral}, IsoRank~\cite{singh2008global}, Umeyama~\cite{umeyama1988eigendecomposition}, GW~\cite{peyre2016gromov}); 2) based solely on node features (MAP~\cite{dai2019database}, kNN); and 3) exploiting both edge structure and node features (FGW~\cite{titouan2019optimal}, REGAL~\cite{heimann2018regal}, PARROT~\cite{zeng2023parrot}).
To ensure a fair comparison, we follow the official implementations and parameter choices recommended in the original papers. Since the baselines are designed for different settings (edge-only, feature-only, or joint), { the results should be viewed within their respective information settings rather than as direct head-to-head comparisons.
The results are reported in Figure~\ref{fig:acm-dblp} and Table~\ref{tab:acm-dblp}.} 

\begin{figure}[htbp]
    \centering
    \includegraphics[width=0.35\linewidth]{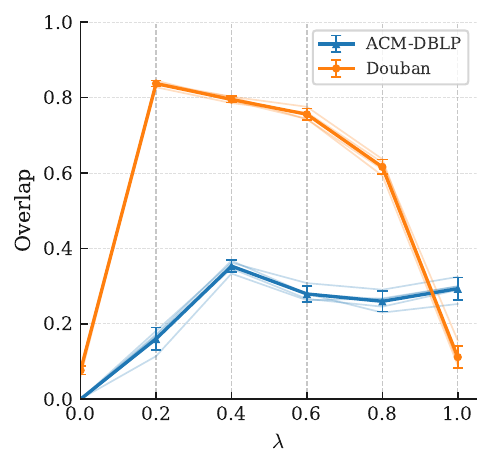}
    \caption{Overlap vs.\ $\lambda$ on ACM-DBLP and Douban datasets.}
    \label{fig:acm-dblp}
\end{figure}%

We report the experimental results averaged over 5 random seeds. The faint curves represent the results of individual runs, while the bold curves show their average.
To fairly compare with baselines using the corresponding information source in Table~\ref{tab:acm-dblp}, we conducted two ablated versions: $\lambda=0$, which corresponds to using only feature information, and $\lambda=1$, which corresponds to using only edge information.  The results in Figure~\ref{fig:acm-dblp} and Table~\ref{tab:acm-dblp} both demonstrate that combining the two sources of information yields performance that surpasses relying on either source alone.
See more details in Appendix~\ref{appendix:acm-dblp}. We also conduct experiments on spatial transcriptomic data; see Appendix~\ref{apd:stdata} for details.

\begin{table}[htbp]
    \centering
    \footnotesize
    \begin{tabular}{ccc}
      \toprule
       & ACM-DBLP & Douban\\
      \midrule
      QPAlign (max) & \textbf{0.3445} & 0.8370 \\
      FGW & 0.0018 & 0.2773 \\
      REGAL & 0.0301 & 0.1118 \\
      PARROT & 0.0441 & \textbf{0.8462} \\
      \midrule
      QPAlign ($\lambda=0$) & 0.0004 & 0.0767 \\
      MAP & 0.0004 & 0.0411 \\
      kNN & 0.0004 & 0.0725 \\
      \midrule
      QPAlign ($\lambda=1$) & 0.2896 & 0.1118 \\
      Grampa & 0.0746 & 0.0027 \\
      IsoRank & 0.0018 &  0.0000 \\
      Umeyama & 0.0346 & 0.0089 \\
      GW & 0.0202 & 0.0000 \\
      \bottomrule
    \end{tabular}
    \caption{Alignment accuracy in ACM-DBLP and Douban datasets.}
    \label{tab:acm-dblp}
\end{table}

\section{Discussions and Future Directions}
In this paper, we studied the graph alignment problem  where both weighted edges and node features are jointly observed under an unknown vertex permutation. We established sharp information-theoretic thresholds for both partial and exact recovery in the featured correlated Gaussian Wigner model, revealing how structural and feature correlations together govern the fundamental limits of alignment.
Our theoretical analysis establishes the optimal rates for both partial and exact recovery regimes. These results primarily depend on the analysis of the maximum likelihood estimator, where careful weighting of edge and feature information is selected to achieve optimal results. This provides a unified theoretical understanding of several previously studied alignment models and highlights the benefits of jointly leveraging both topology and attributes.

From an algorithmic perspective, we proposed QPAlign, which efficiently combines edge and feature information to achieve recovery with theoretical guarantees, confirming the achievability of the oracle solution and convergence. Empirical results on synthetic data and real datasets demonstrate that QPAlign performs effectively, achieving high-quality alignments and comparing favorably with existing baselines. 
There are also several promising directions for future work.
\begin{itemize}
    \item \emph{Extension to partially overlap.} Our framework can be extended to the setting where only a pair of subgraphs of the original graphs are correlated through a latent injective mapping. The optimal rate under this partially overlapping featured correlation model remains unknown.
    \item \emph{Statistical--computational gap.} The computational limits of this model remain unknown. A possible direction is to study this question via the low-degree framework~\cite{hopkins2017power,hopkins2018statistical}.
    \item \emph{Non-Gaussian and heavy-tailed distributions.} An interesting direction is to investigate whether our results under the Gaussian assumption can be extended to more general non-Gaussian settings, including heavy-tailed distributions.
\end{itemize}
\appendix

\section{Experimental Details}

\subsection{Synthetic Data}\label{apd:synthetic-data}

Figure~\ref{fig:heatmaps-n100} reports the results for the low-dimensional setting with $n=100$ and $d=16$. We again observe that the overlap increases monotonically with both $\rho$ and $r$, starting near zero when both correlations vanish and approaching one when either correlation becomes large. These results mirror the high-dimensional case in Figure~\ref{fig:heatmaps}, thereby confirming that our method remains effective in both low-dimensional and high-dimensional regimes.


\begin{figure}[htbp]
    \centering
    \begin{subfigure}[b]{0.4\textwidth}  
        \centering
        \includegraphics[width=\textwidth]{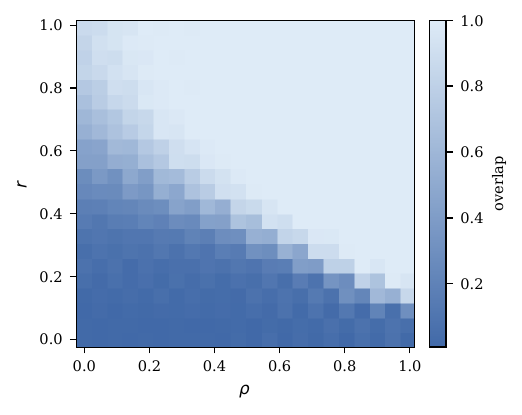}
        \caption{Featured correlated Gaussian Wigner model with $n=100$ and $d=16$.}
        \label{fig:heatmap1-n100}
    \end{subfigure}
    \hspace{0.04\textwidth} 
    \begin{subfigure}[b]{0.4\textwidth}
        \centering
        \includegraphics[width=\textwidth]{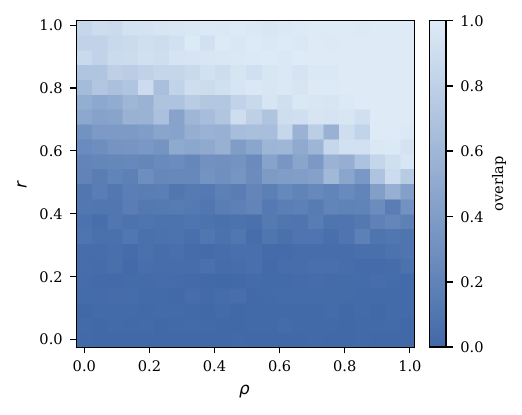}
        \caption{Featured correlated \ER model with $n=100$, $d=16$, and $p=0.5$.}
        \label{fig:heatmap2-n100}
    \end{subfigure}
    \caption{Overlap between the estimator $\hat{\pi}$ in Algorithm~\ref{alg:qap-relax} and the ground truth $\pi^*$ in two models with $n=100$ and $d=16$, evaluated across varying correlations $\rho\in [0,1]$ and $r\in [0,1]$.}
    \label{fig:heatmaps-n100}
\end{figure}


To facilitate a comprehensive comparison, we evaluate our approach against FGW with various fixed values of $r$, as well as against purely topology-based methods, including GW, Grampa, IsoRank, and Umeyama. All evaluations are conducted on synthetic featured Gaussian--Wigner graphs with $n=100$ vertices and $d=16$ dimensions. For the correlation parameter, we report results in terms of 
$\sigma = \sqrt{(1-\rho^2)/\rho^2} \in [0,0.5],$
which serves as a noise-to-signal ratio relative to the original graph, as showed in Figure~\ref{fig:overlap_vs_rho}. We adopt $\sigma$ instead of $\rho$ since several algorithms exhibit sharp performance transitions when $\rho \approx 1$ (i.e., $\sigma \approx 0$), making results easier to interpret under this reparametrization. In addition, we compare with FGW at fixed $\rho$ and with MAP across different values of $r \in [0,1]$ in Figure~\ref{fig:overlap_vs_r}. Because MAP degenerates and becomes numerically unstable at $r=1$, we replace the endpoint with $r=0.999$ for all methods to ensure consistent and stable evaluation. In the synthetic data experiments presented here, our method is evaluated with step size $\eta = 10^{-4}$, $T=400$, $K=80$,  $\lambda=0.1$, and $\mu=0.1$.

Figure~\ref{fig:overlap_vs_rho} reports the alignment accuracy as a function of $\sigma=\sqrt{(1-\rho^2)/\rho^2}$ with $\lambda = 0.05,0.1$, and $0.15$, respectively, where smaller $\sigma$ corresponds to stronger graph correlation. Our method consistently outperforms the purely edge-based baselines (GW, Grampa, IsoRank, Umeyama) and the joint edge–feature baseline FGW across different feature correlations $r$. Notably, even when $r$ is small (e.g., $r=0.1$), our approach achieves higher overlap than FGW under the same setting, indicating robustness to weak feature correlation. By contrast, classical spectral and matching-based methods (Grampa, IsoRank, Umeyama, GW) quickly degrade as noise increases.

\begin{figure}[htbp]
    \centering
    \begin{subfigure}[b]{0.32\textwidth}  
        \centering
        \includegraphics[width=\textwidth]{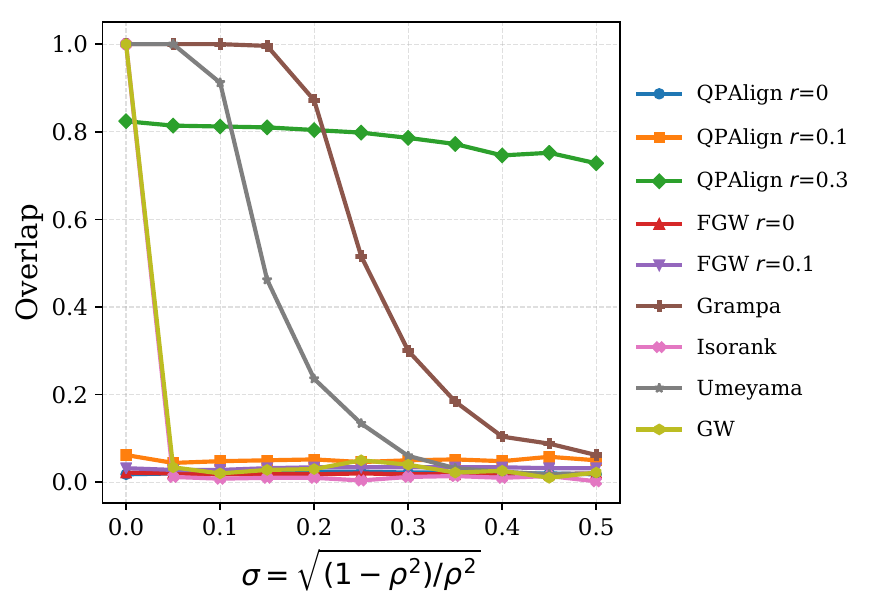}
        \caption{ $\lambda=0.05$.}
        \label{fig:overlap_vs_rho-005}
    \end{subfigure}
    \hfill
    \begin{subfigure}[b]{0.32\textwidth}
        \centering
        \includegraphics[width=\textwidth]{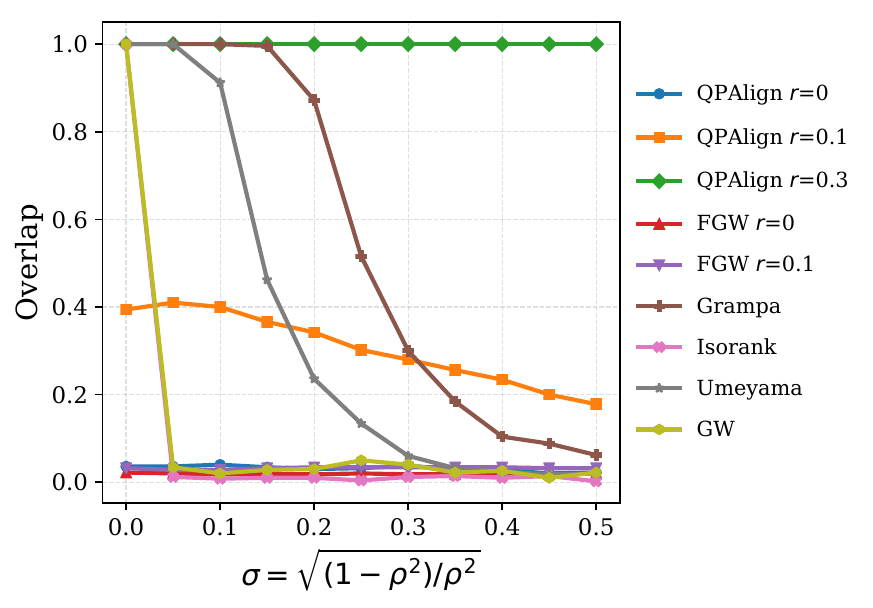}
        \caption{$\lambda=0.1$.}
        \label{fig:overlap_vs_rho-01}
    \end{subfigure}
    \hfill
    \begin{subfigure}[b]{0.32\textwidth}
        \centering
        \includegraphics[width=\textwidth]{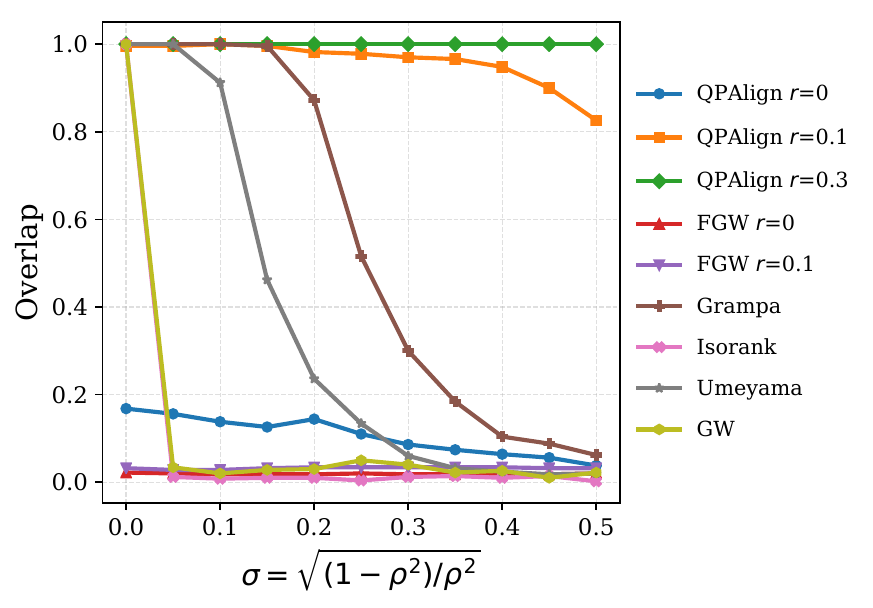}
        \caption{$\lambda=0.15$.}
        \label{fig:overlap_vs_rho-015}
    \end{subfigure}
    \caption{Overlap between the estimator $\hat{\pi}$ in Algorithm~\ref{alg:qap-relax} and the ground truth $\pi^*$ evaluated by different algorithms across varying correlations $\sigma\in [0,0.5]$ with different $\lambda$.}
    \label{fig:overlap_vs_rho}
\end{figure}

Figure~\ref{fig:overlap_vs_r} shows the overlap as a function of feature correlation $r$ with $\lambda = 0.05,0.1$, and $0.15$, respectively, under different edge correlations $\rho$. Our method again demonstrates superior performance, achieving near-perfect alignment at much smaller $r$ compared to FGW and MAP. For example, with $\rho=0.8$, our method reaches almost perfect overlap already at $r= 0.2$, whereas FGW and MAP require significantly larger $r$ to attain comparable accuracy. Overall, except for the degenerate case $r=0$ or $\rho=0$, our method consistently achieves higher overlap than existing baselines under the same $r$ or $\rho$.


\begin{figure}[htbp]
    \centering
    \begin{subfigure}[b]{0.32\textwidth}  
        \centering
        \includegraphics[width=\textwidth]{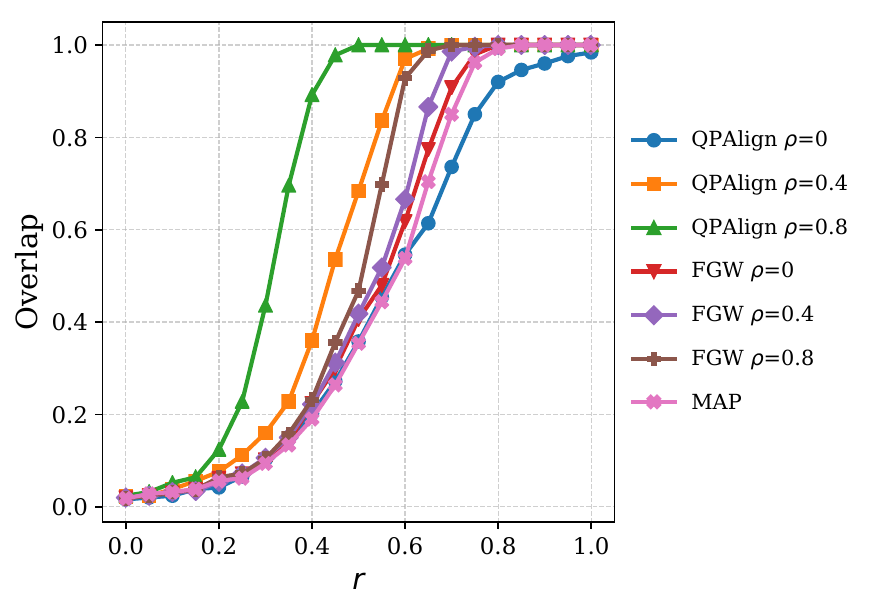}
        \caption{$\lambda=0.05$.}
        \label{fig:overlap_vs_r-005}
    \end{subfigure}
    \hfill
    \begin{subfigure}[b]{0.32\textwidth}
        \centering
        \includegraphics[width=\textwidth]{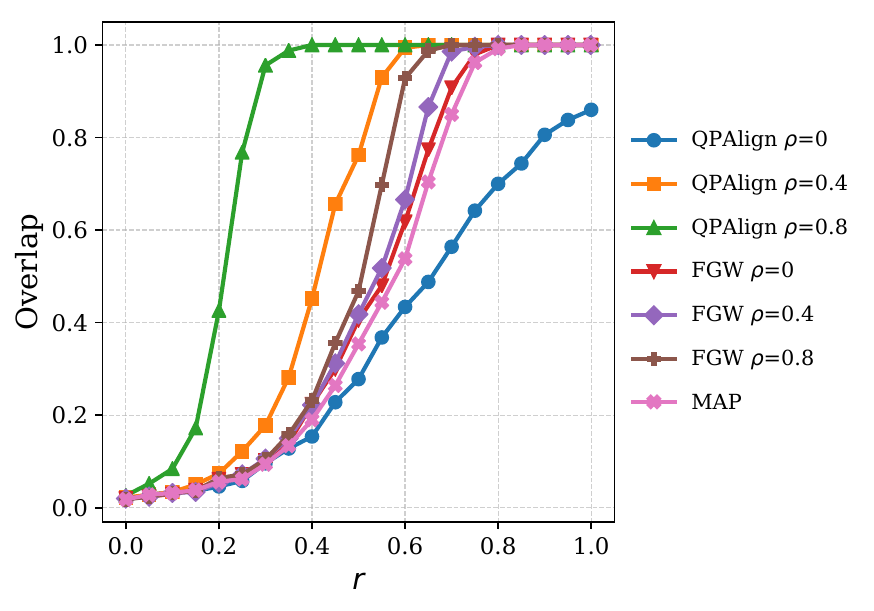}
        \caption{$\lambda=0.1$.}
        \label{fig:overlap_vs_r-01}
    \end{subfigure}
    \hfill
    \begin{subfigure}[b]{0.32\textwidth}
        \centering
        \includegraphics[width=\textwidth]{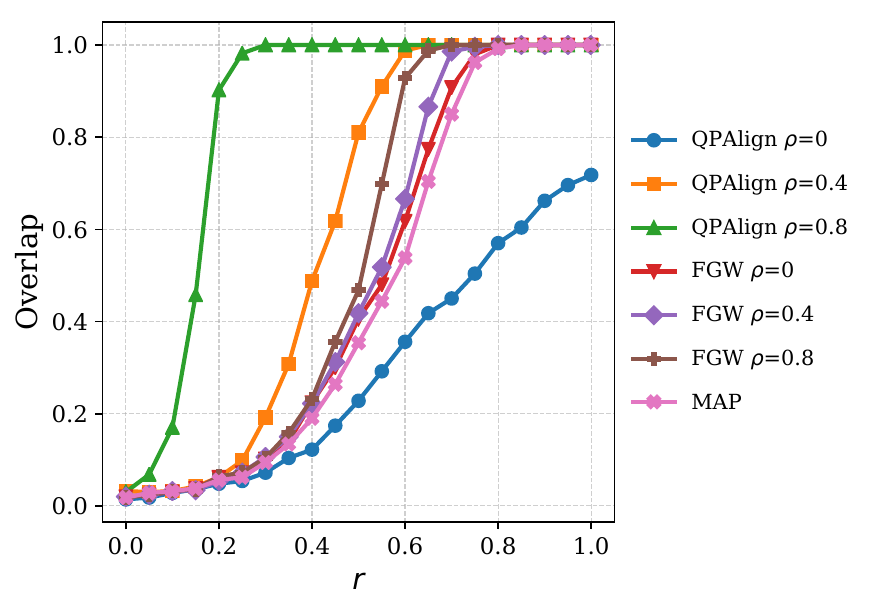}
        \caption{$\lambda=0.15$.}
        \label{fig:overlap_vs_r-015}
    \end{subfigure}
    \caption{Overlap between the estimator $\hat{\pi}$ in Algorithm~\ref{alg:qap-relax} and the ground truth $\pi^*$ evaluated by different algorithms across varying correlations $r\in [0,1]$ with different $\lambda$.}
    \label{fig:overlap_vs_r}
\end{figure}

To investigate the effect of non-Gaussianity and heavy tails, we also conduct experiments under a Student-$t$ model. Specifically, we consider the setting $n=100$ and $d=16$, and generate the edge weights and node features from Student-$t$ distributions with degrees of freedom $\nu \in \{3,5,10\}$. The results in Figure~\ref{fig:t-dis} show that QPAlign remains effective across all three choices of $\nu$: the overlap increases steadily as either $\rho$ or $r$ becomes larger, and the overall phase transition pattern is broadly consistent with that observed in the Gaussian case.

\begin{figure}[htbp]
    \centering
    \begin{subfigure}[b]{0.32\textwidth}  
        \centering
        \includegraphics[width=\textwidth]{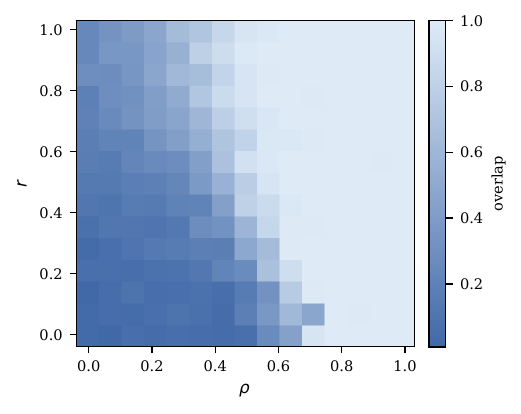}
        \caption{$\nu=3$.}
    \end{subfigure}
    \hfill
    \begin{subfigure}[b]{0.32\textwidth}
        \centering
        \includegraphics[width=\textwidth]{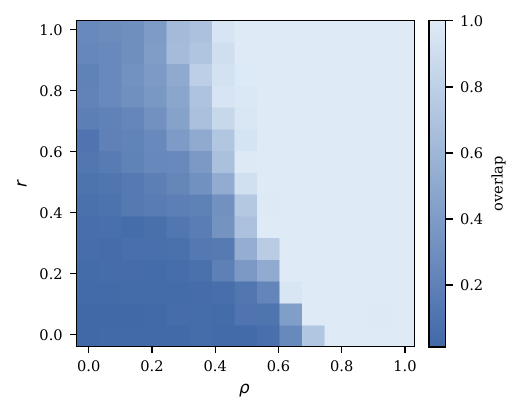}
        \caption{$\nu=5$.}
    \end{subfigure}
    \hfill
    \begin{subfigure}[b]{0.32\textwidth}
        \centering
        \includegraphics[width=\textwidth]{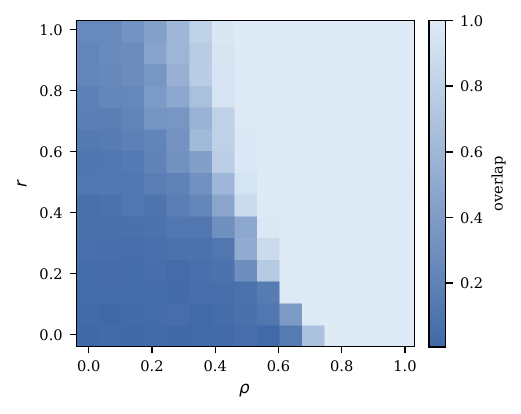}
        \caption{$\nu=10$.}
    \end{subfigure}
    \caption{Overlap between the estimator $\hat{\pi}$ and the ground truth $\pi^*$ under Student-$t$ distributions with degrees of freedom $\nu\in\{3,5,10\}$, for $n=100$ and $d=16$.}
    \label{fig:t-dis}
\end{figure}

To initialize for synthetic datasets, we leverage both feature and degree information to construct a mixed similarity matrix. Specifically, given node feature matrices $X=[\bm{x}_1^\top,\cdots,\bm{x}_n^\top]\in\mathbb{R}^{n\times d}$ and $Y=[\bm{y}_1^\top,\cdots,\bm{y}_n^\top]\in\mathbb{R}^{n\times d}$, we first compute a feature similarity matrix as $S_{\text{feat}}=\max(XY^\top,0)$, i.e., the inner product similarity clamped elementwise at zero. We then compute a degree similarity matrix by setting $d_1=A_1\mathbf{1}$ and $d_2=A_2\mathbf{1}$, and defining $S_{\deg}=(1+\lvert d_1\mathbf{1}^\top-\mathbf{1}d_2^\top\rvert)^{-1}$. These two components are combined into a mixed similarity matrix, $S = S_{\text{feat}} + \nu S_{\deg}$, which balances feature and structural signals. We empirically set $\nu=0.1$. The initial transport plan $\Pi^{(0)}$ is then obtained by applying the Sinkhorn algorithm with $K$ iterations to $S$. Across all datasets, we further employ the Barzilai–Borwein (BB) step-size rule~\cite{barzilai1988two} to adaptively determine the learning rate for the gradient descent updates.


\subsection{ACM-DBLP and Douban Datasets}\label{appendix:acm-dblp}

We introduce the construction of edges and features in the ACM-DBLP dataset as follows:
\begin{itemize}
    \item \emph{Features.} Features are constructed from authors and venues only, while paper titles are discarded. Author strings are lowercased, split on commas or semicolons, and tokenized by collapsing spaces into underscores. Venue names are tokenized into words, and merged phrase tokens are created. We use a pretrained RoBERTa model~\cite{liu2021robustly} to obtain embeddings of the corpus, and all representations are reduced to $d=256$ dimensions via PCA before whitening to zero mean and unit variance.
    \item \emph{Edges.} The graph is constructed by treating papers as nodes with edges defined by co-authorship and same-venue co-occurrence. We assign weights $\alpha_1=1.0$ for shared authors and $\alpha_2=0.5$ for shared venues, the weight edge $\beta_{ij}(G)$ is given by 
    \[
    \ti{\beta}_{ij}(G) = \alpha_1 \, C^{\text{author}}_{ij} + \alpha_2 \, C^{\text{venue}}_{ij}, \quad 
    \beta_{ij}(G) = \frac{\ti{\beta}_{ij}(G)-\expect{\ti{\beta}_{ij}(G)}}{\sqrt{\var\pth{\ti{\beta}_{ij}(G)}}},
    \]
    where $C^{\text{author}}_{ij}$ and $C^{\text{venue}}_{ij}$ denote co-occurrence counts.
\end{itemize}
We employ the same Douban dataset used in PARROT and other prior works, without any additional processing.
\paragraph{Baseline comparison.}
For completeness, we note two method-specific adjustments. First, REGAL assumes non-negative adjacency, which does not strictly match our setting; we therefore follow the standard workaround of omitting nodes with negative degree when running REGAL. Second, PARROT is designed as an anchor-based semi-supervised method, while our experiments do not assume anchor nodes; in this case, we adopt the ablated variant without anchors described in the original paper.

For ACM-DBLP and Douban datasets, we adopt a random initialization followed by a projection step. In practice, we initialize $\Pi^{(0)}$ with a random matrix (using fixed seeds to ensure reproducibility), and then project it onto the Birkhoff polytope using the Sinkhorn algorithm. This initialization avoids numerical instabilities that may arise from directly relying on feature or degree similarities in high-dimensional sparse settings.

\paragraph{Sensitivity analysis.} In Figure~\ref{fig:acm-dblp}, we report the experimental results over five random seeds for $\lambda \in \sth{0,0.2,0.4,0.6,0.8,1}$. The results indicate that the performance is stable with respect to both $\lambda$ and the initialization. We further provide a sensitivity analysis with respect to $\mu$ and $K$ in Table~\ref{tab:sensitivity}. The results again show that the performance remains stable across different choices of $\mu$ and $K$.

\begin{table}[htbp]
\centering
\small
\setlength{\tabcolsep}{5pt}
\renewcommand{\arraystretch}{1.15}
\begin{tabular}{c|ccccc|ccccc}
\hline
& \multicolumn{5}{c|}{\textbf{ACM-DBLP}} & \multicolumn{5}{c}{\textbf{Douban}} \\
\hline
$\mu$
& $0$ & $10^{-5}$ & $10^{-4}$ & $10^{-3}$ & $10^{-2}$
& $0$ & $10^{-5}$ & $10^{-4}$ & $10^{-3}$ & $10^{-2}$ \\
overlap
& $0.3230$ & $0.3219$ & $0.3205$ & $0.3166$ & $0.3249$
& $0.8220$ & $0.8216$ & $0.8229$ & $0.8188$ & $0.7919$ \\
\hline
$K$
& $20$ & $50$ & $100$ & $200$ & $400$
& $20$ & $50$ & $100$ & $200$ & $400$ \\
overlap
& $0.3253$ & $0.3237$ & $0.3246$ & $0.3249$ & $0.3251$
& $0.8220$ & $0.8220$ & $0.8220$ & $0.8220$ & $0.8220$ \\
\hline
\end{tabular}
\caption{Sensitivity analysis of $\mu$ and $K$ on ACM-DBLP and Douban.}
\label{tab:sensitivity}
\end{table}

\subsection{Spatial Transcriptomic Data}\label{apd:stdata}

Spatial transcriptomic (ST) data~\cite{staahl2016visualization} consists of gene expression profiles measured at spatially localized spots on a tissue slice, where each feature corresponds to a gene and the spatial coordinates represent the physical locations of the spots. We use an ST slice containing 255 spots with 7,998 gene features as the base dataset. To enable quantitative evaluation with ground-truth correspondences, we generate simulated slices by rotating the spot coordinates and resampling expression counts after adding a pseudocount $\delta$ to each gene in each spot. We consider five noise levels ($\delta \in {0,1,2,3,4,5}$) to model increasing experimental variability. 

In the experiment, we further examined the sensitivity of our method with respect to $\lambda$. Recall that $\lambda=0$ corresponds to using only the feature information, while $\lambda=1$ corresponds to relying solely on the structural information. Across the entire range of $\lambda$, our method consistently achieved strong alignment performance, indicating remarkable robustness to the choice of $\lambda$. In comparison with widely adopted baselines for spatial transcriptomics data alignment, including BBKNN~\cite{polanski2020bbknn} and Harmony~\cite{korsunsky2019fast}, our approach yields consistently superior results.
We implemented the experiments with parameters $T=400$, $K=80$, $\mu=0.1$, $\eta=10^{-5}$.

\begin{figure}[htbp]
    \centering
    \includegraphics[width=0.65\linewidth]{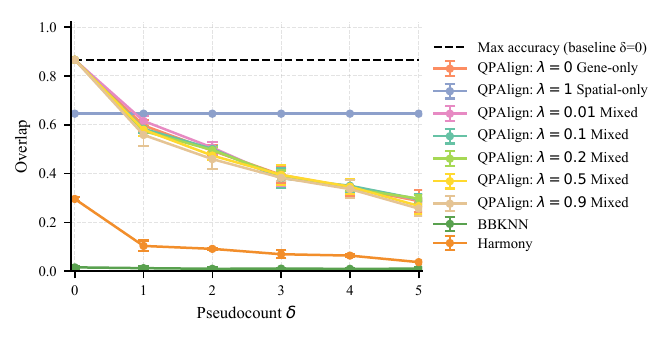}
    \caption{Alignment accuracy across different  $\lambda$ in ST data.}
    \label{fig:placeholder}
\end{figure}

In this following, we describe the experimental setup on simulated spatial transcriptomic data. 
Synthetic slices were generated by perturbing both spatial coordinates and transcript counts. 
Let $(X,Z)$ denote a transcript count matrix $X \in \mathbb{N}^{p \times n}$ and spot coordinates $Z \in \mathbb{R}^{2 \times n}$. 
To model sectioning variability, each coordinate was rotated by
\[
z'_i = R(\theta) z_i, \quad 
R(\theta) = \begin{bmatrix}\cos\theta & -\sin\theta \\ \sin\theta & \cos\theta\end{bmatrix},
\]
where $\theta$ was sampled uniformly from $[0,2\pi)$. 
Spots mapped outside the array were discarded to mimic tissue loss. Pairwise distances 
$d_{ik} = \lVert z_i - z_k \rVert_2$ were used to ensure invariance to global translation and rotation.  

Transcript counts were perturbed in two stages. First, spot-level UMI counts $N_i$ were sampled from a negative binomial distribution $N_i \sim \mathrm{NB}(r,p)$ with mean $\mu$ and variance $\sigma^2$, modeling over-dispersion. 
Second, given $N_i$, gene-level counts were drawn from a multinomial distribution with probabilities
\[
\pi_i = \frac{x_i + \delta}{\lVert x_i + \delta \rVert_1},
\]
where the pseudocount $\delta$ controls smoothing: small $\delta$ preserves heterogeneity, while large $\delta$ yields more uniform profiles. 
After rotation, duplicate grid positions were resolved by keeping one spot per location, preserving a one-to-one mapping.  
Besides, Z-score transformation was applied to both edge and feature information, and $L^2$ normalization was performed on the feature vectors to better align with our theoretical settings.

This perturbation procedure—combining rigid-body rotation, spot loss, and controlled count noise—produces slices that retain biological structure while reflecting realistic variability. 
It ensures that slices retain key biological structures while reflecting realistic experimental variability such as tissue dropout and technical noise.
For example, \cite{zeira2022alignment} used a similar framework to study alignment robustness under geometric perturbations.
We evaluated our method across different values of $\lambda$, where $\lambda=0$ and $\lambda=1$ correspond to feature-only and structure-only settings, respectively. 
The results show that our approach consistently balances edge and feature information, achieving robust performance superior to BBKNN~\cite{polanski2020bbknn} and Harmony~\cite{korsunsky2019fast}.

Finally, we introduce our intialization steps. We exploit domain-specific structure for initialization. Similar to synthetic datasets, we construct $S_{\text{feat}}$ and $S_{\deg}$, combine them with $\nu=0.1$, and apply the Sinkhorn algorithm with $K$ iterations to obtain $\Pi^{(0)}$.

\section{Proof of Theorems}
\subsection{Proof of Theorem~\ref{thm:main-partial}}

The impossibility result directly follows from Proposition~\ref{prop:impossible-partial}. We then show the possibility result.
By Proposition~\ref{prop:partial-possible},
\begin{align*}
        \prob{\sfd(\hat{\pi},\pi^*) = k}\le \exp\pth{-nh\pth{\frac{k}{n}}}\indc{k\le n-1}+\exp\pth{-2\log n}\indc{k=n}+\exp\pth{-\frac{\epsilon k\log n}{16}}.
    \end{align*}
Summing over $k\ge (1-\delta) n$, we have
\begin{align*}
    \sum_{k=(1-\delta)n}^{n-1} \exp\pth{-nh\pth{\frac{k}{n}}}&\le \sum_{k=1}^{n-1} \exp\left( -nh \left( \frac{k}{n} \right) \right)\\ 
    &\stackrel{\mathrm{(a)}}\leq 2 \sum_{1 \leq k \leq n/2} \exp\left( -k \log\frac nk \right)\\
    &\le 2\sum_{k=1}^{10\log n}\exp(-k\log \frac nk)+2\sum_{10\log n\le k\le n/2}2^{-k}\\
    &\le 2e^{-\log n}\cdot 10\log n+4\cdot 2^{-10\log n}= n^{-1 + o(1)},
\end{align*}
where $\mathrm{(a)}$ follows from \( h(x) = h(1 - x) \) and \( h(x) \geq x \log \frac{1}{x} \). Since \begin{align*}
    \sum_{k\ge (1-\delta )n}\exp\left( -\frac{1}{32} \epsilon k \log n \right) \leq \frac{\exp\left( -\frac\epsilon{32} (1-\delta) n \log n\right)}{1 - \exp \left( -\frac{\epsilon}{32} \log n \right)}  = n^{-\Omega(n)}, 
\end{align*}
we obtain that \begin{align}
    \nonumber &~\sum_{k=(1-\delta)n}^{n} \prob{\sfd(\hat{\pi},\pi^*) = k}\\\nonumber \le&~\sum_{k=(1-\delta)n}^{n}\qth{\exp\pth{-nh\pth{\frac{k}{n}}}\indc{k\le n-1}+\exp\pth{-2\log n}\indc{k=n}+\exp\pth{-\frac{\epsilon k\log n}{32}}}\\\label{eq:proof-thm1}\le&~n^{-1+o(1)}+n^{-2}+n^{-\Omega(n)}.
\end{align}
Since $\prob{\overlap(\hat\pi,\pi^*)\ge \delta}\ge 1- \sum_{k=(1-\delta)n}^{n} \prob{\sfd(\hat{\pi},\pi^*)=k}$, we finish the proof of Theorem~\ref{thm:main-partial}.

\subsection{Proof of Theorem~\ref{thm:exact-main}}

The impossibility results directly follows from Proposition~\ref{prop:impossible-exact}. When $n\log\pth{\frac{1}{1-\rho^2}}+d\log\pth{\frac{1}{1-r^2}}\ge (4+\epsilon)\log n$, we have \begin{align*}
    n\log\pth{\frac{1}{1-\rho^2}}+2d\log\pth{\frac{1}{1-r^2}}\ge (4+\epsilon)\log n,
\end{align*}
and thus~\eqref{eq:proof-thm1} holds for $\delta=1-\frac{\epsilon}{16}$. It remains to upper bound $\sum_{k=1}^{\epsilon n/16}\prob{ \sfd(\hat{\pi},\pi^*)=k}$. By Proposition~\ref{prop:exact-possible},\begin{align*}
    \sum_{k=1}^{{\epsilon n}/{16}} \prob{ \sfd(\hat{\pi},\pi^*)=k}\le \sum_{k=1}^{{\epsilon n}/{16}} \exp\pth{-\frac{\epsilon}{8}k\log n}\le \frac{\exp\pth{-\epsilon \log n/8}}{1-\exp\pth{-\epsilon \log n/8}}.
\end{align*}
Combining this with~\eqref{eq:proof-thm1}, we finish the proof of Theorem~\ref{thm:exact-main}.

\section{Proof of Propositions}

\subsection{Proof of Proposition~\ref{prop:partial-possible}}\label{apd:proof-partial-possible}

It is shown in~\cite{wu2022settling} and~\cite{dai2019database} that when $n\log(1/(1-\rho^2))\ge (4+\epsilon)\log n$ or $d\log(1/(1-r^2))\ge (4+\epsilon)\log n$ there exists an estimator $\hat{\pi}$ such that $\prob{\hat{\pi}=\pi^*} = 1-o(1)$. In this paper, we focus on the remaining regime, where $n\log(1/(1-\rho^2)) \vee d\log(1/(1-r^2))\le C_0\log n$ for some universal constant $C_0>4$, where $a\vee b = \max(a,b)$ for any $a,b\in\mathbb{R}$. Recall that we assume $d= \omega(\log n)$, which implies that $\rho,r=o(1)$. 

For any bijective mappings $\pi':V(G_1)\mapsto V(G_2)$, let $F_{\pi'}\triangleq\sth{v\in V(G_1):\pi^*(v) = \pi'(v)}$ be the set of fixed points. Recall that $\maT_k=\sth{\pi\in \maS_n:\sfd(\pi,\pi^*) 
= k}$ and \begin{align*}
    \sth{\sfd(\hat{\pi},\pi^*) = k}\subseteq \sth{\exists \pi'\in \maT_k:\score_{\pi^*}(G_1,G_2)\le \score_{\pi'}(G_1,G_2)},
\end{align*}
where $\score_{\pi}(G_1,G_2) = \varphi(\rho) \sum_{e\in E(G_1)} \beta_e(G_1) \beta_{\pi(e)}(G_2)+\varphi(r)\sum_{v\in V(G_1)} \bm{x}_v \bm{y}_{\pi(v)}$. Let $G[F]$ denote the induced subgraph of $G$ over a vertex set $F$. Then for any $\tau\in \mathbb{R}$, we have that \begin{align*}
    &~\sth{\sfd(\hat{\pi},\pi^*) = k}\\\subseteq&~ \sth{\exists \pi'\in \maT_k:\score_{\pi^*}(G_1,G_2)\le \score_{\pi'}(G_1,G_2)}\\=&~ \sth{\exists\pi'\in \maT_k:\qth{\score_{\pi^*}(G_1,G_2)-\score_{\pi^*}(G_1[F_{\pi'}],G_2[F_{\pi'}])}\le \qth{\score_{\pi'}(G_1,G_2)-\score_{\pi'}(G_1[F_{\pi'}],G_2[F_{\pi'}])}}\\\subseteq&~\sth{\exists\pi'\in \maT_k:{\score_{\pi^*}(G_1,G_2)-\score_{\pi^*}(G_1[F_{\pi'}],G_2[F_{\pi'}])}<\tau}\\&\cup\sth{\exists\pi'\in \maT_k:{\score_{\pi'}(G_1,G_2)-\score_{\pi'}(G_1[F_{\pi'}],G_2[F_{\pi'}])}\ge \tau}.
\end{align*}
We then bound the two events separately.

\subsubsection{Bad Event of Weak Signal}

We first upper bound $\prob{\exists\pi'\in \maT_k:{\score_{\pi^*}(G_1,G_2)-\score_{\pi^*}(G_1[F_{\pi'}],G_2[F_{\pi'}])}<\tau}$. We write $F=F_{\pi'}$ when $\pi'$ is given.
Let $N_k = \binom{n}{2}-\binom{n-k}{2}$.
Without loss of generality, we define $E(G_1)\backslash\binom{F}{2} = \sth{e_1,e_2,\cdots, e_{N_k}}$ and $V(G_1)\backslash F = \sth{v_1,v_2,\cdots,v_{k}}$. Let $X,Y\in \mathbb{R}^{(N_k+dk)\times 1}$ be defined as  \begin{align*}
    X&\triangleq \pth{\beta_{e_1}(G_1),\beta_{e_2}(G_1),\cdots,\beta_{e_{N_k}}(G_1),\bm{x}_{v_1}^\top,\bm{x}_{v_2}^\top,\cdots,\bm{x}_{v_k}^\top}^\top,\\
    Y&\triangleq \pth{\beta_{\pi^*(e_1)}(G_2),\beta_{\pi^*(e_2)}(G_2),\cdots,\beta_{\pi^*(e_{N_k})}(G_2),\bm{y}_{\pi^*(v_1)}^\top,\bm{y}_{\pi^*(v_2)}^\top,\cdots,\bm{y}_{\pi^*(v_k)}^\top}^\top.
\end{align*}
Let $W = \score_{\pi^*}(G_1,G_2)-\score_{\pi^*}(G_1[F],G_2[F])$. Then, we have that $W=X^\top AY$, where $A = \mathrm{diag}\sth{\varphi(\rho) I_{N_k},\varphi(r)I_{dk}}$ with $\Vert A\Vert_F^2 = \varphi(\rho)^2 N_k+\varphi(r)^2 dk$ and $\Vert A\Vert_2 = \varphi(\rho)\vee \varphi(r)$. The following Lemma provides concentration for $W$.

\begin{lemma}\label{lem:hwforweaksignal}
    There exists a universal constant $C$, such that with probability at least $1-\delta_0$,
    $$|W-(\rho\varphi(\rho)N_k+r\varphi(r)kd)|\le  C\pth{\Vert A\|_F\sqrt{\log\frac{1}{\delta_0}}+\|A\|_2\log \frac{1}{\delta_0}}.$$
\end{lemma}

Pick $\tau=(\rho\varphi(\rho)N_k+r\varphi(r)kd)-a_k$, where 
\begin{align}\label{eq:choice_of_tau}
    a_k=\begin{cases}
    3C\sqrt{(\rho^2N_k+r^2kd)2nh(\frac{k}n)},&k\le n-1,\\
    3C\sqrt{(n\rho^2+dr^2)n\log n},&k=n,
\end{cases}
\end{align}
and  $h(x)=-x\log x-(1-x)\log(1-x)$ is the binary entropy function.

\paragraph{Case 1: $k=n$.} We choose $\delta_0 = \exp\pth{-2\log n}$ in Lemma~\ref{lem:hwforweaksignal}. Recall that $\rho,r = o(1)$.
Then, with probability at least $1-\delta_0$,\begin{align*}
    \vert W-(\rho \varphi(\rho) N_k+r\varphi(r)dk)\vert 
    &\le C\qth{\sqrt{\varphi(\rho)^2 N_k+\varphi(r)^2dk}\sqrt{2\log n}+\pth{(\varphi(\rho)\vee\varphi(r))2\log n}}\\
    &\le C\sqrt{4n\log n(dr^2+n\rho^2)}+C\sqrt{n\log n(n\rho^2+dr^2)}\\&=3C\sqrt{(n\rho^2+dr^2)n\log n} =a_n,
\end{align*}
where the last inequality follows from 
$$(\varphi(\rho)\vee \varphi(r))2\log n\le \sqrt{n\rho^2+dr^2}\log n\le C\sqrt{n\log n(n\rho^2+dr^2)}.$$
Consequently, we have $\prob{W\le \tau}\le \exp\pth{-2\log n}$.

\paragraph{Case 2: $\delta n\le k\le n-1$.} We choose $\delta_0 = \exp\pth{-2nh\pth{\frac{k}{n}}}$ in Lemma~\ref{lem:hwforweaksignal}.
Then, with probability $1-\delta_0$, we have \begin{align*}
    &~|W-(\rho\varphi(\rho)N_k+r\varphi(r)kd)|\\\le&~  C\left[\left(\sqrt{\varphi(\rho)^2N_k+\varphi(r)^2dk}\sqrt{2nh(\frac kn)}\right)+ \left((\varphi(\rho)\vee \varphi(r))2nh(\frac kn)\right)\right]\\
    \le&~  C\left(\Big(\sqrt{4(\rho^2N_k+r^2dk)}\sqrt{2nh(\frac kn)}\Big)+\Big((\varphi(\rho)\vee \varphi(r))2nh(\frac kn)\Big) \right)\\
    \le&~ 3C\sqrt{(\rho^2N_k+r^2dk)2nh(\frac{k}n)},
\end{align*}
where the last inequality follows from 
\begin{align*}
    (\varphi(\rho)\vee \varphi(r))2nh(\frac kn)&\le \sqrt{(\rho^2N_k+r^2kd)2nh(\frac kn)}\sqrt{\frac{4n(\rho^2+r^2)}{(\rho^2N_k+r^2kd)}}\\
    &\le\sqrt{(\rho^2N_k+r^2kd)2nh(\frac kn)}\sqrt{\frac{8n(\rho^2+r^2)}{k(\rho^2n+r^2d)}}, 
\end{align*}
and $$\frac{8n(\rho^2+r^2)}{k(\rho^2n+r^2d)}\le \frac{16n}{\delta n(\rho^2n +r^2 d)}\le \frac{16n}{\delta n(4+\epsilon/2)\log n}\le  1,$$
where the last inequality is because $\rho,r=o(1)$ and $n\log(\frac{1}{1-\rho^2})+d\log(\frac{1}{1-r^2})\ge (4+\epsilon)\log n$ implies $n\rho^2+dr^2\ge (4+\epsilon/2)\log n$.
Consequently, $\prob{W\le \tau}\le \exp\pth{-2nh\pth{\frac{k}{n}}}$ when $k\le n-1$. 
By the union bound, we obtain \begin{align}
    \nonumber &~\prob{\exists\pi'\in \maT_k:{\score_{\pi^*}(G_1,G_2)-\score_{\pi^*}(G_1[F_{\pi'}],G_2[F_{\pi'}])}<\tau}\\
    \nonumber\le&~\prob{\bigcup_{F\subseteq V(G_1):|F| = n-k}\sth{ 
    {\score_{\pi^*}(G_1,G_2)-\score_{\pi^*}(G_1[F],G_2[F])}<\tau
    }}\\
    \nonumber\le&~\binom{n}{k}\prob{ 
    {\score_{\pi^*}(G_1,G_2)-\score_{\pi^*}(G_1[F],G_2[F])}<\tau
    }\\
    \nonumber\le&~\binom{n}{k}\prob{W\le \tau}\indc{k\le n-1}+\prob{W\le \tau}\indc{k=n}\\\label{eq:bad-event-signal-final}\le&~ \exp\pth{-nh\pth{\frac{k}{n}}}\indc{k\le n-1}+\exp\pth{-2\log n}\indc{k=n},
\end{align}
where the last inequality is because $\binom{n}{k}\le \exp\pth{nh\pth{\frac{k}{n}}}$.

\subsubsection{Bad Event of Strong Noise}
We then upper bound $\prob{\exists\pi'\in \maT_k:{\score_{\pi'}(G_1,G_2)-\score_{\pi'}(G_1[F_{\pi'}],G_2[F_{\pi'}])}\ge\tau}$.
Given $\pi'$, we define $Z\triangleq \score_{\pi'}(G_1,G_2)-\score_{\pi'}(G_1[F_{\pi'}],G_2[F_{\pi'}])$. 
We also write $F=F_{\pi'}$ when $\pi'$ is given.
By Chernoff's inequality, for any $t>0$,\begin{align*}
    \prob{Z\ge \tau}\le e^{-t\tau}\expect{e^{tZ}}.
\end{align*}
In order to compute the moment generating function $\expect{e^{tZ}}$, we introduce the definition of orbits.

\paragraph{Cycle decomposition} For any $\sigma\in \maS_n$, it induces a permutation $\sigma^\sfE$ on the edge set $\binom{V(G_1)}{2}$ with $\sigma^\sfE((u,v))\triangleq (\sigma(u),\sigma(v))$ for $u,v\in V(G_1)$. We refer to $\sigma$ and $\sigma^\sfE$ as a node permutation and edge permutation. Each permutation can be decomposed as disjoint cycles known as orbits. Orbits of $\sigma$ (resp. $\sigma^\sfE$) are referred as \emph{node orbits} (resp. \emph{edge orbits}). For example, a node orbit $(u_1,u_2,\cdots, u_k)$ indicates that $u_{i+1} = \sigma(u_i)$ for $1\le i\le k-1$ and $u_1 = \sigma(u_k)$. Let $n_k$ (resp. $N_k$) denote the number of $k$-node (resp. $k$-edge) orbits in $\sigma$ (resp. $\sigma^\sfE$).

For any $\pi'\in \maT_k$,
let $\sigma\triangleq(\pi^*)^{-1}\circ\pi'$. Define $\maC^\sfV_i$ and $\maC^\sfE_i$ the set of \emph{node orbits} and \emph{edge orbits} of length $i$ induced by $\sigma$, respectively. Denote $\maC^\sfV = \cup_{i\ge 1}\maC^\sfV_i$ and $\maC^\sfE = \cup_{i\ge 1}\maC^\sfE_i$.
Then, $V(G_1)=\cup_{i\ge 1} \sth{v:v\in \maC_i^\sfV}, E(G_1) = \cup_{i\ge 1}\sth{e:e\in \maC_i^\sfE}$, and $\maC^\sfV_1= F$.
Let \begin{align}\label{eq:ZEZV}
    Z^\sfE = \varphi (\rho)\sum_{e\in E(G_1)\backslash \binom{F}{2}} \beta_e(G_1) \beta_{\pi'(e)}(G_2),\quad Z^\sfV = \varphi(r)\sum_{v\in V(G_1)\backslash F}\bm{x}_v\bm{y}_{\pi'(v)}.
\end{align}
Then $Z = Z^\sfV+Z^\sfE$. Since $Z^\sfV$ and $Z^\sfE$ are independent, we obtain that \begin{align}\label{eq:mgf-independent}
    \expect{e^{tZ}} = \expect{e^{tZ^\sfV}} \expect{e^{tZ^\sfE}}.
\end{align}
We then derive the upper bounds for $\expect{e^{tZ^\sfV}}$ and $\expect{e^{tZ^\sfE}}$, respectively. For any edge cycle $C=\sth{e_1,e_2,\cdots,e_{|C|}}$ with $e_{i+1} = \sigma^\sfE(e_i)$ for all $1\le i\le |C|-1$ and $e_1 = \sigma^\sfE(e_{|C|})$, we define the cumulant generating function as \begin{align*}
    \kappa^\sfE_{|C|}(t) = \log\expect{\exp\pth{t\varphi(\rho)\sum_{i=1}^{|C|} \beta_{e_i}(G_1)\beta_{\pi'(e_i)}(G_2)}},
\end{align*}
where we define $(u_{|C|+1},v_{|C|+1}) = (u_1,v_1)$. Similarly, for any node cycle $C=\sth{v_1,\cdots,v_{|C|}}$, we define $v_{|C|+1} = v_1$ and \begin{align*}
    \kappa^\sfV_{|C|}(t) = \log\expect{\exp\pth{t\varphi(r)\sum_{i=1}^{|C|}\bm{x}_{v_i}\bm{y}_{\pi'(v_i)}}}.
\end{align*}

The lower-order cumulants can be calculated directly:\begin{align*}
    \kappa^\sfE_1(t) &= -\frac12\log(1-2t\rho\varphi(\rho)-t^2\varphi^2(\rho)(1-\rho^2)),\\ \kappa^\sfV_1(t) &= -\frac d2\log(1-2tr\varphi(r)-t^2\varphi^2(r)(1-r^2)),\\\kappa_2^\sfE(t) &=-\frac12\log(1-2t^2\varphi^2(\rho)(1+\rho^2)+t^4\varphi^4(\rho)(1-\rho^2)^2) ,\\\kappa_2^\sfV(t)&=-\frac d2\log(1-2t^2\varphi^2(r)(1+r^2)+t^4\varphi^4(r)(1-r^2)^2) .
\end{align*}
Let $N_k = \binom{n}{2}-\binom{n-k}{2}$.
The following Lemma provides an upper bound on the cumulant function $\log\expect{\exp(tZ)}$.
\begin{lemma}\label{lem:cumulant}
    If $\sfd(\pi^*,\pi')=k$, for any $0<t\le (\rho^{-1}-2)\wedge (r^{-1}-2)$, we have \begin{align*}
        \log\expect{\exp(tZ)}\le \frac{N_k}{2} \kappa_2^\sfE(t)+ \frac{k}{2}\pth{\kappa_1^\sfE(t)-\frac{1}{2}\kappa_2^\sfE(t)+\frac{1}{2}\kappa_2^\sfV(t)}.
    \end{align*}
\end{lemma}

Recall that $\tau=(\rho\varphi(\rho)N_k+r\varphi(r)kd)-a_k$, where 
\begin{align}
    a_k=\begin{cases}
    3C\sqrt{(\rho^2N_k+r^2kd)2nh(\frac{k}n)},&k\le n-1\\
    3C\sqrt{(n\rho^2+dr^2)n\log n},&k=n
\end{cases}
\end{align}
and  $h(x)=-x\log x-(1-x)\log(1-x)$ is the binary entropy function. We then show that  $$\pth{1-\frac{\epsilon}{16}}(\rho\varphi(\rho)N_k+r\varphi(r)kd)\le \tau\le \rho\varphi(\rho)N_k+r\varphi(r)kd.$$ 
Recall that $\rho,r=o(1)$. 
For any $\delta n\le k\le n-1$, since $n\log\pth{\frac{1}{1-\rho^2}}+2d\log\pth{\frac{1}{1-r^2}}\ge (4+\epsilon)\log n$, we have $n\rho^2 +d r^2\ge (2+\epsilon/4)\log n$. 
Therefore, we obtain $\frac{n}{k}h\pth{\frac{k}{n}}\le \frac{h(\delta)}{\delta}\le \frac{\epsilon^2}{2^{15}C^2} (n\rho^2+dr^2)$ for sufficiently large $n$.
When $k\le n-1$, we have \begin{align*}  a_k&=3C\sqrt{(\rho^2N_k+r^2kd)2nh(\frac{k}n)}\le 4C\sqrt{(\rho^2N_k+r^2kd)\frac{\epsilon^2}{2^{14}C^2}(n\rho^2+dr^2)k}\\
    &\overset{\mathrm{(a)}}{\le} \frac{\epsilon}{32}\sqrt{(\rho^2N_k+r^2kd)4(\rho^2N_k+r^2kd)}= \frac{\epsilon}{16}(\rho^2N_k+r^2kd)\le \frac{\epsilon}{16}(\rho\varphi(\rho)N_k+r\varphi(r)kd),
\end{align*}
where $\mathrm{(a)}$ is because $nk\le 4N_k$. For $k=n$, since $n\log(\frac1{1-\rho^2})+2d\log(\frac1{1-r^2})\ge (4+\epsilon)\log n$ and $\rho^2,r^2=o(1)$, we conclude that $n\rho^2+dr^2\ge \frac{1}{2}(n\rho^2+2dr^2)\ge  2\log n$ holds for sufficiently large $n$. Therefore, 
\begin{align*}
    a_k&=3C\sqrt{(n\rho^2+dr^2)n\log n}\le 3C\sqrt{(n\rho^2+dr^2)\frac{(n\rho^2+dr^2)n}{2}}\\
    &\le \frac{4C}{\sqrt n}(\rho\varphi(\rho)n^2+r\varphi(r)nd)\le \frac{\epsilon}{16}(\rho\varphi(\rho)N_k+r\varphi(r)kd).
\end{align*}
Therefore, we conclude
 $$\pth{1-\frac{\epsilon}{16}}(\rho\varphi(\rho)N_k+r\varphi(r)kd)\le \tau\le \rho\varphi(\rho)N_k+r\varphi(r)kd.$$

We then upper bound $\prob{Z\ge \tau}$. We note that \begin{align*}
\prob{Z\ge \tau}&\le e^{-t\tau}\expect{e^{tZ}}\le \exp\pth{-t\tau+\frac{N_k}{2} \kappa_2^\sfE(t)+ \frac{k}{2}\pth{\kappa_1^\sfE(t)-\frac{1}{2}\kappa_2^\sfE(t)+\frac{1}{2}\kappa_2^\sfV(t)}}.
\end{align*}
Pick $t=1$. Since $\rho,r=o(1)$, we have $\rho^2,r^2\le \frac{\epsilon}{256}$. Recall that $\varphi(\rho) = \frac{\rho}{1-\rho^2}$. Then,
\begin{align*}
    \kappa_1^\sfE(t)-\frac12 \kappa_2^\sfE(t)&=\frac14\log\frac{1-2t^2\varphi^2(\rho)(1+\rho^2)+t^4\varphi^4(\rho)(1-\rho^2)^2}{(1-2t\rho\varphi(\rho)-t^2\varphi^2(\rho)(1-\rho^2))^2}
    \\
    &=\frac14\log\left(1+\frac{4t\rho^2}{1-\rho^2(1+t)^2}\right)\le \frac{\rho^2}{1-4\rho^2}\le 2,
\end{align*}
where the last inequality is because $\log(1+x)\le x$ and $\rho<\frac14$. We then bound $\kappa_2^\sfE(t)$. We note that 
\begin{align*}
    \frac{\kappa_2^\sfE(t)}{\rho\varphi(\rho)}&=-\frac{1-\rho^2}{2\rho^2}\log(1-2t^2\varphi^2(\rho)(1+\rho^2)+t^4\varphi^4(\rho)(1-\rho^2)^2)\\    &=-\frac{1-\rho^2}{2\rho^2}\log\left(1-\frac{\rho^2(2+\rho^2)}{(1-\rho^2)^2}\right)\overset{\mathrm{(a)}}{\le} 1+4\rho^2\le \pth{1+\frac{\epsilon}{64}},
\end{align*}
where the inequality $\mathrm{(a)}$ is from Lemma~\ref{lem:kappa2bound}. Hence, we have $\kappa_2^\sfE(t)\le \pth{1+\frac{\epsilon}{64}}\rho\varphi(\rho)$.
Similarly, we have 
\begin{align*}
 \kappa_2^\sfV(t)&= -\frac d2\log(1-2t^2\varphi^2(r)(1+r^2)+t^4\varphi^4(r)(1-r^2)^2)\le \pth{1+\frac{\epsilon}{64}}dr\varphi(r).
\end{align*}
Therefore, for $t=1$,
\begin{align*}
    &\quad -t\tau+\frac{N_k}{2}\kappa_2^\sfE(t)+\frac k2\pth{\kappa_1^\sfE(t)-\frac12\kappa_2^\sfE(t)}+\frac{k}{2} \kappa_2^\sfV(t) \\
    &\le -\pth{1-\frac\epsilon{16}}(\rho\varphi(\rho)N_k+r\varphi(r)kd)+\pth{\frac12+\frac{\epsilon}{128}}(\rho\varphi(\rho)N_k+kdr\varphi(r))+k\\
    &\le-\pth{\frac12-\frac\epsilon{32}}(\rho\varphi(\rho)N_k+r\varphi(r)kd)+k\\
    &\le -\pth{\frac12-\frac\epsilon{32}}\pth{2+\frac\epsilon4}k\log n+k,
\end{align*}
where the last inequality is because $N_k=\binom{n}{2}-\binom{n-k}{2}\ge \frac{1}{2}\pth{1-\frac{1}{n}}kn$ and
$$\rho\varphi(\rho)N_k+r\varphi(r)kd\ge k\pth{\frac{n-1}{2}\log(\frac1{1-\rho^2})+d\log(\frac1{1-r^2})}\ge k(2+\frac\epsilon4)\log n.$$
Consequently, by the union bound and $|\maT_k| = \binom{n}{k}k!\le n^k$,
\begin{align*}
    &~\prob{\exists\pi'\in \maT_k:{\score_{\pi'}(G_1,G_2)-\score_{\pi'}(G_1[F_{\pi'}],G_2[F_{\pi'}])}\ge\tau}\\
    \le&~n^k\prob{Z\ge \tau}\\\le&~\exp\pth{k\log n-\pth{\frac12-\frac\epsilon{32}}\pth{2+\frac\epsilon4}k\log n+k}\le \exp\pth{-\frac{\epsilon}{32}k\log n}.
\end{align*}
Combining this with~\eqref{eq:bad-event-signal-final}, we obtain that \begin{align*}
    \prob{\sfd(\hat{\pi},\pi^*) = k}&\le\prob{\exists\pi'\in \maT_k:{\score_{\pi^*}(G_1,G_2)-\score_{\pi^*}(G_1[F_{\pi'}],G_2[F_{\pi'}])}<\tau}\\&~~~~+\prob{\exists\pi'\in \maT_k:{\score_{\pi'}(G_1,G_2)-\score_{\pi'}(G_1[F_{\pi'}],G_2[F_{\pi'}])}\ge\tau}\\&\le \exp\pth{-nh\pth{\frac{k}{n}}}\indc{k\le n-1}+\exp\pth{-2\log n}\indc{k=n}+\exp\pth{-\frac{\epsilon k\log n}{32}}.
\end{align*}

\subsection{Proof of Proposition~\ref{prop:exact-possible}}\label{apd:proof-exact-possible}

For any $\pi'\in \maT_k$, let \begin{align*}
    Y_{\pi'} &\triangleq \varphi(\rho) \sum_{e\in E(G_1)\backslash\maC_1^\sfE} \pth{\beta_e(G_1)\beta_{\pi'(e)}(G_2)-\beta_e(G_1)\beta_{\pi^*(e)}(G_2)}\\
    &~~~~+\varphi(r)\sum_{v\in V(G_1)\backslash \maC_1^\sfV} \pth{\bm{x}_v^\top \bm{y}_{\pi'(v)}-\bm{x}_v^\top \bm{y}_{\pi^*(v)}},
\end{align*}
where $\maC_1^\sfE$ and $\maC_1^\sfV$ denote the edge orbit and vertex orbit of length 1 induced by $\sigma = (\pi^*)^{-1}\circ \pi'$. For notational simplicity, we write $Y=Y_{\pi'}$ when $\pi'$ is given. Then for any $t>0$, $\sth{\sfd(\hat{\pi},\pi^*) = k}\subseteq \sth{\exists \pi'\in \maT_k: Y\ge 0}$ and \begin{align}\label{eq:chernoff-smallk}
    \prob{\sfd(\hat{\pi},\pi^*) = k}\le \prob{\exists \pi'\in \maT_k: Y \ge 0}\overset{\mathrm{(a)}}{\le} |\maT_k|\prob{Y \ge 0}\overset{\mathrm{(b)}}{\le} n^k \expect{\exp\pth{tY }},
\end{align}
where $\mathrm{(a)}$ uses union bound and $\mathrm{(b)}$ follows from Chernoff's inequality and $|\maT_k| = \binom{n}{k}k!\le n^k$. Let \begin{align*}
    Y^\sfE&\triangleq\varphi(\rho) \sum_{e\in E(G_1)\backslash\maC_1^\sfE} \pth{\beta_e(G_1)\beta_{\pi'(e)}(G_2)-\beta_e(G_1)\beta_{\pi^*(e)}(G_2)},\\
    Y^\sfV&\triangleq\varphi(r)\sum_{v\in V(G_1)\backslash \maC_1^\sfV} \pth{\bm{x}_v^\top \bm{y}_{\pi'(v)}-\bm{x}_v^\top \bm{y}_{\pi^*(v)}}.
\end{align*}
Then $Y = Y^\sfE+Y^\sfV$ and $\expect{\exp(tY)} = \expect{\exp(tY^\sfE)}\expect{\exp\pth{tY^\sfV}}$. We then derive the upper bounds for $\expect{\exp\pth{tY^\sfV}}$ and $\expect{\exp\pth{tY^\sfE}}$, respectively. For any edge cycle $C=\sth{e_1,e_2,\cdots,e_{|C|}}$ with $e_{i+1} = \sigma^\sfE(e_i)$ for all $1\le i\le |C|-1$ and $e_1 = \sigma^\sfE(e_{|C|})$, we define the cumulant generating function as \begin{align*}
    \mu^\sfE_{|C|}(t) = \log\expect{\exp\pth{t\varphi(\rho)\sum_{i=1}^{|C|} \beta_{e_i}(G_1)\beta_{\pi'(e_i)}(G_2)-t\varphi(\rho)\sum_{i=1}^{|C|} \beta_{e_i}(G_1)\beta_{\pi^*(e_i)}(G_2)}},
\end{align*}
where we define $(u_{|C|+1},v_{|C|+1}) = (u_1,v_1)$. Similarly, for any node cycle $C=\sth{v_1,\cdots,v_{|C|}}$, we define $v_{|C|+1} = v_1$ and \begin{align*}
    \mu^\sfV_{|C|}(t) = \log\expect{\exp\pth{t\varphi(r)\sum_{i=1}^{|C|}\bm{x}_{v_i}\bm{y}_{\pi'(v_i)}-t\varphi(r)\sum_{i=1}^{|C|}\bm{x}_{v_i}\bm{y}_{\pi^*(v_i)}}}.
\end{align*}
The lower-order cumulants can be calculated directly:\begin{align*}
    \mu_2^\sfE(t)&=-\frac{1}{2}\log\pth{1+\frac{\rho^2}{1-\rho^2}(4t-4t^2)},\quad \mu_2^\sfV(t)=-\frac{d}{2}\log\pth{1+\frac{r^2}{1-r^2}(4t-4t^2)}.
\end{align*}
Recall that $N_k = \binom{n}{2}-\binom{n-k}{2}$. The following Lemma provides an upper bound on the cumulant function $\log\expect{\exp\pth{tY}}$.
\begin{lemma}\label{lem:cumulant-small}
    If $\sfd(\pi^*,\pi') = k$, for any $0<t<1$, we have \begin{align*}
        \log\expect{\exp\pth{tY}}\le \frac{1}{2}\pth{N_k-\frac{k}{2}}\mu_2^\sfE(t)+\frac{k}{2}\mu_2^\sfV(t).
    \end{align*}
\end{lemma}
Pick $t=\frac{1}{2}$. By Lemma~\ref{lem:cumulant-small}, \begin{align*}
    \log\expect{\exp\pth{tY}}\le -\frac{nk}{4}\pth{1-\frac{k+2}{2n}}\log\pth{\frac{1}{1-\rho^2}}-\frac{kd}{4}\log\pth{\frac{1}{1-r^2}}.
\end{align*}
Combining this with~\eqref{eq:chernoff-smallk}, we have \begin{align*}
    \prob{\sfd(\hat{\pi},\pi^*) = k}&\le n^k \expect{\exp\pth{tY}}\\
    &\le\exp\pth{k\log n-\frac{nk}{4}\pth{1-\frac{k+2}{2n}}\log\pth{\frac{1}{1-\rho^2}}-\frac{kd}{4}\log\pth{\frac{1}{1-r^2}}}\\&\le \exp\pth{k\log n-\frac{k}{4}\pth{1-\frac{k+2}{2n}}\pth{n\log\pth{\frac{1}{1-\rho^2}}+d\log\pth{\frac{1}{1-r^2}}}}\\&\overset{\mathrm{(a)}}{\le} \exp\pth{-\pth{\frac{\epsilon}{4}-{\frac{k+2}{2n}}\pth{1+\frac{\epsilon}{4}}}k\log n} \overset{\mathrm{(b)}}{\le} \exp\pth{-\frac{\epsilon}{8}k\log n},
\end{align*} 
where $\mathrm{(a)}$ is because $n\log\pth{\frac{1}{1-\rho^2}}+d\log\pth{\frac{1}{1-r^2}}\ge (4+\epsilon)\log n$;
$\mathrm{(b)}$ follows from $k\le \frac{\epsilon}{16}n$.

\subsection{Proof of Proposition~\ref{prop:impossible-partial}}\label{apd:proof-partial-impossible}

We first introduce the following lemma.
\begin{lemma}\label{lem:mutual-pack}
    For any $0<\delta\le 1$, we have \begin{align}\label{eq:pack-and-mutual}
        &|\maM_\delta|\ge \pth{\frac{\delta n}{e}}^{\delta n},\\
        &I(\pi^*;G_1,G_2)\le \frac{1}{2}\binom{n}{2}\log\Big(\frac{1}{1-\rho^2}\Big)+\frac{nd}{2}\log\Big(\frac{1}{1-r^2}\Big).
    \end{align}
\end{lemma}
The proof of Lemma~\ref{lem:mutual-pack} is deferred to Appendix~\ref{apd:proof-mutual-pack}.
We directly apply Fano's inequality in~\eqref{eq:Fano}. For any $0<\delta<1$, by Lemma~\ref{lem:mutual-pack},
\begin{align*}
    \prob{\overlap(\hat{\pi},\pi^*)<\delta}&\ge 1-\frac{I(\pi^*;G_1,G_2)+\log 2}{\log |\maM_\delta|}\\&\ge 1-\frac{\binom n2 \frac12\log(\frac1{1-\rho^2})+\frac{nd}{2}\log(\frac1{1-r^2})+\log 2}{\delta n\log(\delta n/e)}\ge 1-\frac{c}{4\delta},
\end{align*}
where the last inequality follows from $n\log\pth{\frac{1}{1-\rho^2}}+2d\log\pth{\frac{1}{1-r^2}}\le c\log n$.

\subsection{Proof of Proposition~\ref{prop:impossible-exact}}\label{apd:proof-exact-impossible}
In this subsection, we provide the proof on Proposition~\ref{prop:impossible-exact}.
Recall that $$\score_{\pi}(G_1,G_2) = \varphi(\rho)\sum_{e\in E(G_1)} \beta_e(G_1)\beta_{\pi(e)}(G_2)+\varphi(r)\sum_{v\in V(G_1)}\bm{x}_v\bm{y}_{\pi(v)}.$$ Define \begin{align*}
    \maE(\pi^*,\pi') &\triangleq \sth{(G_1,G_2):\score_{\pi*}(G_1,G_2)\le \score_{\pi'}(G_1,G_2)},\\ \maM &\triangleq\sth{\pi'\in \maS_n:\pi'\neq \pi^*,(G_1,G_2)\in \maE(\pi^*,\pi')}.
\end{align*}
Since the true permutation $\pi^*$ is uniformly distributed, the MLE $\hat\pi_{\mathrm{ML}}$ minimizes the error probability among all estimators. Therefore, to prove the impossibility result, it suffices to prove the failure of MLE. We note that $\hat{\pi}$ in~\eqref{eq:estimator} achieves exact recovery is equivalent to $\mathcal M=\emptyset$. To prove the impossibility of exact recovery, it suffices to show $\prob{|\mathcal M|=0}=o(1)$.

Let $I = |\maM\cap \maT_2|$ with $\maT_2 = \sth{\pi'\in \maS_n:\sfd(\pi^*,\pi') = 2}$. Then $I\le |\maM|$. By Chebyshev's inequality, we have
\begin{align}\label{eq:prob-M-0}
    \prob{|\maM| = 0}\le \prob{I=0}\le \prob{(I-\expect{I})^2\ge (\expect{I})^2}\le \frac{\var\qth{I}}{(\expect{I})^2}.
\end{align}
Given $\pi'$, let $\epsilon_1\triangleq \prob{(G_1,G_2)\in \maE(\pi^*,\pi')}$. Since $|\maT_2| = \binom{n}{2}$, the expectation $\expect{I}$ is then given by \begin{align*}
    \expect{I} = \sum_{\pi'\in \maT_2}\prob{(G_1,G_2)\in \maE(\pi^*,\pi')} = \binom{n}{2}\epsilon_1.
\end{align*}

We then compute the second moment $\expect{I^2}$. Note that \begin{align}
    \nonumber I^2 &= \pth{\sum_{\pi'\in\maT_2} \indc{(G_1,G_2)\in \maE(\pi^*,\pi')}}^2 \\\label{eq:I-second-moment}&= \sum_{\pi'\in \maT_2} \indc{(G_1,G_2)\in \maE(\pi^*,\pi')}+\sum_{\pi_1,\pi_2\in \maT_2:\pi_1\neq \pi_2}\indc{(G_1,G_2)\in \maE(\pi^*,\pi_1)}\indc{(G_1,G_2)\in \maE(\pi^*,\pi_2)}.
\end{align}
It remains to compute $\sum_{\pi_1\neq \pi_2\in \maT_2}\prob{(G_1,G_2)\in \maE(\pi^*,\pi_1)\cap  \maE(\pi^*,\pi_2)}$. Indeed, we have $\sfd(\pi_1,\pi_2) \in\sth{3,4}$ for any $\pi_1\neq \pi_2 \in \maT_2$. The number of pairs $(\pi_1,\pi_2)$ with $\pi_1\neq \pi_2\in \maT_2$ with $\sfd(\pi_1,\pi_2) = 3$ and $\sfd(\pi_1,\pi_2) = 4$ are $6\binom{n}{3}$ and $6\binom{n}{4}$, respectively. 
For $\pi_1\neq \pi_2\in \maT_2$ with $\sfd(\pi_1,\pi_2) = 4$, since $\score_{\pi^*}(G_1,G_2)-\score_{\pi_1}(G_1,G_2)$ and $\score_{\pi^*}(G_1,G_2)-\score_{\pi_2}(G_1,G_2)$ are independent, we have \begin{align*}
    \prob{(G_1,G_2)\in \maE(\pi^*,\pi_1)\cap  \maE(\pi^*,\pi_2)} = \prob{(G_1,G_2)\in \maE(\pi^*,\pi_1)}\prob{(G_1,G_2)\in \maE(\pi^*,\pi_2)} = \epsilon_1^2.
\end{align*}

For $\pi_1\neq \pi_2\in \maT_2$ with $\sfd(\pi_1,\pi_2) = 3$, we have the following Lemma.
\begin{lemma}\label{lem:d=3}
    For any $\pi_1\neq \pi_2\in \maT_2$ with $\sfd(\pi_1,\pi_2) = 3$, we have \begin{align*}
        \prob{(G_1,G_2)\in \maE(\pi^*,\pi_1)\cap \maE(\pi^*,\pi_2)}\le (1-\rho^2)^{\frac{3(n-2)}{4}}(1-r^2)^{\frac{3d}{4}}.
    \end{align*}
\end{lemma}

Next we prove a lower bound of $\epsilon_1$. For any $\pi'\in \maT_2$, we assume that $\pi^*(v)=\pi'(v)$ for any $v\in V(G_1)\backslash \sth{v_1,v_2}$, $\pi^*(v_1)=\pi'(v_2)$, and $\pi^*(v_2) = \pi'(v_1)$.
Consequently,\begin{align*}
    \epsilon_1 &= \prob{(G_1,G_2)\in \maE(\pi^*,\pi')}\\
    &=\mathbb{P}\Bigg[\varphi(\rho)\sum_{e\in E(G_1)} \beta_e(G_1)\pth{\beta_{\pi'(e)}(G_2)-\beta_{\pi^*(e)}(G_2)}\\&~~~~~~~+\varphi(r)\sum_{v\in V(G_1)}\bm{x}_v^\top \pth{\bm{y}_{\pi'(v)}-\bm{y}_{\pi^*(v)}}\ge 0\Bigg]\\&=\mathbb{P}\Bigg[\varphi(\rho)\sum_{v\in V(G_1)\backslash\sth{v_1,v_2}}(\beta_{vv_1}(G_1)-\beta_{vv_2}(G_1))(\beta_{\pi^*(vv_1)}(G_2)-\beta_{\pi^*(vv_2)}(G_2))\\&~~~~~~~+ \varphi(r) (\bm{x}_{v_1}-\bm{x}_{v_2})^\top (\bm{y}_{\pi^*(v_1)}-\bm{y}_{\pi^*(v_2)})\le 0   \Bigg]\\&\ge \prob{\sum_{v\in V(G_1)\backslash\sth{v_1,v_2}}(\beta_{vv_1}(G_1)-\beta_{vv_2}(G_1))(\beta_{\pi^*(vv_1)}(G_2)-\beta_{\pi^*(vv_2)}(G_2))\le 0}\\&~~~~\cdot \prob{(\bm{x}_{v_1}-\bm{x}_{v_2})^\top (\bm{y}_{\pi^*(v_1)}-\bm{y}_{\pi^*(v_2)})\le 0}.
\end{align*}

We then bound the probability of two events separately.
We note that \begin{align*}
    \begin{bmatrix}
        X_v\\Y_v
    \end{bmatrix}\triangleq\begin{bmatrix} \beta_{vv_1}(G_1)-\beta_{vv_2}(G_1) \\ \beta_{\pi^*(vv_1)}(G_2)-\beta_{\pi^*(vv_2)}(G_2) \end{bmatrix} \sim \mathcal{N}\left( \begin{bmatrix} 0 \\ 0 \end{bmatrix}, 2 \begin{bmatrix} 1 & \rho \\ \rho & 1 \end{bmatrix} \right).
\end{align*}

Let $\xi_v\overset{\mathrm{i.i.d.}}{\sim} \maN(0,1)$ for any $v\in V(G_1)\backslash\sth{v_1,v_2}$. We note that \begin{align*}
    &~\prob{\sum_{v\in V(G_1)\backslash\sth{v_1,v_2}}X_v Y_v\ge 0}\\=&~\expect{\prob{\sum_{v\in V(G_1)\backslash\sth{v_1,v_2}}X_v Y_v\ge 0\Big| \sth{v\in V(G_1)\backslash\sth{v_1,v_2}}}}\\\overset{\mathrm{(a)}}{=}&~\expect{\prob{\sum_{v\in V(G_1)\backslash\sth{v_1,v_2}}\rho X_v^2+\sqrt{2(1-\rho^2)}X_v\xi_v\le 0\Big| \sth{v\in V(G_1)\backslash\sth{v_1,v_2}}}}\\\overset{\mathrm{(b)}}{=}&~\expect{\prob{\maN(0,1)\ge \frac{\rho\sqrt{\sum_{v\in V(G_1)\backslash\sth{v_1,v_2}}X_v^2}}{\sqrt{2(1-\rho^2)}}\Big| \sth{v\in V(G_1)\backslash\sth{v_1,v_2}}}},
\end{align*}
where $\mathrm{(a)}$ is because $Y_v|X_v\sim \maN(\rho X_v,2(1-\rho^2))$ and $\sth{Y_v|X_v,v\in V(G_1)\backslash\sth{v_1,v_2}}$ are independent; $\mathrm{(b)}$ is because $$\sum_{v\in V(G_1)\backslash\sth{v_1,v_2}} X_v\xi_v\Bigg|\sth{X_v:v\in V(G_1)\backslash\sth{v_1,v_2}}\sim \maN\pth{0,2(1-\rho^2)\sum_{v\in V(G_1)\backslash\sth{v_1,v_2}} X_v^2}.$$ By Lemma~\ref{lem:chisquare}, since $\sum_{v\in V(G_1)\backslash\sth{v_1,v_2}}\frac{1}{2} X_v^2\sim \chi^2(n-2)$, we have \begin{align*}
    \prob{\sum_{v\in V(G_1)\backslash\sth{v_1,v_2}} X_v^2\le 2(n+\sqrt{n\log n})}=1-o(1).
\end{align*}
Consequently, \begin{align*}
    &~\prob{\sum_{v\in V(G_1)\backslash\sth{v_1,v_2}}X_v Y_v\ge 0}\\\ge&~\expect{(1-o(1))\prob{\maN(0,1)\ge \frac{\rho\sqrt{2(n+\sqrt{n\log n})}}{\sqrt{2(1-\rho^2)}}}}\\\overset{\mathrm{(a)}}{\ge}&~\frac{1-o(1)}{\sqrt{2\pi}}\exp\pth{-\frac{1}{2}\frac{\rho^2(n+\sqrt{n\log n})}{1-\rho^2}}\frac{2}{\frac{\rho\sqrt{n+\sqrt{n\log n}}}{\sqrt{1-\rho^2}}+\sqrt{4+\frac{\rho^2(n+\sqrt{n\log n})}{1-\rho^2}}}\\\overset{\mathrm{(b)}}{\ge}&~\frac{1}{16\sqrt{\log n}}(1-\rho^2)^{-\frac{1}{2}(n-2)(1+o(1))},
\end{align*}
where $\mathrm{(a)}$ is because $\prob{Z\ge t}\ge \frac{2}{t+\sqrt{t^2+4}}\frac{1}{\sqrt{2\pi}}\exp\pth{-\frac{1}{2}t^2}$ for $Z\sim \maN(0,1)$ ~\cite{Birnbaum1942inequality}; 
$\mathrm{(b)}$ is because $n\log\pth{\frac{1}{1-\rho^2}}\le (4-\epsilon)\log n$ implies $\frac{\rho^2}{1-\rho^2}(n+\sqrt{n\log n})\le 4\log n$, $\frac{1-o(1)}{\sqrt{2\pi}}\cdot\frac{2}{\sqrt{4\log n}+\sqrt{4+4\log n}}\ge \frac{1}{16\sqrt{\log n}}$, and $\frac{\rho^2}{1-\rho^2} = (1+o(1))\log\pth{\frac{1}{1-\rho^2}}$. 
It follows from~\cite[Proposition 4.3]{kunisky2022strong} that when $r^2 \ge \frac{40}{d}$, \begin{align*}
    \prob{(\bm{x}_{v_1}-\bm{x}_{v_2})^\top (\bm{y}_{\pi^*(v_1)}-\bm{y}_{\pi^*(v_2)})\le 0}\ge \frac{1}{1000\sqrt{d}}(1-r^2)^{\frac{d}{2}}.
\end{align*}
Consequently,\begin{align*}
    \epsilon_1\ge \frac{1}{16000\sqrt{d\log n
    }}(1-\rho^2)^{\frac{n-2}{2}(1+o(1))}(1-r^2)^{\frac{d}{2}}.
\end{align*}

By Lemma~\ref{lem:d=3}, for any $\pi_1\neq \pi_2\in \maT_2$ with $\sfd(\pi_1,\pi_2) = 3$, we have \begin{align*}
    \epsilon_2\triangleq\prob{(G_1,G_2)\in \maE(\pi^*,\pi_1)\cap \maE(\pi^*,\pi_2)}\le (1-\rho^2)^{\frac{3(n-2)}{4}}(1-r^2)^{\frac{3d}{4}}.
\end{align*}
By~\eqref{eq:prob-M-0} and~\eqref{eq:I-second-moment},\begin{align*}
    \prob{|\maM| = 0}&\le \frac{\expect{I^2}-(\expect{I})^2}{(\expect{I})^2}=\frac{\binom{n}{2}\epsilon_1+{6\binom{n}{3}\epsilon_2+6\binom{n}{4}\epsilon_1^2}-\binom{n}{2}^2\epsilon_1^2}{\binom{n}{2}^2 \epsilon_1^2}\le \frac{4}{n^2 \epsilon_1}+\frac{4\epsilon_2}{n\epsilon_1^2}.
\end{align*}
Since $n\log\pth{\frac{1}{1-\rho^2}}+d\log\pth{\frac{1}{1-r^2}}+4\log d\le (4-\epsilon)\log n$ , we obtain \begin{align*}
    n^2 \epsilon_1 &\ge \frac{1}{16000\sqrt{\log n}}\\&~~~~\cdot\exp\pth{2\log n-\frac{1}{2}\log d-\frac{n-2}{2}(1+o(1))\log\pth{\frac{1}{1-\rho^2}}-\frac{d}{2}\log\pth{\frac{1}{1-r^2}}}\\&\ge \frac{1}{16000\sqrt{\log n}}\exp\pth{\frac{\epsilon}{4}\log n}
\end{align*}
and \begin{align*}
    \frac{n\epsilon_1^2}{\epsilon_2}&\ge \frac{1}{256\cdot 10^6\log n}\\&~~~~\cdot\exp\pth{\log n-\log d-\frac{n-2}{4}(1+o(1))\log\pth{\frac{1}{1-\rho^2}}-\frac{d}{4}\log\pth{\frac{1}{1-r^2}}}\\&\ge\frac{1}{256\cdot 10^6\log n}\exp\pth{\frac{\epsilon}{8}\log n}. 
\end{align*}
Therefore, we obtain $\prob{|\maM| = 0}\le \frac{4}{n^2\epsilon_1}+\frac{4\epsilon_2}{n\epsilon_1^2}=o(1)$, we finish the proof.

\subsection{Proof of Proposition~\ref{prop:convergenceguarantee}} 
Recall $f$: 
$$f(\Pi)\triangleq\lambda \Vert A_1\Pi-\Pi A_2 \Vert_F^2+(1-\lambda) \sum_{i=1}^d \Vert B_1^i\Pi-\Pi B_2^i\Vert_F^2.$$
To prove Proposition~\ref{prop:convergenceguarantee}, we need the following proposition to establish the convergence of function $f$.
\begin{proposition}\label{prop:gradient}
    For any two graphs $G_1,G_2$, there exists a universal constant $L = L(G_1,G_2)$ such that for any $\eta\le L^{-1}$, we have \begin{align*}
        |f(\Pi^K) - f(\Pi')|\le \frac{1}{2\eta K} \Vert \Pi^0-\Pi'\Vert_F^2\le \frac{n}{\eta K}
    \end{align*}
    for any integer $K\ge 1$, where $\Pi'\in \arg\min_{\Pi\in \mathbb{W}^n} f(\Pi)$ and $\Pi^0$ is the initial state.
\end{proposition}
The proof of Proposition~\ref{prop:convergenceguarantee} is deferred to Appendix~\ref{apd:proof-prop-gradient}.
Without loss of generality, in the following analysis, we suppose $\pi^*=\mathrm{id}$ for simplicity. To obtain the convergence guarantee of $\|\Pi^K-\Pi'\|$, we need the following lemma:
\begin{lemma}\label{lem:distanceunionbound}
    For the distance matrix $D\in \mathbb R^{n\times n}$ defined as $D_{ij}=\|\bm x_i-\bm y_j\|_2^2$, for any $0<\delta<1$, if $d>32\log\frac{n}{\sqrt\delta}$, then with probability at least $1-\delta$, $\min_{i,j}D_{i,j}\ge (1-r)d$.
\end{lemma}
\begin{proof}[Proof of Lemma~\ref{lem:distanceunionbound}]
    Since $\pi^*=\mathrm{id}$, for correlated pair $\bm x_i, \bm y_i$, we have $\bm x_i-\bm y_i\sim \calN(\bm 0, 2(1-r)I_{d})$, which implies $D_{ii}=\|\bm x_i-\bm y_i\|^2\sim 2(1-r)\chi^2_{d}$. For independent pair $\bm x_k, \bm y_j$, $k\ne j$, $\bm x_k-\bm y_j\sim \calN(\bm 0, 2I_d)$, $D_{kj}=\|\bm x_k-\bm y_j\|^2\sim 2\chi^2_d$.
    
    By Lemma~\ref{lem:chisquare}, we have
    $$\prob{D_{ii}<(1-r)d}\le \exp\pth{-\frac1{16}d},\quad \prob{D_{kj}<d}\le \exp\pth{-\frac1{16}d}.$$
    Taking union bound, we obtain
    $$\prob{\min_{i,j}D_{ij}<(1-r)d}\le n^2\exp\pth{-\frac1{16}d}\le \delta.$$
\end{proof}

    The Hessian of $f$ has matrix expression
    \begin{align*}
\nabla^2 f&=2\lambda(I\otimes A_1-A_2^\top\otimes I)^\top(I\otimes A_1-A_2^\top\otimes I)+2(1-\lambda)\operatorname{diag}(\{D_{ij}\}_{1\le i,j\le n})\\
&\succeq 2(1-\lambda)\min_{1\le i,j\le n}(D_{ij})I,
    \end{align*}
    where $A\succeq B$ means that for symmetric matrices $A$ and $B$, $A-B$ is positive semidefinite.
    Denote $m\triangleq 2(1-\lambda)\min_{i,j}(D_{ij})$. Since $\Pi'$ minimizes $f$ on $\mathbb W_n$, it satisfies $\langle\nabla f(\Pi'),\Pi-\Pi'\rangle\ge 0$. Therefore,  for any $\Pi\in\mathbb W_n$, $f(\Pi)\ge f(\Pi')+\langle\nabla f(\Pi'),\,\Pi-\Pi'\rangle+\frac{m}{2}\|\Pi-\Pi'\|_F^2\ge f(\Pi')+\frac m2\|\Pi-\Pi'\|_F^2.$
    Following Lemma~\ref{lem:distanceunionbound}, for any $0<\delta<1$, if $d>32\log\frac{n}{\sqrt\delta}$, then with probability at least $1-\delta$, $\min_{i,j}D_{i,j}\ge (1-r)d$, which implies
    $$\|\Pi^K-\Pi'\|_F^2\le \frac{2}{m}|f(\Pi^K)-f(\Pi')|\le \frac{n}{(1-\lambda)(1-r)d\eta K},$$
    where the second inequality follows from Proposition~\ref{prop:gradient}.

\subsection{Proof of Proposition~\ref{prop:possible-all-lambda}}\label{apd:proof-prop-all-lambda}
We consider the following estimator:\begin{align*}
    \hat \pi_\lambda=\mathop{\text{argmax}}_{\pi\in \maS_n}\left\{\lambda\sum_{e\in E(G_1)}\beta_e(G_1)\beta_{\pi(e)}(G_2)+(1-\lambda)\sum_{v\in V(G_1)}\bm x_v^\top\bm y_{\pi(v)}\right\}.
\end{align*}
Note that for any $\tau\in \mathbb R$ we have 
\begin{align}
    \nonumber \{d(\hat \pi_\lambda,\pi)=k\}\subseteq&~ \{\exists\; \pi'\in \mathcal T_k, s.t., \lambda\sum_e\beta_e(G_1)\beta_{\pi(e)}(G_2)+(1-\lambda)\sum_v \bm x_v^\top\bm y_{\pi(v)}\\\nonumber
    &-\lambda\sum_e\beta_e(G_1)\beta_{\pi'(e)}(G_2)-(1-\lambda)\sum_v \bm x_v^\top\bm y_{\pi'(v)}\le 0\}\\\nonumber
    =&~\{\exists\; \pi'\in \mathcal T_k, s.t., \lambda\left(\sum_{e\notin \mathcal E}\beta_e(G_1)\beta_{\pi(e)}(G_2)-\sum_{e\notin \mathcal E}\beta_e(G_1)\beta_{\pi'(e)}(G_2)\right)\\\nonumber
    &+(1-\lambda)\left(\sum_{v\notin F} \bm x_v^\top\bm y_{\pi(v)}-\sum_{v\notin F} \bm x_v^\top\bm y_{\pi'(v)}\right)\le 0\}\\\label{eq:apd-term1}
    \subseteq &~ \left\{\exists\; \pi'\in \mathcal T_k, s.t., \lambda\sum_{e\notin \mathcal E}\beta_e(G_1)\beta_{\pi(e)}(G_2)+(1-\lambda)\sum_{v\notin F} \bm x_v^\top\bm y_{\pi(v)}<\tau\right\}\\\label{eq:apd-term2}
    &~\bigcup\left\{\exists\; \pi'\in \mathcal T_k, s.t., \lambda\sum_{e\notin \mathcal E}\beta_e(G_1)\beta_{\pi'(e)}(G_2)+(1-\lambda)\sum_{v\notin F} \bm x_v^\top\bm y_{\pi'(v)}\ge\tau\right\}.
\end{align}

We then upper bound the error probability for~\eqref{eq:apd-term1} and~\eqref{eq:apd-term2}, respectively.
We first consider~\eqref{eq:apd-term1}.

Let $X=(X_1,\cdots,X_{N_k},\tilde X_1,\cdots, \tilde X_k)^\top$ and $Y=(Y_1,\cdots, Y_{N_k},\tilde Y_1,\cdots,\tilde Y_k)^\top$, where $(X_i, Y_i)\overset{i.i.d.}{\sim}\maN(\begin{bmatrix}
        0\\0
    \end{bmatrix}, \begin{bmatrix}
        1&\rho\\\rho&1
    \end{bmatrix})$, and $(\tilde X_i, \tilde Y_i)\overset{i.i.d.}{\sim}\maN(\begin{bmatrix}         \bm0\\\bm0     \end{bmatrix}, \begin{bmatrix}         I_d&rI_d\\rI_d&I_d     \end{bmatrix})$.

Then $$W\triangleq\lambda\sum_{e\notin \binom F2}\beta_e(G_1)\beta_{\pi(e)}(G_2)+(1-\lambda)\sum_{v\notin F} \bm x_v^\top\bm y_{\pi(v)}\stackrel d=X^\top AY,$$ where $A = \mathrm{diag}\sth{\lambda I_{N_k},(1-\lambda)I_{dk}}$ with $\Vert A\Vert_F^2 = \lambda^2 N_k+(1-\lambda)^2 dk$ and $\Vert A\Vert_2 = \lambda\vee (1-\lambda)$. By Lemma \ref{lem:hwforweaksignal}, there exists a universal constant $C$, such that with probability at least $1-\delta_0$,
    $$|W-(\rho\lambda N_k+r(1-\lambda)kd)|\le  C\pth{\Vert A\|_F\sqrt{\log\frac{1}{\delta_0}}+\|A\|_2\log \frac{1}{\delta_0}}.$$ 
    Let $\tau=(\rho\lambda N_k+r(1-\lambda)kd)-C(\|A\|_F\sqrt{\log\frac1{\delta_0}}\vee \|A\|_2\log \frac{1}{\delta_0})$ with $\delta_0 = \exp(-2k\log n)$:\begin{align}\label{eq:def-of-tau-apd}
        \tau = (\rho\lambda N_k+r(1-\lambda)kd) - C(\|A\|_F\sqrt{2k\log n}+\|A\|_2\cdot 2k\log n).
    \end{align} 

We obtain \begin{align}\label{eq:term1-proof}
    \prob{\bigcup_{\pi'\in \calT_k}\{\lambda\sum_{e\notin \mathcal E}\beta_e(G_1)\beta_{\pi(e)}(G_2)+(1-\lambda)\sum_{v\notin F} \bm x_v^T\bm y_{\pi(v)}<\tau\}}\le \binom nk\delta_0\le n^{-k},
\end{align}
where the last inequality is because $\delta_0 = \exp(-2k\log n)$ and $\binom{n}{k}\le n^k$.

We then focus on~\eqref{eq:apd-term2}. We first introduce the following lemma.

\begin{lemma}\label{lem:term2-error-bound}
   Under Assumption~\ref{assm:choiceoflambda}, if $d=\omega(\log n)$ and $n\log\frac1{1-\rho^2}+d\log\frac1{1-r^2}\ge C_0\log n$ for some $C_0\ge \frac{(32(13+C))^2(1+\Gamma)}{\delta^2}$, then for $\tau$ in~\eqref{eq:def-of-tau-apd} and any $\lambda\in (\delta, 1-\delta)$, we have
    \begin{align}\label{eq:term2-proof}
        \prob{\bigcup_{\pi'\in T_k}\left\{\lambda\sum_{e\notin \mathcal E}\beta_e(G_1)\beta_{\pi'(e)}(G_2)+(1-\lambda)\sum_{v\notin F} \bm x_v^T\bm y_{\pi'(v)}\ge\tau\right\}}\le n^{-2k}.
    \end{align}
\end{lemma}
By~\eqref{eq:apd-term1},~\eqref{eq:apd-term2},~\eqref{eq:term1-proof}, and~\eqref{eq:term2-proof}, we obtain\begin{align*}
    \prob{\hat\pi_\lambda\neq \pi^*}\le \sum_{k\ge 1}\pth{n^{-k}+n^{-2k}} = o(1).
\end{align*}

\begin{proof}[Proof of Lemma~\ref{lem:term2-error-bound}]
Let $Z_\lambda\triangleq\lambda\sum_{e\notin \mathcal E}\beta_e(G_1)\beta_{\pi'(e)}(G_2)+(1-\lambda)\sum_{v\notin F} \bm x_v^T\bm y_{\pi'(v)}$. Following a similar argument with Appendix~\ref{apd:proof-exact-possible}, we have $$\prob{Z_\lambda\ge\tau}\le e^{-t\tau}\expect{e^{tZ}}=\exp\left(-t\tau+\frac{N_k}{2}\kappa_2^{\sfE,\lambda}(t)+\frac{k}{2}\Big(\kappa_1^{\sfE,\lambda}(t)-\frac12\kappa_2^{\sfE,\lambda}(t)\Big)+\frac{k}{2}\kappa_2^{\sfV,\lambda}(t)\right),$$
where 
\begin{align*}
    \kappa_2^{\sfE,\lambda}(t)&\triangleq -\frac12\log(1-2t^2\lambda^2(1+\rho^2)+t^4\lambda^4(1-\rho^2)^2),\\  \kappa_2^{\sfV,\lambda}(t)&\triangleq  -\frac d2\log(1-2t^2(1-\lambda)^2(1+r^2)+t^4(1-\lambda)^4(1-r^2)^2),
\end{align*}
and \begin{align*}
    \kappa_1^{\sfE,\lambda}(t)&\triangleq -\frac12\log(1-2t\rho\lambda-t^2\lambda^2(1-\rho^2)),\\
    \kappa_1^{\sfV,\lambda}(t)&\triangleq -\frac d2\log(1-2tr(1-\lambda)-t^2(1-\lambda)^2(1-r^2)).
\end{align*} 

When $t\le \frac1{16}$, $(1\pm\rho)^2,(1\pm r)^2\le 4$, and thus $4t^2\le 1/64$. Since $-\log(1-x)\le \frac{x}{1-x}\le 2x$ holds for all $x\le 1/2$, we have
\begin{align*}
\kappa_2^{\sfE,\lambda}(t)&=-\frac12\log(1-t^2\lambda^2(1+\rho)^2)-\frac12\log(1-t^2\lambda^2(1-\rho)^2)\\
&\le t^2\lambda^2(1+\rho)^2+t^2\lambda^2(1-\rho)^2\le 4t^2\lambda^2,
\end{align*}
where the last inequality follows from $(1+x)^2+(1-x)^2\le 4$ for $0\le x\le 1$. Similarly,
$$\kappa_2^{\sfV,\lambda}(t)\le 4dt^2(1-\lambda)^2,\quad \kappa_1^{\sfE,\lambda}\le 2t\rho\lambda +t^2\lambda^2.$$
Recall that \begin{align}
        \nonumber \tau &= (\rho\lambda N_k+r(1-\lambda)kd) - C(\|A\|_F\sqrt{2k\log n}+\|A\|_2\cdot 2k\log n)\\\label{eq:def-of-tau-2}&=(\rho\lambda N_k+r(1-\lambda)kd)-C\left(\Big(\sqrt{\lambda^2N_k+(1-\lambda)^2kd}\sqrt{2k\log n}\Big)+ \Big((\lambda\vee (1-\lambda))\cdot 2k\log n\Big)\right).
    \end{align} 
Note that $N_k = nk(1-\frac{k+1}{2n})\ge \frac{kn}{3}$. 
Let $P=\rho\lambda(N_k-k)+r(1-\lambda)kd$, $Q=\lambda^2N_k+(1-\lambda)^2kd$.
We have \begin{align}
    \nonumber&~-t\tau+\frac{N_k}{2}\kappa_2^{\sfE,\lambda}(t)+\frac{k}{2}\Big(\kappa_1^{\sfE,\lambda}(t)-\frac12\kappa_2^{\sfE,\lambda}(t)\Big)+\frac{k}{2}\kappa_2^{\sfV,\lambda}(t)\\ \nonumber
    \overset{\mathrm{(a)}}{\le} &~ -t\tau+2N_kt^2\lambda^2+kt\rho\lambda+\frac k2t^2\lambda^2+2kd t^2(1-\lambda)^2\\ \nonumber
    \le&~ -t\tau+3t^2(\lambda^2N_k+(1-\lambda)^2kd)+kt\rho\lambda\\ \nonumber
     \overset{\mathrm{(b)}}{\le}&~ -t\pth{\rho\lambda(N_k-k)+r(1-\lambda)kd}+3t^2(\lambda^2N_k+(1-\lambda)^2kd)\\ \nonumber
    &~ +tC\sqrt{2k\log n(\lambda^2N_k+(1-\lambda)^2kd)}+tC2k\log n\\\label{eq:chernoff-mid-1}
    =&~-tP+3t^2Q+tC\sqrt{2Qk\log n}+tC2k\log n,
\end{align}
where $\mathrm{(a)}$ is because $\kappa_2^{\sfE,\lambda}(t)\le 4t^2 \lambda^2$, $\kappa_1^{\sfE,\lambda}(t)-\frac12\kappa_2^{\sfE,\lambda}(t)\le \kappa_1^{\sfE,\lambda}(t)\le 2t\rho \lambda+t^2\lambda^2$, and $\kappa_2^{\sfV,\lambda}(t)\le 4dt^2(1-\lambda)^2$; $\mathrm{(b)}$ follows from~\eqref{eq:def-of-tau-2} and $\lambda\vee (1-\lambda)\le 1$.

Pick $t_0=\frac{1}{16}\pth{\sqrt{\frac{2k\log n}{Q}}\wedge 1}$. Then  $3t_0^2Q\le \frac 18k\log n$, $t_0C\sqrt{2k\log nQ}\le \frac{C}8\log n$, and $t_0C2k\log n\le \frac{C}8\log n$.
Under Assumption~\ref{assm:choiceoflambda}, since $n\log\frac1{1-\rho^2}+d\log\frac1{1-r^2}\ge C_0\log n$, we have $$n\log\pth{\frac1{1-\rho^2}}\ge \frac{C_0\log n}{1+\Gamma},\quad d\log\pth{\frac1{1-r^2}}\ge \frac{C_0\log n}{1+\Gamma}.$$
Since $n = \omega(\log n)$ and $d = \omega(\log n)$, we have $\rho = o(1)$ and $r = o(1)$. Hence $\log(1/(1-\rho^2)) = (1+o(1)) \rho^2$ and $\log(1/(1-r^2)) = (1+o(1))r^2$. Therefore, \begin{align*}
    n\rho^2 \ge \frac{C_0\log n}{2(1+\Gamma)},\quad dr^2 \ge \frac{C_0\log n}{2(1+\Gamma)}.
\end{align*}
Note that $N_k-k = kn(1-\frac{k+3}{2n})\ge \frac{kn}{3}$ for sufficiently large $n$. Consequently, if $t_0 = \frac{1}{16}$, since $\lambda,1-\lambda\ge \delta$, then \begin{align*}
    t_0P&\ge \frac{\delta}{16}(\frac{\rho kn}{3}+rkd)\ge \frac{\delta k}{64}(n\rho+rd)\\
    &\ge \frac{\delta k}{64}\pth{\sqrt n\sqrt{\frac12 \frac{C_0\log n}{1+\Gamma}}+\sqrt d\sqrt{\frac12\frac{C_0\log n}{1+\Gamma}}}\ge \frac{\sqrt{C_0}\delta k}{128}\frac{\log n}{\sqrt{1+\Gamma}}.
\end{align*}
If $t_0=\frac{1}{16}\sqrt{\frac{2k\log n}{Q}}$, then  
\begin{align*}
    t_0P&\ge \frac{\sqrt{2k\log n}}{32}\frac{\frac12\lambda\rho N_k+(1-\lambda)rkd
    }{\sqrt{\lambda^2N_k+(1-\lambda)^2kd}}\\
    &\stackrel{\text{(a)}}{\ge} \frac{\delta\sqrt{2k\log n}}{32}\frac{(\frac\rho2\sqrt{N_k}\wedge r\sqrt{kd})(\lambda\sqrt{N_k}+(1-\lambda)\sqrt{kd})}{\lambda \sqrt{N_k}+(1-\lambda)\sqrt{kd}}\\
    &=\frac{\delta\sqrt{2k\log n}}{32}(\frac\rho2\sqrt{N_k}\wedge r\sqrt{kd})\ge  \frac{\delta\sqrt{2k\log n}\sqrt{k}}{128}(\sqrt{n\rho^2}\wedge \sqrt{dr^2})\\
    &\ge \frac{\delta}{128}\frac{\sqrt{C_0}k\log n}{\sqrt{1+\Gamma}},
\end{align*}
where $\mathrm{(a)}$ is because $\lambda,1-\lambda \ge \delta $ and $\sqrt{\lambda^2 N_k+(1-\lambda)^2 kd}\le \lambda\sqrt{N_k}+(1-\lambda )\sqrt{kd}$. Therefore, by Chernoff's bound,\begin{align*}
    \prob{Z_\lambda \ge \tau }&\le \exp(-t_0\tau)\expect{\exp(t_0 Z_\lambda)}\\&=\exp\left(-t\tau+\frac{N_k}{2}\kappa_2^{\sfE,\lambda}(t)+\frac{k}{2}\Big(\kappa_1^{\sfE,\lambda}(t)-\frac12\kappa_2^{\sfE,\lambda}(t)\Big)+\frac{k}{2}\kappa_2^{\sfV,\lambda}(t)\right)\\&\overset{\mathrm{(a)}}{\le} \exp\pth{-t_0P+3t_0^2Q+t_0C\sqrt{2kQ\log n }+t_0C2k\log n}\\&\overset{\mathrm{(b)}}{\le} \exp\pth{-\frac{\delta\sqrt{C_0}}{128\sqrt{1+\Gamma}}k\log n+\frac{1+C}{4}k\log n}\le \exp(-3k\log n),
\end{align*}
where $\mathrm{(a)}$ follows from~\eqref{eq:chernoff-mid-1}; $\mathrm{(b)}$ is because $3t_0^2Q\le \frac 18k\log n$, $t_0C\sqrt{2k\log nQ}\le \frac{C}8\log n$, $t_0C2k\log n\le \frac{C}8\log n$, and $t_0P\ge \frac{\delta}{128}\frac{\sqrt{C_0}k\log n}{\sqrt{1+\Gamma}}$.
Applying union bound yields
\begin{align*}
    &\prob{\bigcup_{\pi'\in T_k}\left\{\lambda\sum_{e\notin \mathcal E}\beta_e(G_1)\beta_{\pi'(e)}(G_2)+(1-\lambda)\sum_{v\notin F} \bm x_v^T\bm y_{\pi'(v)}\ge\tau\right\}}\\
        \le &~\binom nk k!\exp\left(\inf_{t>0}\sth{-t\tau+\frac{N_k}{2}\kappa_2^{\sfE,\lambda}(t)+\frac{k}{2}\Big(\kappa_1^{\sfE,\lambda}(t)-\frac12\kappa_2^{\sfE,\lambda}(t)\Big)+\frac{k}{2}\kappa_2^{\sfV,\lambda}(t)}\right)\\
        \le &~ n^k\exp\left(-3k\log n\right)\\
        \le &~\exp\left(k\log n-3k\log n\right)=n^{-2k}.
\end{align*}
\end{proof}

\subsection{Proof of Proposition~\ref{prop:gradient}}\label{apd:proof-prop-gradient}

Let $L\triangleq 2\sth{\lambda(\|A_1\|^2+\|A_2\|^2)^2+(1-\lambda)\sum_{i=1}^d(\|B_1^i\|_2+\|B_2^i\|_2)^2}$. We first show that $\nabla f$ is $L$-Lipschitz. Define the linear operator $$T(X)\triangleq\pth{\sqrt \lambda (A_1X-XA_2), \sqrt{1-\lambda}(B_1^1X-XB_2^1),\cdots, \sqrt{1-\lambda} (B_1^dX-XB_2^d) },$$
    then $f(\Pi)=\|T(\Pi)\|_F^2=\langle \Pi, T^*T\Pi\rangle$ and $\nabla f(\Pi)=2T^*T(\Pi)$, where \(T^*\) denotes the adjoint of \(T\) with respect to the Frobenius inner product 
\(\langle X,Y\rangle_F = \operatorname{tr}(X^\top Y)\). Therefore, \begin{align}\label{eq:L-lip}
    \|\nabla f(X)-\nabla f(Y)\|_F\le 2\|T\|^2\|X-Y\|_F.
\end{align} For each component of $T(X)$, $\|A_1X-XA_2\|_F\le (\|A_1\|_2+\|A_2\|_2)\|X\|_F$, and similarly, $\|B_1^iX-XB_2^i\|_F\le (\|B_1^i\|_2+\|B_2^i\|_2)\|X\|_F$, which implies 
    $$\|T\|^2\le \lambda(\|A_1\|^2+\|A_2\|^2)^2+(1-\lambda)\sum_{i=1}^d(\|B_1^i\|_2+\|B_2^i\|_2)^2.$$
Combining this with~\eqref{eq:L-lip}, we conclude that $\nabla f$ is $L$-Lipschitz.

Recall that $\Pi^{k+1}=\mathsf{Proj}_{\mathbb W^n}(\Pi^k-\eta\nabla f(\Pi^k)).$
Let $Y^k=\Pi^k-\eta\nabla f(\Pi^k)$. Since 
    \begin{align*}
        \|X-Y^k\|_F^2&=\|X-\Pi^k+\eta\nabla f(\Pi^k)\|_F^2\\
        &=\|X-\Pi^k\|_F^2+2\eta\langle \nabla f(\Pi^k),X-\Pi^k\rangle+\eta^2\|\nabla f(\Pi^k)\|_F^2,
    \end{align*}
    we have
    \begin{align*}
        \langle\nabla f(\Pi^k), X-\Pi^k\rangle+\frac1{2\eta}\|X-\Pi^k\|_F^2=\frac{1}{2\eta}\|X-Y^k\|_F^2-\frac\eta2\|\nabla f(\Pi^k)\|_F^2.
    \end{align*}
Therefore, the Euclidean projection
    \begin{align*}
        \mathsf{Proj}_{\mathbb W^n}(\Pi^k-\eta\nabla f(\Pi^k))&=\argmin_{X\in \mathbb W^n}\|X-Y^k\|_F^2\\
        &=\argmin_{X\in \mathbb W^n}\sth{\langle\nabla f(\Pi^k), X-\Pi^k\rangle+\frac1{2\eta}\|X-\Pi^k\|_F^2}.
    \end{align*}
     Since $\mathbb W^n$ is convex, we have that $\Pi^{k+1}+t(\Pi-\Pi^{k+1})\in \mathbb W^n$ for any $t\in (-1,1)$ and $\Pi\in \mathbb W^n$. Since $\Pi^{k+1}$ minimizes $g(X)\triangleq\frac12\|X-Y^k\|_F^2$, we have $\frac{d}{dt}g(\Pi^{k+1}+t(\Pi-\Pi^{k+1}))\Big|_{t=0+}\ge 0$, which implies \begin{align}\label{eq:project2}
         \langle \Pi^{k+1}-Y^k,\Pi-\Pi^{k+1}\rangle\ge 0
     \end{align} for any $\Pi \in \mathbb W^n$. Consequently, take $\Pi = \Pi'$ yields that \begin{align}
         \nonumber \langle \nabla f(\Pi^k),\Pi^{k+1}-\Pi'\rangle &\le \frac{1}{\eta}\langle \Pi^{k+1}-\Pi^k,\Pi^{k+1}-\Pi' \rangle  \\\label{eq:project-prop}&=\frac1{2\eta}\pth{\|\Pi^k-\Pi'\|_F^2-\|\Pi^{k+1}-\Pi'\|_F^2-\|\Pi^{k+1}-\Pi^k\|_F^2}
     \end{align}

We then establish the upper bound for $|f(\Pi^K) - f(\Pi')|$. For $L-$Lipschitz function $f$, we have 
    $$f(Y)\le f(X)+\langle \nabla f(X), Y-X\rangle +\frac L2\|Y-X\|_F^2,\quad \forall X,Y\in \mathbb W^n,$$
    which implies 
    \begin{align}\label{eq:project1}
        f(\Pi^{k+1})\le f(\Pi^k)+\langle \nabla f(\Pi^k),\Pi^{k+1}-\Pi^k\rangle+\frac{L}{2}\|\Pi^{k+1}-\Pi^{k}\|_F^2.
    \end{align}
We decompose the second term as 
    $$\langle \nabla f(\Pi^k),\Pi^{k+1}-\Pi^k\rangle=\langle\nabla f(\Pi^k), \Pi^{k+1}-\Pi'\rangle+\langle \nabla f(\Pi^k),\Pi'-\Pi^k\rangle.$$    
    Since $f$ is convex, $f(\Pi')\ge f(\Pi^k)+\langle \nabla f(\Pi^k), \Pi'-\Pi^k\rangle$. Combining this with~\eqref{eq:project-prop} and~\eqref{eq:project1}, we obtain that 
\begin{align*}
        f(\Pi^{k+1})-f(\Pi')&\le \frac1{2\eta}\pth{\|\Pi^k-\Pi'\|_F^2-\|\Pi^{k+1}-\Pi'\|_F^2-\|\Pi^{k+1}-\Pi^k\|_F^2}+\frac L2\|\Pi^{k+1}-\Pi^k\|_F^2\\
        &= \frac1{2\eta}\pth{\|\Pi^k-\Pi'\|_F^2-\|\Pi^{k+1}-\Pi'\|_F^2}+\frac {L\eta-1}{2\eta}\|\Pi^{k+1}-\Pi^k\|_F^2\\
        &\le \frac1{2\eta}\pth{\|\Pi^k-\Pi'\|_F^2-\|\Pi^{k+1}-\Pi'\|_F^2},
    \end{align*}
    where the last inequality follows from $\eta\le 1/L$.

Take $\Pi = \Pi^k$ in~\eqref{eq:project2}, we obtain that \begin{align*}
    \langle \nabla f(\Pi^k),\Pi^{k+1}-\Pi^k\rangle\le -\frac1\eta\|\Pi^{k+1}-\Pi^k\|_F^2.
\end{align*}
Combining this with~\eqref{eq:project1}, we obtain 
\begin{align*}
    f(\Pi^{k+1})-f(\Pi^k)\le \frac{L\eta-2}{2\eta}\|\Pi^{k+1}-\Pi^k\|_F^2\le 0.
\end{align*}
    Taking sum for $k=0,1,\cdots, K-1$, since $f(\Pi^K)\le f(\Pi^k)$ for any $k=0,1,\cdots K-1$,
    $$K|f(\Pi^K)-f(\Pi')|=K(f(\Pi^K)-f(\Pi'))\le \sum_{k=0}^{K-1}f(\Pi^k)-Kf(\Pi')\le \frac1{2\eta}\|\Pi^0-\Pi'\|_F^2.$$
    The Birkhoff-von Neumann theorem (see, e.g. \cite[Theorem 8.7.2]{horn2012matrix}) states that a doubly stochastic matrix is a convex combination of permutation matrices, which implies $\mathbb W^n$ is the convex hull of $n\times n$ permutation matrices. For any $P,Q\in \mathbb W^n$, $(P,Q)\mapsto \|P-Q\|_F^2$ is convex for each variable, hence the maximum admits on the extreme points of $\mathbb W^n$, i.e., $P,Q$ are permutation matrices. For permutation matrices $P,Q$, $\tr(P^\top P)=\tr(Q^\top Q)=n$. Therefore, 
    $$\|\Pi^0-\Pi'\|_F^2\le \|P-Q\|_F^2=\tr(P^\top P)+\tr(Q^\top Q)-2\tr(P^\top Q)\le 2n.$$
    Consequently, 
    $$|f(\Pi^K)-f(\Pi')|\le \frac{1}{2\eta K} \Vert \Pi^0-\Pi'\Vert_F^2 \le  \frac{n}{\eta K}.$$

\section{Proof of Lemmas}

\subsection{Proof of Lemma~\ref{lem:hwforweaksignal}}
 Note that $W=X^{\top} AY = \frac{1}{4}(X+Y)^\top A(X+Y)-\frac{1}{4}(X-Y)^\top A(X-Y)$ and \begin{align*}
        \expect{(X+Y)^\top A(X+Y)} &= (2+2\rho)\varphi(\rho)N_k+(2+2r)\varphi(r)dk\\\expect{(X-Y)^\top A(X-Y)} &= (2-2\rho)\varphi(\rho)N_k+(2-2r)\varphi(r)dk.
    \end{align*}
    By Hanson-Wright inequality \cite{hanson1971bound}, there exists some universal constant $C$ such that \begin{align*}
        &\mathbb{P}\Bigg[\left| \frac{1}{4}(X+Y)^\top A(X+Y) - \expect{\frac{1}{4}(X+Y)^\top A(X+Y)}\right| \\&~~~~~~\ge \frac{C}{2}\pth{\Vert A\Vert_F \sqrt{\log\pth{\frac{1}{\delta_0}}}\vee\Vert A\Vert_2\log\pth{\frac{1}{\delta_0}}}\Bigg]\le \frac{\delta_0}{2},\\&\mathbb{P}\Bigg[\left| \frac{1}{4}(X-Y)^\top A(X-Y) - \expect{\frac{1}{4}(X-Y)^\top A(X-Y)}\right| \\&~~~~~~\ge \frac{C}{2}\pth{\Vert A\Vert_F \sqrt{\log\pth{\frac{1}{\delta_0}}}\vee\Vert A\Vert_2\log\pth{\frac{1}{\delta_0}}}\Bigg]\le \frac{\delta_0}{2}
    \end{align*} for any $\delta_0>0$.
    Consequently, \begin{align*}
        &~\prob{\left| X^\top A Y -\rho \varphi(\rho) N_k-r\varphi(r)dk\right|\ge C\pth{\Vert A\Vert_F \sqrt{\log\pth{\frac{1}{\delta_0}}}\vee\Vert A\Vert_2\log\pth{\frac{1}{\delta_0}}}}\\ \le &~\mathbb{P}\Bigg[\left| \frac{1}{4}(X+Y)^\top A(X+Y) - \expect{\frac{1}{4}(X+Y)^\top A(X+Y)}\right| \\&~~~~~~\ge \frac{C}{2}\pth{\Vert A\Vert_F \sqrt{\log\pth{\frac{1}{\delta_0}}}\vee\Vert A\Vert_2\log\pth{\frac{1}{\delta_0}}}\Bigg]\\+&~ \mathbb{P}\Bigg[\left| \frac{1}{4}(X-Y)^\top A(X-Y) - \expect{\frac{1}{4}(X-Y)^\top A(X-Y)}\right| \\&~~~~~~\ge \frac{C}{2}\pth{\Vert A\Vert_F \sqrt{\log\pth{\frac{1}{\delta_0}}}\vee\Vert A\Vert_2\log\pth{\frac{1}{\delta_0}}}\Bigg]\le \delta_0.
    \end{align*}
We finish the proof of Lemma~\ref{lem:hwforweaksignal}.

\subsection{Proof of Lemma~\ref{lem:cumulant}}

By~\eqref{eq:mgf-independent}, the cumulant generating function is given by\begin{align*}
    \log \expect{\exp\pth{tZ}} = \log \expect{\exp\pth{tZ^\sfV}}+\log\expect{\exp\pth{tZ^\sfE}}.
\end{align*}

We first calculate $\expect{\exp\pth{tZ^\sfE}}$.
Define the moment generating function (MGF) as $m_k^\sfE = \exp\pth{\kappa_k^\sfE}$ for any $k\ge 1$.
For any edge cycle $C = \sth{e_1,e_2,\cdots,e_k}$ with $e_{i+1} = \sigma^\sfE(e_i)$ for all $1\le i\le k-1$ and $e_1 = \sigma^\sfE(e_k)$, let $A_{i-1} = \beta_{e_i}(G_1)$ and $B_i = \beta_{\pi'(e_i)}(G_2)$ for any $1\le i\le k$, and we set $A_k = A_0$ for notational simplicity. Since $\pi^*(e_{i+1}) = \pi'(e_i)$, each pair $(A_i,B_i)$ follows bivariate normal distribution $\maN\pth{0,\begin{bmatrix}
    1&\rho\\\rho&1
\end{bmatrix})}$, and thus the conditional distribution is given by $A_i\vert B_i\sim \maN(\rho B_i,1-\rho^2)$. Consequently, the MGF is given by \begin{align*}
    m_k^\sfE &= \expect{\expect{\prod_{i=1}^k \exp\pth{t\varphi(\rho) A_{i-1} B_i}\big| B_1,\cdots B_k }}\\&=\expect{\prod_{i=1}^k \exp\pth{t\rho \varphi(\rho)B_{i-1}B_i+\frac{1}{2}t^2 \varphi(\rho)^2 B_{i}^2(1-\rho^2)}},
\end{align*}
where the last equality is because $\expect{\exp\pth{tX}} = \exp\pth{t\mu+\frac{1}{2} t^2 \sigma^2}$ for $X\sim \maN(\mu,\sigma^2)$.

Let $\lambda_1,\lambda_2$ denote the roots of the quadratic function $x^2 - \qth{1-t^2 \varphi(\rho)^2(1-\rho^2)} x + t^2 \varphi(\rho)^2 \rho^2 = 0$. Since $t \le \frac{1}{\rho}-2$, we have $\lambda_1+\lambda_2>0$ and the discriminant $\qth{1-t^2 \varphi(\rho)^2(1-\rho^2)}^2-4t^2 \varphi(\rho)^2 \rho^2 >0$. Since $\lambda_1 \lambda_2>0$, we have $\lambda_1>\lambda_2>0$. Define the matrix
\begin{align*}
    \mathbf{J}_{k} \triangleq \begin{bmatrix}
        \lambda_1^{1/2} & -\lambda_{2}^{1/2} &0& \cdots  &0 \\
        0&\lambda_1^{1/2 }&-\lambda_2^{1/2 } &\cdots &0\\
        0&0& \lambda_1^{1/2 }&\cdots &0\\
        \vdots & \vdots &\vdots &\ddots &-\lambda_2^{1/2 }\\
        -\lambda_2^{1/2 } &0&\cdots 0&0& \lambda_1^{1/2 }
    \end{bmatrix}\in \mathbb{R}^{k\times k}.
\end{align*}
Denote $\mathbf{B}_k = \qth{B_1,B_2,\cdots,B_k}^\top$. Then we have \begin{align*}
    m_{k}^\sfE & = \expect{\prod_{i=1}^k \exp\pth{t\rho \varphi(\rho)B_{i-1}B_i+\frac{1}{2}t^2 \varphi(\rho)^2 B_{i}^2(1-\rho^2)}} \\
    &= \idotsint  
 \pth{\frac{1}{\sqrt{2\pi}}}^k \exp\pth{-\frac{1}{2} \sum_{i=1}^{k}\pth{ B_i^2- \pth{2 t\rho\varphi(\rho) B_{i-1}B_{i}+t^2\varphi(\rho)^2(1-\rho^2)B_{i}^2}}} \,\\&~~~~\mathrm{d} B_1\cdots \mathrm{d} B_{k}\\
 &= \idotsint  
 \pth{\frac{1}{\sqrt{2\pi}}}^k \exp\pth{-\frac{1}{2} \sum_{i=1}^{k} \pth{\lambda_1^{1/2} B_{i-1}- \lambda_2^{1/2} B_{i}}^2} \,\mathrm{d} B_1\cdots \mathrm{d} B_{k}\\ &=\idotsint \pth{\frac{1}{\sqrt{2\pi}}}^{k} \exp\pth{-\frac{1}{2} \mathbf{B}_{k}^\top  {\mathbf{J}}_{k}^\top {\mathbf{J}}_{k} \mathbf{B}_{k}}\, \mathrm{d} B_1\cdots \mathrm{d} B_{k} = \qth{\det({\mathbf{J}}_{k})}^{-1}=\frac{1}{\lambda_1^{k/2}-\lambda_2^{k/2}}.
\end{align*}

We then calculate $\expect{\exp\pth{tZ^\sfV}}$.
Define the moment generating function (MGF) as $m_k^\sfV = \exp\pth{\kappa_k^\sfV}$ for any $k\ge 1$.
For any vertex cycle $C=\sth{v_1,\cdots,v_k}$ with $v_{i+1} = \sigma(v_i)$ for any $1\le i\le k-1$ and $v_1 = \sigma(v_k)$, let $\ti{A}_{i-1} = \bm{x}_{v_i}$ and $\ti{B}_{i} = \bm{y}_{\pi'(v_i)}$, and we set $\ti{A}_k = \ti{A}_0$ for notational simplicity. Since $\pi^*(v_{i+1}) = \pi'(v_i)$, each pair $(\tilde A_i,\tilde B_i)\sim \mathcal N(\bm 0, \begin{bmatrix}
     I_d&rI_d\\rI_d&I_d
 \end{bmatrix})$. Similarly,\begin{align*}
     m_k^\sfV &= \expect{\expect{\prod_{i=1}^k\exp\pth{t\varphi(r)\ti{A}_{i-1}^\top \ti{B}_i}\big| \ti{B}_1,\cdots \ti{B}_k}}\\&=\expect{\prod_{i=1}^k\exp\pth{tr\varphi(r)\ti{B}_{i-1}^\top \ti{B}_i+\frac{1}{2}t^2 \varphi(r)^2\ti{B}_i^\top \ti{B}_i(1-r^2)}}\\&=\prod_{j=1}^d \expect{\prod_{i=1}^k\exp\pth{tr\varphi(r)\ti{B}_{i-1,j}^\top \ti{B}_{i,j}+\frac{1}{2}t^2 \varphi(r)^2\ti{B}_{i,j}^\top \ti{B}_{i,j}(1-r^2)}},
 \end{align*}
where $\ti{B}_{i,j}$ denotes the $j-$th element of vector $\ti{B}_i$ and the last equality is because $\ti{B}_{i,j}$ and $\ti{B}_{i',j'}$ are independent for any $(i,j)\neq (i',j')$.
Let $\mu_1>\mu_2$ denote two roots of the quadratic equation $x^2-\qth{1-t^2\varphi(r)^2(1-r^2)}x+t^2\varphi(r)^2r^2=0$. Since $t\le \frac{1}{r}-2$, we have $\mu_1+\mu_2>0$ and the discriminant $\qth{1-t^2 \varphi(r)^2(1-r^2)}^2-4t^2\varphi(r)^2 r^2>0$. Since $\mu_1\mu_2>0$, we have $\mu_1>\mu_2>0$.
By a similar argument with calculation for the edge cycle, we have \begin{align*}
    m_k^\sfV = \pth{\frac{1}{\mu_1^{k/2}-\mu_2^{k/2}}}^{d}.
\end{align*}
For any $k\ge 2$, we have $\lambda_1^{k/2}-\lambda_2^{k/2}\ge (\lambda_1-\lambda_2)^{k/2}$, and thus $m_k^\sfE\le (m_2^\sfE)^{k/2}$ for any $k\ge 2$. Similarly, $m_k^\sfV\le (m_2^\sfV)^{k/2}$ for any $k\ge 2$. 
Recall $Z^\sfE$ and $Z^\sfV$ defined in~\eqref{eq:ZEZV}.
We have \begin{align*}
    \log \expect{\exp(tZ)} &= \log \expect{\exp(t Z^\sfE)}+\log\expect{\exp(tZ^\sfV)}\\
    &=\sum_{C\in \maC^\sfE\backslash\binom{F}{2}} \kappa_{|C|}^{\sfE}+\sum_{C\in \maC^\sfV\backslash F} \kappa_{|C|}^\sfV\\
    &\overset{\mathrm{(a)}}{\le} \sum_{i\ge 2}\sum_{C\in \maC_i^\sfE} \frac{|C|}{2} \kappa_2^\sfE(t)+\sum_{C\in \maC_1^\sfE\backslash\binom{F}{2}} \kappa_1^\sfE(t)+\sum_{i\ge 2}\sum_{C\in \maC_i^\sfV} \frac{|C|}{2}\kappa_2^\sfV(t)\\
    &=\sum_{C\in \maC^\sfE\backslash\binom{F}{2}} \frac{|C|}{2}\kappa_2^\sfE(t)+\sum_{C\in \maC_1^\sfE\backslash\binom{F}{2}}\pth{\kappa_1^\sfE(t)-\frac{1}{2}\kappa_2^\sfE(t)}+\sum_{C\in \maC^\sfV\backslash F} \frac{|C|}{2}\kappa_2^\sfV(t)\\
    &\overset{\mathrm{(b)}}{\le} \frac{N_k}{2} \kappa_2^\sfE(t)+\frac{k}{2}\pth{\kappa_1^\sfE-\frac{1}{2}\kappa_2^\sfE(t)+\frac{1}{2}\kappa_2^\sfV(t)},
\end{align*}
 where $\mathrm{(a)}$ is because $\kappa_k^\sfE(t)\le \frac{k}{2}\kappa_2^\sfE(t)$ and $\kappa_k^\sfV(t)\le \frac{k}{2}\kappa_2^\sfV(t)$ for any $k\ge 2$; $(\mathrm{b})$ is because $\sum_{C\in \maC^\sfE\backslash\binom{F}{2}} {|C|} = \binom{n}{2}-\binom{n-k}{2}=N_k$, $\sum_{C\in \maC^\sfV\backslash F} {|C|} = n-(n-k)=k$. It remains to show $|\maC_1^\sfE\backslash \binom{F}{2}|\le \frac{k}{2}$. Indeed, for any $e=uv\in C\in \maC_1^\sfE\backslash \binom{F}{2}$, we have $\pi'(uv) = \pi^*(uv)$. Since $e\notin \binom{F}{2}$, we have $\pi'(u) = \pi^*(v)$ and $\pi'(v) =\pi^*(u)$, which contribute two
mismatched vertices in the reconstruction of the underlying mapping. Since the total number of mismatched vertices for $\pi\in \maT_k$ equals $k$, we have $|\maC_1^\sfE\backslash \binom{F}{2}|\le \frac{k}{2}$. Therefore, we finish the proof of Lemma~\ref{lem:cumulant}.

\subsection{Proof of Lemma~\ref{lem:cumulant-small}}

The cumulant generating function is given by\begin{align*}
    \log \expect{\exp\pth{tY}} = \log \expect{\exp\pth{tY^\sfV}}+\log\expect{\exp\pth{tY^\sfE}}.
\end{align*}

We first calculate $\expect{\exp\pth{tY^\sfE}}$.
Define the moment generating function (MGF) as $\ti{m}_k^\sfE = \exp\pth{\mu_k^\sfE}$ for any $k\ge 1$.
For any edge cycle $C = \sth{e_1,e_2,\cdots,e_k}$ with $e_{i+1} = \sigma^\sfE(e_i)$ for all $1\le i\le k-1$ and $e_1 = \sigma^\sfE(e_k)$, let $A_{i-1} = \beta_{e_i}(G_1)$ and $B_i = \beta_{\pi'(e_i)}(G_2)$ for any $1\le i\le k$, and we set $A_k = A_0$ for notational simplicity. Since $\pi^*(e_{i+1}) = \pi'(e_i)$, each pair $(A_i,B_i)$ follows bivariate normal distribution $\maN\pth{0,\begin{bmatrix}
    1&\rho\\\rho&1
\end{bmatrix})}$, and thus the conditional distribution is given by $A_i\vert B_i\sim \maN(\rho B_i,1-\rho^2)$. Consequently, the MGF is given by \begin{align*}
    \ti{m}_k^\sfE &= \expect{\expect{\prod_{i=1}^k \exp\pth{t\varphi(\rho) (A_{i-1} B_i-A_{i-1}B_{i-1})}\big| B_1,\cdots B_k }}\\&=\expect{\prod_{i=1}^k \exp\pth{t\rho \varphi(\rho)B_{i-1}(B_i-B_{i-1})+\frac{1}{2}t^2 \varphi(\rho)^2( B_{i}-B_{i-1})^2(1-\rho^2)}}\\
    &=\expect{\prod_{i=1}^k\exp\pth{(t\rho \varphi(\rho)-t^2\varphi(\rho)^2(1-\rho^2))(B_{i-1}B_i-B_{i}^2)}},
\end{align*}
where the second equality is because $\expect{\exp\pth{tX}} = \exp\pth{t\mu+\frac{1}{2} t^2 \sigma^2}$ for $X\sim \maN(\mu,\sigma^2)$.

Let $\lambda_1,\lambda_2$ denote the roots of the quadratic function $$x^2 -  \qth{1-2(t^2\varphi(\rho)^2(1-\rho^2)-t\rho \varphi(\rho))}x + (t^2\varphi(\rho)^2(1-\rho^2)-t\rho \varphi(\rho))^2 = 0.$$ Since $0<t<1$, we have \begin{align*}
    f(t,\rho)\triangleq t^2\varphi(\rho)^2(1-\rho^2)-t\rho \varphi(\rho) = (t^2-t)\frac{\rho^2}{1-\rho^2}\in \pth{-\frac{1}{4},0}.
\end{align*}
Therefore, 
we have $\lambda_1+\lambda_2=1-2f(t,\rho)>0$, $\lambda_1\lambda_2 = f(t,\rho)^2>0$ and the discriminant $(1-2f(t,\rho))^2-4f(t,\rho)^2=1-4f(t,\rho) >0$, and thus $\lambda_1>\lambda_2>0$. Define the matrix
\begin{align*}
    \mathbf{J}_{k} \triangleq \begin{bmatrix}
        \lambda_1^{1/2} & -\lambda_{2}^{1/2} &0& \cdots  &0 \\
        0&\lambda_1^{1/2 }&-\lambda_2^{1/2 } &\cdots &0\\
        0&0& \lambda_1^{1/2 }&\cdots &0\\
        \vdots & \vdots &\vdots &\ddots &-\lambda_2^{1/2 }\\
        -\lambda_2^{1/2 } &0&\cdots 0&0& \lambda_1^{1/2 }
    \end{bmatrix}\in \mathbb{R}^{k\times k}.
\end{align*}
Denote $\mathbf{B}_k = \qth{B_1,B_2,\cdots,B_k}^\top$. Then we have \begin{align*}
    \ti{m}_{k}^\sfE & = \expect{\prod_{i=1}^k \exp\pth{(t\rho \varphi(\rho)-t^2\varphi(\rho)^2(1-\rho^2))(B_{i-1}B_i-B_{i}^2)}} \\&= \idotsint  
 \pth{\frac{1}{\sqrt{2\pi}}}^k \exp\pth{-\frac{1}{2} \sum_{i=1}^{k} B_i^2}\\&~~~~\exp\pth{\sum_{i=1}^{k} \qth{{(t\rho \varphi(\rho)-t^2\varphi(\rho)^2(1-\rho^2))(B_{i-1}B_i-B_{i}^2)}}} \,\mathrm{d} B_1\cdots \mathrm{d} B_{k}\\&= \idotsint  
 \pth{\frac{1}{\sqrt{2\pi}}}^k \exp\pth{-\frac{1}{2} \sum_{i=1}^{k} \pth{\lambda_1^{1/2} B_{i-1}- \lambda_2^{1/2} B_{i}}^2} \,\mathrm{d} B_1\cdots \mathrm{d} B_{k}\\ &=\idotsint \pth{\frac{1}{\sqrt{2\pi}}}^{k} \exp\pth{-\frac{1}{2} \mathbf{B}_{k}^\top  {\mathbf{J}}_{k}^\top {\mathbf{J}}_{k} \mathbf{B}_{k}}\, \mathrm{d} B_1\cdots \mathrm{d} B_{k}\\ &= \qth{\det({\mathbf{J}}_{k})}^{-1}=\frac{1}{\lambda_1^{k/2}-\lambda_2^{k/2}}.
\end{align*}

We then calculate $\expect{\exp\pth{tY^\sfV}}$.
Define the moment generating function (MGF) as $\ti{m}_k^\sfV = \exp\pth{\mu_k^\sfV}$ for any $k\ge 1$.
For any vertex cycle $C=\sth{v_1,\cdots,v_k}$ with $v_{i+1} = \sigma(v_i)$ for any $1\le i\le k-1$ and $v_1 = \sigma(v_k)$, let $\ti{A}_{i-1} = \bm{x}_{v_i}$ and $\ti{B}_{i} = \bm{y}_{\pi'(v_i)}$, and we set $\ti{A}_k = \ti{A}_0$ for notational simplicity. Since $\pi^*(v_{i+1}) = \pi'(v_i)$, each pair $(\tilde A_i,\tilde B_i)\sim \mathcal N(\bm 0, \begin{bmatrix}
     I_d&rI_d\\rI_d&I_d
 \end{bmatrix})$. Similarly,\begin{align*}
     m_k^\sfV &= \expect{\expect{\prod_{i=1}^k\exp\pth{t\varphi(r)(\ti{A}_{i-1}^\top \ti{B}_i-\ti{A}_{i-1}^\top \ti{B}_{i-1})}\big| \ti{B}_1,\cdots \ti{B}_k}}\\&=\expect{\prod_{i=1}^k\exp\pth{tr\varphi(r)\ti{B}_{i-1}^\top (\ti{B}_i-\ti{B}_{i-1})+\frac{1}{2}t^2 \varphi(r)^2(\ti{B}_i-\ti{B}_{i-1})^\top (\ti{B}_i-\ti{B}_{i-1})(1-r^2)}}\\&=\prod_{j=1}^d \expect{\prod_{i=1}^k\exp\pth{(tr\varphi(r)-t^2\varphi(r)^2(1-r^2))(B_{i-1,j}B_{i,j}-B_{i,j}^2)}},
 \end{align*}
where $\ti{B}_{i,j}$ denotes the $j-$th element of vector $\ti{B}_i$ and the last equality is because $\ti{B}_{i,j}$ and $\ti{B}_{i',j'}$ are independent for any $(i,j)\neq (i',j')$.
Let $\mu_1>\mu_2$ denote two roots of the quadratic equation $$x^2 -  \qth{1-2(t^2\varphi(r)^2(1-r^2)-tr \varphi(r))}x + (t^2\varphi(r)^2(1-r^2)-tr \varphi(r))^2 = 0.$$

Since $0<t<1$, we have \begin{align*}
    f(t,r)\triangleq t^2\varphi(r)^2(1-r^2)-tr \varphi(r) = (t^2-t)\frac{r^2}{1-r^2}\in \pth{-\frac{1}{4},0}.
\end{align*}
Therefore, 
we have $\lambda_1+\lambda_2=1-2f(t,r)>0$, $\lambda_1\lambda_2 = f(t,r)^2>0$ and the discriminant $(1-2f(t,r))^2-4f(t,r)^2=1-4f(t,r) >0$, and thus $\mu_1>\mu_2>0$.
By a similar argument with calculation for the edge cycle, we have \begin{align*}
    \ti{m}_k^\sfV = \pth{\frac{1}{\mu_1^{k/2}-\mu_2^{k/2}}}^{d}.
\end{align*}
For any $k\ge 2$, we have $\lambda_1^{k/2}-\lambda_2^{k/2}\ge (\lambda_1-\lambda_2)^{k/2}$, and thus $\ti{m}_k^\sfE\le (\ti{m}_2^\sfE)^{k/2}$ for any $k\ge 2$. Similarly, $\ti{m}_k^\sfV\le (\ti{m}_2^\sfV)^{k/2}$ for any $k\ge 2$.

Then, we upper bound $\log\expect{\exp\pth{tY}}$. We have \begin{align*}
    \log \expect{\exp(tY)} &= \log \expect{\exp(t Y^\sfE)}+\log\expect{\exp(tY^\sfV)}\\
    &=\sum_{C\in \maC^\sfE\backslash\maC^\sfE_1} \mu_{|C|}^{\sfE}+\sum_{C\in \maC^\sfV\backslash F} \mu_{|C|}^\sfV\\
    &\le \sum_{C\in \maC^\sfE\backslash\maC^\sfE_1}  \frac{|C|}{2} \mu_2^\sfE(t)+\sum_{C\in \maC^\sfV\backslash F} \frac{|C|}{2}\mu_2^\sfV(t),
\end{align*}
 where the inequality is because $\mu_k^\sfE(t)\le \frac{k}{2}\mu_2^\sfE(t)$ and $\mu_k^\sfV(t)\le \frac{k}{2}\mu_2^\sfV(t)$ for any $k\ge 2$.
 We note that \begin{align*}
     \mu_2^\sfE(t) = -\frac{1}{2}\log\pth{1+\frac{\rho^2}{1-\rho^2}(4t-4t^2)}<0,\text{ for any }0<t<1.
 \end{align*} Consequently,
 \begin{align*}
     \log \expect{\exp(tY)}&\le\sum_{C\in \maC^\sfE\backslash\maC^\sfE_1}  \frac{|C|}{2} \mu_2^\sfE(t)+\sum_{C\in \maC^\sfV\backslash F} \frac{|C|}{2}\mu_2^\sfV(t)\\
    &\le \frac{1}{2}\pth{N_k-\frac{k}{2}}\mu_2^\sfE(t)+\frac{k}{2}\mu_2^\sfV(t),
 \end{align*}
where the inequality is because $\sum_{C\in \maC^\sfE\backslash\binom{F}{2}} {|C|} = \binom{n}{2}-\binom{n-k}{2}=N_k$, $\sum_{C\in \maC^\sfV\backslash F} {|C|} = n-(n-k)=k$, and $|\maC_1^\sfE\backslash\binom{F}{2}|\le \frac{k}{2}$. We finish the proof of Lemma~\ref{lem:cumulant-small}.

\subsection{Proof of Lemma~\ref{lem:mutual-pack}}\label{apd:proof-mutual-pack}

We first lower bound the packing number $|\maM_\delta|$. For any $0<\delta<1$ and $\pi\in \maS_n$, let $B(\pi,r)\triangleq \sth{\pi':\sfd(\pi,\pi')\le r}$ denote the ball of radius $r$ centered at $\pi$. By a standard volume argument~\cite[Theorem 27.3]{polyanskiy2025information}, we have \begin{align*}
    |\maM_\delta| \ge \frac{|\maS_n|}{\max_\pi\vert B(\pi, (1-\delta )n)\vert } = \frac{n!}{\max_\pi\vert B(\pi, (1-\delta )n)\vert}. 
\end{align*}
To upper bound $\vert B(\pi,(1-\delta)n)\vert$, we first  choose $\delta n$ elements from the domain of $\pi$ and map to the same value as $\pi$, and the remaining
domain and range of size $n-\delta n$ and the mapping are selected arbitrarily. We get $B(\pi,(1-\delta)n)\le \binom{n}{\delta n} (n-\delta n)!$. Consequently,
\begin{align*}
    |\maM_\delta|\ge \frac{n!}{\max_\pi\vert B(\pi, (1-\delta )n)\vert}\ge \frac{n!}{\binom{n}{\delta n} (n-\delta n)!}=(\delta n)! \ge \pth{\frac{\delta n}{e}}^{\delta n}.
\end{align*}

We then upper bound the mutual information. 
Recall that the likelihood function is given by 
$$\mathcal P_{G_1,G_2\mid\pi^*}=\prod_{e\in E(G_1)}P(\beta_e(G_1),\beta_{\pi^*(e)}(G_2))\prod_{v\in V(G_1)}f(\bm x_v,\bm y_{\pi^*(v)}).$$
Next, we introduce an auxiliary distribution $\mathcal{Q}$ under which $G_1$ and $G_2$ are independent, while maintaining the same marginals as under $\mathcal{P}$. Denote $Q(\cdot,\cdot)$ as the distribution of two independent standard normals and $g(\bm{x}, \bm{y})$ as the multivariate normal distribution $\mathcal N\left(\bm 0,\begin{bmatrix}
    I_d&O\\ O&I_d
\end{bmatrix}\right)$. Then
$$\mathcal Q_{G_1,G_2}=\prod_{e\in E(G_1)}Q(\beta_e(G_1),\beta_{\pi^*(e)}(G_2))\prod_{v\in V(G_1)}g(\bm x_v, \bm y_{\pi (v)}).$$
The KL-divergence between the product measures $\maP_{G_1,G_2|\pi^*}$ and $\maQ_{G_1,G_2}$ is given by
$$D_{\mathrm{KL}}(\mathcal P_{G_1,G_2\mid \pi^*}\| \mathcal Q_{G_1,G_2})=\binom n2 D_{\mathrm{KL}}(P\|Q)+nD_{\mathrm{KL}}(f\|g).$$
We note that \begin{align*}
    D_{\mathrm{KL}}(P\Vert Q) &= \iint P(a,b)\log\pth{\frac{P(a,b)}{Q(a,b)}}\,\mathrm{d}a\mathrm{d}b\\
    &=\iint P(a,b)\qth{\frac{1}{2}\log\pth{\frac{1}{1-\rho^2}}+\frac{\rho ab}{1-\rho^2}-\frac{\rho^2(a^2+b^2)}{2(1-\rho^2)}}\,\mathrm{d}a\mathrm{d}b\\&=\frac{1}{2}\log\pth{\frac{1}{1-\rho^2}}+\frac{\rho^2}{1-\rho^2}-\frac{2\rho^2}{2(1-\rho^2)}=\frac{1}{2}\log\pth{\frac{1}{1-\rho^2}}.
\end{align*}
Similarly, $D_{\mathrm{KL}}(f\Vert g) = \frac{d}{2}\log\pth{\frac{1}{1-r^2}}$.
Consequently,\begin{align*}
    I(\pi^*; G_1, G_2)&=\E_{\pi^*}\qth{D_{\mathrm{KL}}(\mathcal P_{G_1,G_2\mid \pi^*}\|\mathcal P_{G_1,G_2})}\\
    &\le \max_{\pi\in\maS_n} D_{\mathrm{KL}}(\mathcal P_{G_1,G_2\mid \pi}\|\mathcal Q_{G_1,G_2})=\binom n2 \frac12\log(\frac1{1-\rho^2})+\frac{nd}{2}\log(\frac1{1-r^2}).
\end{align*}
\subsection{Proof of Lemma~\ref{lem:d=3}}

In this subsection, without loss of generality, we assume $V(G_1) = V(G_2) = [n]$. Define adjacent matrices $A,B\in \mathbb{R}^{n\times n}$  with $A_{ij} = \beta_{ij}(G_1)$ and $B_{ij}=\beta_{ij}(G_2)$ for any $1\le i<j\le n$. Let $X,Y\in \mathbb{R}^{n\times 1}$ with $X_i = \bm{x}_i$ and $Y_i = \bm{y}_i$ with $1\le i\le n$. For any $\pi\in \maS_n$, define $A^\pi\in \mathbb{R}^{n\times n}$ with $A^\pi_{ij} = A_{\pi(i)\pi(j)}$ for any $1\le i<j\le n$, $X^\pi\in\mathbb{R}^{n\times 1}$ with $X^\pi_i = X_{\pi(i)}$ for any $1\le i\le n$. For two matrices $A$ and $B$, define the inner product as $\langle A,B\rangle = \sum_{1\le i<j\le n}A_{ij}B_{ij}$. Then, \begin{align*}
    \score_\pi(G_1,G_2) &= \varphi(\rho)\sum_{e\in E(G_1)} \beta_e(G_1) \beta_{\pi(e)}(G_2)+\varphi(r)\sum_{v\in V(G_1)} \bm{x}_v \bm{y}_{\pi(v)}\\&=\varphi(\rho) \langle A,B^\pi\rangle+\varphi(r) \langle X, Y^\pi\rangle.
\end{align*}

For any $\pi_1\neq\pi_2\in \maT_2$ with $\sfd(\pi_1,\pi_2) = 3$, \begin{align*}
    &~\prob{(G_1,G_2)\in \maE(\pi^*,\pi_1)\cap \maE(\pi^*,\pi_2)}\\=&~\expect{\indc{\score_{\pi^*}(G_1,G_2)\le \score_{\pi_1}(G_1,G_2)}\indc{\score_{\pi^*}(G_1,G_2)\le \score_{\pi_2}(G_1,G_2)}}\\\le&~\expect{\exp\pth{\frac{1}{2}\pth{\score_{\pi_1}(G_1,G_2)- \score_{\pi^*}(G_1,G_2)}}\exp\pth{\frac{1}{2}\pth{\score_{\pi_2}(G_1,G_2)- \score_{\pi^*}(G_1,G_2)}}}\\=&~\expect{\exp\pth{\varphi(\rho)\pth{\frac{1}{2}\langle A,B^{\pi_1}\rangle+\frac{1}{2}\langle A,B^{\pi_2}\rangle-\langle A,B^{\pi^*}\rangle}}}\\&~~~~\cdot\expect{\exp\pth{\varphi(r)\pth{\frac{1}{2}\langle X,Y^{\pi_1}\rangle+\frac{1}{2}\langle X,Y^{\pi_2}\rangle-\langle X,Y^{\pi^*}\rangle}}}.
\end{align*}

For simplicity, we denote $\mathrm{d}A\mathrm{d}B = \mathrm{d}A_{12}\mathrm{d}A_{13}\cdots \mathrm{d}A_{n-1\, n}\mathrm{d}B_{12}\mathrm{d}B_{13}\cdots \mathrm{d}B_{n-1\, n}$.
For the first term, we note that \begin{align*}
    &~\expect{\exp\pth{\varphi(\rho)\pth{\frac{1}{2}\langle A,B^{\pi_1}\rangle+\frac{1}{2}\langle A,B^{\pi_2}\rangle-\langle A,B^{\pi^*}\rangle}}}\\=&~\pth{\frac{1}{2\pi\sqrt{1-\rho^2}}}^{\binom{n}{2}}\\&~~~~\cdot\idotsint\exp\pth{\frac{\varphi(\rho)}{2}\pth{\langle A,B^{\pi_1}\rangle+\langle A,B^{\pi_2}\rangle}-\frac{1}{2(1-\rho^2)}\pth{\Vert A\Vert_F^2+\Vert B\Vert_F^2}} \mathrm{d}A\mathrm{d}B.
\end{align*}
Let $\mathrm{vec}(A) = (A_{12},A_{13},\cdots,A_{21},\cdots A_{n-1\ n})$ for any adjacent matrix $A$. For any $\pi_1\neq \pi_2\in \maT_2$, define permutation matrices $\Pi^\sfE_1$ and $\Pi^\sfE_2\in \sth{0,1}^{\binom{n}{2}\times \binom{n}{2}}$ as \begin{align*}
    \mathrm{vec}(B^{\pi_1}) = \Pi_1^\sfE \mathrm{vec}(B),\quad \mathrm{vec}(B^{\pi_2}) = \Pi_2^\sfE \mathrm{vec}(B).
\end{align*}
Then, \begin{align*}
    &~\frac{\varphi(\rho)}{2}\pth{\langle A,B^{\pi_1}\rangle+\langle A,B^{\pi_2}\rangle}-\frac{1}{2(1-\rho^2)}\pth{\Vert A\Vert_F^2+\Vert B\Vert_F^2}\\=&~\frac{\rho}{2(1-\rho^2)}\mathrm{vec}(A)^\top (\Pi_1^\sfE+\Pi_2^\sfE)\mathrm{vec}(B)-\frac{1}{2(1-\rho^2)}(\Vert \mathrm{vec}(A)\Vert_2^2+\Vert \mathrm{vec}(B)\Vert_2^2)\\=&~-\frac{1}{2(1-\rho^2)}\begin{bmatrix}
      \mathrm{vec}(A)\\\mathrm{vec}(B)  
    \end{bmatrix}^\top \Sigma\begin{bmatrix}
      \mathrm{vec}(A)\\\mathrm{vec}(B)  
    \end{bmatrix},
\end{align*}
where $\Sigma\triangleq \begin{bmatrix}
        I_{\binom{n}{2}\times \binom{n}{2}}&-\frac\rho2 (\Pi^\sfE_1+\Pi_2^\sfE)\\-\frac{\rho}{2}
        (\Pi^\sfE_1+\Pi_2^\sfE)^\top&I_{\binom{n}{2}\times \binom{n}{2}}
    \end{bmatrix}$. Therefore, \begin{align*}
        &~\expect{\exp\pth{\varphi(\rho)\pth{\frac{1}{2}\langle A,B^{\pi_1}\rangle+\frac{1}{2}\langle A,B^{\pi_2}\rangle-\langle A,B^{\pi^*}\rangle}}}\\=&~\pth{\frac{1}{2\pi\sqrt{1-\rho^2}}}^{\binom{n}{2}}\idotsint\exp\pth{-\frac{1}{2(1-\rho^2)}\begin{bmatrix}
      \mathrm{vec}(A)\\\mathrm{vec}(B)  
    \end{bmatrix}^\top \Sigma\begin{bmatrix}
      \mathrm{vec}(A)\\\mathrm{vec}(B)  
    \end{bmatrix}}\mathrm{d}A\mathrm{d}B\\=&~\sqrt{\frac{(1-\rho^2)^{\binom{n}{2}}}{\det(\Sigma)}}.
    \end{align*}

Let $\sigma \triangleq \pi_2\circ \pi_1^{-1}$
Recall that $\maC_i^\sfE$ and $\maC_i^\sfV$ denote the set of edge orbits and  node orbits with length $i$ induced by $\sigma$. 
It follows from~\cite[Lemmas 4.2 and 4.3]{dai2019database} that \begin{align*}
    \det(\Sigma) &=\prod_{k=1}^n \pth{\prod_{j=1}^k\pth{1-\frac{\rho^2}{2}\pth{1+\cos\pth{\frac{2\pi j}{k}}}}}^{|\maC^\sfE_k|}\\&\ge (1-\rho^2)^{|\maC_1^\sfE|}\prod_{k=2}^{n}(1-\rho^2)^{k|\maC_k^\sfE|/2} = (1-\rho^2)^{\frac{1}{2}\pth{\binom{n}{2}+|\maC_1^\sfE|}},
\end{align*}
where the inequality follows from Lemma~\ref{lem:det_bound} and the last equality is because $\sum_{k=1}^n k|\maC_k^\sfE| = \binom{n}{2}$.
Therefore,\begin{align*}
    \expect{\exp\pth{\varphi(\rho)\pth{\frac{1}{2}\langle A,B^{\pi_1}\rangle+\frac{1}{2}\langle A,B^{\pi_2}\rangle-\langle A,B^{\pi^*}\rangle}}}\le (1-\rho^2)^{\frac{1}{4}\pth{\binom{n}{2}-|\maC_1^\sfE|}}.
\end{align*}
Similarly,\begin{align*}
    \expect{\exp\pth{\varphi(r)\pth{\frac{1}{2}\langle X,Y^{\pi_1}\rangle+\frac{1}{2}\langle X,Y^{\pi_2}\rangle-\langle X,Y^{\pi^*}\rangle}}}\le (1-r^2)^{\frac{\sfd(n-|\maC^\sfV_1|)}{4}}.
\end{align*}
For any $\pi_1\neq \pi_2\in\maT_2$ with $\sfd(\pi_1,\pi_2) = 3$.
Let $F\triangleq\sth{i\in[n]:\pi_1(i) = \pi_2(i)}$.
Then there exists $1\le i,j,k\le n$ with $i\neq j\neq k$ such that, \begin{align*}
 \pi_1(i) = \pi_2(j),\pi_1(j) = \pi_2(k),\pi_1(k)=\pi_2(i),\text{ and }\pi_1(\ell) = \pi_2(\ell) \text{ for any }\ell\in [n]\backslash\sth{i,j,k}.  
\end{align*}
Then,\begin{align*}
    |\maC_1^\sfV| &= |\sth{\ell\in [n]:\ell\notin \sth{i,j,k}}|=n-3,\\|\maC_1^\sfE|&=|(\ell_1,\ell_2):1\le \ell_1<\ell_2\le n,\ell_1,\ell_2\notin \sth{i,j,k}| =\binom{n-3}{2}.
\end{align*}


Consequently,\begin{align*}
    &~\prob{(G_1,G_2)\in\maE(\pi^*,\pi_1)\cap\maE(\pi^*,\pi_2)}\\\le&~\expect{\exp\pth{\varphi(\rho)\pth{\frac{1}{2}\langle A,B^{\pi_1}\rangle+\frac{1}{2}\langle A,B^{\pi_2}\rangle-\langle A,B^{\pi^*}\rangle}}}\\&~~~~\cdot\expect{\exp\pth{\varphi(r)\pth{\frac{1}{2}\langle X,Y^{\pi_1}\rangle+\frac{1}{2}\langle X,Y^{\pi_2}\rangle-\langle X,Y^{\pi^*}\rangle}}}\\\le&~(1-\rho^2)^{\frac{\binom{n}{2}-\binom{n-3}{2}}{4}}(1-r^2)^{\frac{3d}{4}}=(1-\rho^2)^{\frac{3(n-2)}{4}}(1-r^2)^{\frac{3d}{4}}.
\end{align*}

\section{Auxiliary Results}

\begin{lemma}\label{lem:kappa2bound}
    For $0<\rho^2<\frac{1}{10}$, we have
    $$-\frac{1-\rho^2}{2\rho^2}\log\left(1-\frac{\rho^2(2+\rho^2)}{(1-\rho^2)^2}\right)\le 1+4\rho^2.$$
\end{lemma}

\begin{proof}
Let $x \triangleq \rho^2 \in (0,\frac{1}{10})$ and define
$$
t \triangleq \frac{x(2+x)}{(1-x)^2}, \quad (0 < t < 1).
$$
Since $x < \frac{1}{4}$, we have
$$1 - t = \frac{(1-x)^2 - x(2+x)}{(1-x)^2} = \frac{1 - 4x}{(1-x)^2} > 0.$$
For any $0 < t < 1$,
$$
-\log(1-t) = \sum_{k=1}^\infty \frac{t^k}{k}
= t + \sum_{k=2}^\infty \frac{t^k}{k}
\le t + \frac12 \sum_{k=2}^\infty t^k
= t + \frac{t^2}{2(1-t)}.
$$
Consequently, we have
\begin{align*}
    -\frac{1-x}{2x} \log\!\left( 1 - \frac{x(2+x)}{(1-x)^2} \right)-(1+4x)
&\le \frac{1-x}{2x} \left( t + \frac{t^2}{2(1-t)} \right)-(1+4x)\\&= \frac{3x\,(-21x^2 + 20x - 2)}{4(4x^2 - 5x + 1)}.
\end{align*}
For $0 < x < \frac{1}{4}$, we have $4x^2-5x+1>0$.
The numerator factor
$f(x)= -21x^2 + 20x - 2$
is strictly increasing on $[0, \frac{1}{10}]$ (since $f'(x) = -42x + 20 > 0$), and
$f\pth{\frac{1}{10}} = -0.21 < 0.$
Thus $f(x) < 0$ for all $0 < x \le \frac{1}{10}$, making the whole fraction nonpositive. Therefore, for any $0<\rho^2<\frac{1}{10}$,
\[
-\frac{1-\rho^2}{2\rho^2} \log\!\left( 1 - \frac{\rho^2(2+\rho^2)}{(1-\rho^2)^2} \right)
\le 1 + 4\rho^2.\qedhere
\]
\end{proof}

\begin{lemma}[Chernoff's inequality for Chi-squared distribution]\label{lem:chisquare}
    Suppose $\xi$ follows the chi-squared distribution with $n$ degrees of freedom. Then, for any  $\delta>0$, \begin{align}\label{eq:concentration_for_chisquare}
        \prob{\xi>(1+\delta) n}&\le  \exp\pth{-\frac{n}{2}\pth{\delta-\log\pth{1+\delta}}},\\\prob{\xi<(1-\delta)n}&\le \exp\pth{-\frac{\delta^2}{4}n}.
    \end{align}
\end{lemma}
\begin{proof}
    The result follows from \citet[Theorems 1 and 2]{ghosh2021exponential}.
\end{proof}

\begin{lemma}\label{lem:det_bound}
For any integer $k \ge 1$ and any $0 \le \rho \le 1$, we have
\[
\prod_{j=1}^{k}\left(\left(1-\frac{\rho^2}{2}\right)-\frac{\rho^2}{2}\cos\frac{2\pi j}{k}\right)
\;\;\geq\;\;(1-\rho^2)^{k/2}.
\]
\end{lemma}

\begin{proof}
We note that
\[
\left(1-\frac{\rho^2}{2}\right)-\frac{\rho^2}{2}\cos\frac{2\pi j}{k}
=1-\rho^2 \cos^2\left(\frac{\pi j}{k}\right),
\]
so the product equals
\[
P=\prod_{j=1}^k \Bigl(1-\rho^2\cos^2(\frac{\pi j}{k})\Bigr).
\]
For any $a,t\in[0,1]$, we have $1-at \ge (1-a)^t$, which follows from concavity of $g(t)=\ln(1-at)-t\ln(1-a)$ and $g(0)=g(1)=0$. Applying this with $a=\rho^2$ and $t=\cos^2(\pi j/k)$ yields
\[
1-\rho^2\cos^2(\frac{\pi j}{k}) \ge (1-\rho^2)^{\cos^2(\pi j/k)}.
\]
Multiplying over $j=1,\ldots,k$ gives 
\[
P\ge (1-\rho^2)^{\sum_{j=1}^k \cos^2(\pi j/k)}.
\]
Since $\sum_{j=1}^k \cos^2(\pi j/k)=\sum_{j=1}^k\frac{1+\cos(2\pi j/k)}{2}=\frac{k}{2}$, we conclude that
\[
P=\prod_{j=1}^k\pth{\left(1-\frac{\rho^2}{2}\right)-\frac{\rho^2}{2}\cos\frac{2\pi j}{k} }\ge(1-\rho^2)^{k/2}.
\]
\end{proof}

\bibliographystyle{alpha}
\bibliography{main}

\end{document}